\documentclass[a4paper,11pt]{amsart}

\usepackage[T1]{fontenc}              
\usepackage{rotating}
\usepackage{lmodern} \usepackage{amsmath,amssymb,amscd,bm,mathtools,xypic,extpfeil,graphicx,color,transparent}   
\usepackage{amsthm}
\newtheorem{thm}{Theorem}
\theoremstyle{definition}
\newtheorem{defn}[thm]{Definition}
\theoremstyle{remark}
\newtheorem{example}[thm]{Example}
\theoremstyle{remark}
\newtheorem{remark}[thm]{Remark}
\theoremstyle{lemma}
\newtheorem{lemma}[thm]{Lemma}
\theoremstyle{corollary}
\newtheorem{corollary}[thm]{Corollary}
\newtheorem{proposition}[thm]{Proposition}

\newcommand {\NN}{{\mathbb N}}
\newcommand {\QQ}{{\mathbb Q}}
\newcommand {\CC}{{\mathbb C}}

\newcommand {\RR}{{\mathbb R}}
\newcommand {\ZZ}{{\mathbb Z}}
\newcommand {\TT}{{\mathbb T}}

\newcommand {\PP}{{\mathbb P}}

\newcommand {\UU}{{\mathbb U}}

\newcommand{\zz}{{\mathbf{z}}} 
\newcommand{\kk}{{\mathbf{k}}}
\newcommand{\uu}{{\mathbf{u}}}
\renewcommand{\ll}{{\mathbf{l}}}
\newcommand{\lp}{{\mathcal{L}}}
\newcommand{\coker}{\mathrm{coker}}

\newcommand{\BraidSet}{{\mathsf B}}
\newcommand{\DivBraid}{{\mathsf{DB}}}
\newcommand{\CF}{{\mathsf{Conf}}}
\newcommand{\class}[1]{{\mathbf{#1}}}

\newcommand{\basic}{{\mathcal B}}

\usepackage{amsmath}
\usepackage[all,cmtip]{xy}

\begin{document}
\title{Divisor braids}
\author{Marcel B\"okstedt}
\address{Institut for Matematik, Aarhus Universitet, Ny Munkegade 118, bygn.~1530, 8000 \AA rhus C, Denmark}
\email{marcel@imf.au.dk}
\author{Nuno M.~Rom\~ao}
\address{Mathematisches Institut, Georg-August-Universit\"at G\"ottingen, Bunsenstra\ss e 3--5, 37073 G\"ottingen, Germany;  \newline
${}\qquad\quad\!$Institut des Hautes \'Etudes Scientifiques, 35 route de Chartres, 91440 Bures-sur-Yvette, France}
\email{nromao@ihes.fr}
\date{\today}
\begin{abstract} We study a novel type of braid groups on a closed orientable surface $\Sigma$. 
These are fundamental groups of certain manifolds that are hybrids between symmetric products and configuration spaces of
points on $\Sigma$; a class of examples arises naturally in  gauge theory, as moduli spaces of vortices in toric fibre bundles over $\Sigma$.
The elements of these braid groups, which we call divisor braids, have coloured strands that are allowed to intersect according to rules specified by a graph $\Gamma$.
In situations where there is more than one strand of each colour, we show that the corresponding braid group admits a metabelian presentation as a central
extension of the free Abelian group $H_1(\Sigma;\ZZ)^{\oplus r}$, where $r$ is the number of colours, and describe its Abelian commutator.
This computation relies crucially on producing a link invariant (of closed divisor braids) in the three-manifold $S^1 \times \Sigma $ for each graph $\Gamma$.
We also describe the von Neumann algebras associated to these groups in terms of rings that are familiar from noncommutative geometry.
\end{abstract}

\maketitle
\tableofcontents

MSC2010: {\tt 20F36; 57M27, 81T60.}

\newpage

\section{Introduction} 

In this paper, we study groups of generalised braids on a surface $\Sigma$. We shall assume that this surface is connected, oriented and closed, 
and we assign colours to all the strands of our braids.
The main novelty is that we want to extend the set of homotopies so as to allow strands  to intersect (i.e. pass through each other)
according to certain rules, unlike ordinary braids. These rules will depend on the colours of the strands, which we take from a finite set.
To implement the obvious composition law, we must require that all our braids are {\em colour-pure} in the sense that the point on $\Sigma$ where a given
strand starts will be the endpoint for a strand of the same colour (but this could possibly be a different strand); such coloured braids, pictured as usual in $[0,1] \times \Sigma$, also determine closed  braids inside $S^1 \times \Sigma$. 

Let $r$ be the number of colours used.
The rules of intersection of strands will be determined by a graph $\Gamma$ 
(without orientation) whose set of vertices ${\rm Sk}^0(\Gamma)$ is in
bijection with the set of colours.  As additional data, we will need a function 
$$\mathbf k: {\rm Sk}^0(\Gamma) \rightarrow \NN$$
on the set of vertices. The value $k_\lambda := \mathbf{k} (\lambda)$ will be thought of as a decoration or label at each vertex $\lambda$ of $\Gamma$; alternatively, we can
introduce an order on the set of vertices and record the labels $k_\lambda$ as a vector in $\NN^r$. The pair $(\Gamma, \mathbf{k})$ will be referred to as a {\em negative colour scheme}. Together with the genus of the surface $\Sigma$, it will completely specify a particular group of braids with
\begin{equation}\label{normk}
|\mathbf{k}| := \sum_{\lambda \in {\rm Sk}^0(\Gamma)} k_\lambda
\end{equation}
strands which we shall denote as $\DivBraid_{\mathbf{k}}(\Sigma, \Gamma  )$.

The intersection of strands is determined by the following rule: strands of two different colours are forbidden to intersect whenever the corresponding vertices are connected by an edge
in  $\Gamma$. In particular, we may discard graphs with multiple edges without loss of generality.
If $\Gamma$ consisted of a single vertex $\lambda$ and a single edge starting and ending at this vertex
(which we will call a {\em self-loop}), then $\DivBraid_{\mathbf{k}}(\Sigma, \Gamma )$ would simply be the braid group ${\sf B}_{k}(\Sigma)$ on $k:=k_\lambda$ strands
on the surface $\Sigma$, which
is a familiar object~\cite{Bir}. However, in this paper we want to restrict our attention to negative colour schemes whose graphs do not contain self-loops. As we shall see, the corresponding generalised braid groups will still be interesting objects, arising quite naturally in mathematics.
As a simple example, if $\Gamma$ is a complete graph (i.e. all pairs of vertices are connected by an edge) with $r$ vertices, all labelled by the integer $1$, we
obtain the pure braid group~\cite{KasTur} on $r$ strands ${\sf PB}_r(\Sigma) \subset {\sf B}_r(\Sigma)$ on the surface; a particular case is of course  the fundamental group $\pi_1(\Sigma) = {\sf PB}_1(\Sigma)$.

Given an element $\gamma$ of  $\DivBraid_{\mathbf{k}}(\Sigma, \Gamma)$ where $\Gamma$ has no self-loops, one can interpret the sub-braid $\gamma_{\lambda}$ consisting of all strands of a fixed colour $\lambda\in {\rm Sk}^0(\Gamma)$ as describing a homotopy class of loops on the space of effective divisors of degree $k_\lambda$ on $\Sigma$ based at a reduced divisor. In this spirit, we shall from now on refer to the elements of our braid groups $\DivBraid_{\mathbf{k}}(\Sigma, \Gamma)$ as {\it divisor braids}.

Our study of divisor braid groups is directly motivated by a problem in gauge theory. It is well known that when $\Sigma$ is a Riemann surface the symmetric products 
$S^k \Sigma :=\Sigma^k/\mathfrak{S}_k$ are smooth manifolds with an induced complex structure, and they are of interest in algebraic geometry
as spaces of effective divisors of degree $k$ on $\Sigma$. If in addition we endow $\Sigma$ with a symplectic structure  compatible with the
complex structure (a K\"ahler area form), there is a way of inducing K\"ahler structures (and in particular Riemannian metrics) on $S^k \Sigma$  as well --- in fact, a real one-dimensional family of them. This comes about because these manifolds are moduli spaces of solutions (modulo gauge
equivalence) of the {\em vortex equations} in line bundles of degree $k$ over $\Sigma$. The fundamental groups of these spaces $\mathcal{M}^\CC_{k}(\Sigma)\cong S^k \Sigma $ are either $\pi_1(\Sigma)$ for $k=1$ or its
Abelianisation $H_1(\Sigma;\ZZ)$ for $k>1$, and both of these groups are particular cases of our divisor braid groups (for the latter, we take a graph consisting of a single vertex with label $k$). 

There are many ways to generalise the vortex equations beyond line bundles.
One possibility is to consider vortices with more general K\"ahler toric target manifolds $X$, using the real torus $\TT$ of $X$ as the structure group of the gauge theory, and specify a moment map $\mu:X \rightarrow {\rm Lie}(\TT)^*$. When $X$ is compact, one is dealing with a class of Abelian nonlinear vortices whose moduli spaces $\mathcal{M}^X_{\mathbf{h}}(\Sigma)$ (under suitable stability asumptions) have been identified~\cite{BokRomTFB} with
 certain open submanifolds of a Cartesian product of symmetric products
 \begin{equation*}\label{Symk}
 S^{\mathbf{k}} \Sigma:=\prod_{\lambda \in {\sf Fan}_{\mu(X)}(1)}S^{k_\lambda} \Sigma.
 \end{equation*}
In this product, the indices $\lambda$ are taken from the set of rays $\mathsf{Fan}_{\mu(X)}(1)$  in the normal fan determining the toric manifold $X$; see e.g.~\cite{CLS} for background on toric geometry.
The label $\mathbf h$ can be interpreted as an element of the  $\TT$-equivariant homology group   $H^\TT_2(X;\ZZ)$, and it corresponds to the image of the generator
$[\Sigma]\in H_2(\Sigma;\ZZ) \cong \ZZ$  under the {\em BPS charge} that is relevant to this setup ---  see~\cite{BokRomTFB}
for the general definition in the framework of vortices in toric fibre bundles over closed K\"ahler manifolds of arbitrary dimension. The integers $k_\lambda$ are obtained from the formula
\begin{equation}\label{charges}
k_\lambda= \langle c_1^\TT (D_\lambda) , \mathbf{h} \rangle
\end{equation}
using the pairing of $\TT$-equivariant cohomology and homology of $X$ in degree 2; in this formula, $c_1^\TT(D_\lambda)$ are the $\TT$-equivariant first Chern classes of irreducible $\TT$-equivariant divisors in $X$ corresponding to the rays in the normal fan~\cite{BokRomTFB}. 
 The spaces  $\mathcal{M}^X_{\mathbf{h}}(\Sigma) \subset S^\mathbf{k} \Sigma $ also carry a natural K\"ahler structure (see Section~\ref{sec:vortices} below) which plays an important role in the description of  gauged nonlinear sigma-models associated to the vortex equations, both at classical and quantum level --- see~\cite{RomSpe} for a study of this so-called
 $L^2$-metric  in simple examples where $X=\PP^1$.

 It turns out that the moduli spaces $\mathcal{M}^X_{\mathbf{h}}(\Sigma)$ in this situation have fundamental groups that provide examples of the divisor braid groups described in this paper.  This fact is established in Proposition~\ref{onlySk1}, which also clarifies how to construct the relevant graph $\Gamma$ in the negative colour scheme from the combinatorial data of the toric target $X$; see also equation (\ref{fundgroups}).
 The set of vertices or colours ${\rm Sk}^0(\Gamma)= \mathsf{Fan}_{\mu(X)}(1)$ corresponds to the set of rays in the normal fan of $X$, 
 whereas the labels $k_\lambda$ are determined as in (\ref{charges}).
We shall provide the reader with more background about the whole setup in Section~\ref{sec:vortices}, and also indicate why understanding  the structure of this particular class of divisor braid groups and their representation varieties, as well as certain Hilbert modules over the associated von Neumann algebras, is significant in the context of supersymmetric quantum field theory in two dimensions.

The rest of the paper is organised as follows.  In Section~\ref{sec:fundamental}, we make our definition of divisor braids precise, and identify the
 groups  $\DivBraid_{ \mathbf{k}}(\Sigma, \Gamma)$  introduced more informally above as fundamental groups of a canonical type of $2|\mathbf{k}|$-dimensional configuration spaces  $\CF_{\mathbf{k}}(\Sigma, \Gamma)$; these spaces are hybrids between symmetric products and  ordinary configuration spaces modelled on the surface $\Sigma$. In Section~\ref{sec:pictures}, we show that each such $\DivBraid_{ \mathbf{k}}(\Sigma, \Gamma)$ is a central extension of the 
group $H_1(\Sigma;\ZZ)^{\oplus r}$ by a certain Abelian group $E$. We give generators for $E$ and some relations; this provides a surjective map $D\to E$, where $D$ is an Abelian group
presented in terms of generators and relations.  In Section~\ref{sec:linking} we prove that, for the case of two colours $r=2$, the map $D\to E$ is an isomorphism. The proof we shall give relies on the construction of a link invariant on $E$.  In Section~\ref{sec:colors}, we extend this link invariant to the case of an arbitrary number of colours. Again, this can be used to show that the map $D\to E$ is an isomorphism. In Section~\ref{sec:center} we explain how the group $D$ depends on the decorated graph $(\Gamma,\mathbf{k})$. As might be expected, the most intricate dependence occurs at the level of the finite Abelian group ${\rm Tor}\, D$; an approach to the study this phenomenon for negative colour schemes $(\Gamma,\mathbf{k})$ based on a fixed graph $\Gamma$ is presented in~\cite{Bok}. Section~\ref{sec.examples} collects a handful of examples chosen to illustrate some features of the general theory. Finally, in Section~\ref{sec:ncg}, we describe the von Neumann algebras associated to very composite divisor braid groups in terms of noncommutative tori; this exercise is motivated by the applications to mathematical physics that we describe in Section~2.\\

\noindent
{\bf Acknowledgements:}
The authors would like to thank Carl-Friedrich B\"odigheimer (Bonn) for a discussion about the theory of configuration spaces and for giving them access to his PhD thesis; Chris Brookes (Cambridge) for sharing some insights related to Section~\ref{sec:ncg}, as well as Vadim Alekseev (G\"ottingen) for general advice on that section; and J\o rgen Tornehave (Aarhus) for very helpful discussions about aspects of topology connected to this work.


\section{Divisor braids from gauge theory} \label{sec:vortices}

This section is intended to give a brief account of
our original motivation to study the groups 
$\DivBraid_{\mathbf{k}}(\Sigma, \Gamma)$, which was only mentioned in passing in the Introduction. A reader who is not interested in this material can skip  it without loss of continuity, referring back to it later as needed.

\subsection{The vortex equations} \label{sec:vorteq}

Let $\Sigma$ be a closed oriented  surface  equipped with a Riemannian metric $g_\Sigma$. The metric
determines a K\"ahler structure  $(\Sigma, j_\Sigma, \omega_\Sigma)$ on $\Sigma$, where the complex structure $j_\Sigma$ corresponds to a rotation by a right angle in
the direction prescribed by the orientation, and the symplectic structure $\omega_\Sigma$ is the associated area form.
We shall consider another K\"ahler manifold $(X, j_X, \omega_X)$ where a holomorphic Hamiltonian action of a Lie group $G$ is given, and fix a moment 
map $\mu: X\rightarrow {\rm Lie}(G)^*=: \mathfrak{g}^*$ for this action. We also fix a $G$-equivariant isomorphism of vector spaces $\kappa: \mathfrak{g}^*{\rightarrow} \mathfrak{g}$
and write $\mu^\kappa:= \kappa \circ \mu$; $\kappa_\mathfrak{g}$ and $\kappa_{\mathfrak{g}^*}$ denote the induced $G$-invariant inner products on $\mathfrak g$ and $\mathfrak{g}^*$.
The surface $\Sigma$ will play the role of source, whereas the manifold $X$ will be the target for the gauge field theories we are interested in.

Let $\pi: P\rightarrow \Sigma$ be a $G$-principal bundle over $\Sigma$. The {\em vortex equations} are the PDEs\footnote{With a slight abuse of language, one often omits the
pull-back $\pi^*$ in the second equation; this amounts to identifying the form $F_A$ with its ($G$-equivariant) pull-back to $\Sigma$, which is only unambiguous when $G$ is Abelian.}
\begin{equation} \label{vortex}
\bar \partial^A_{j_X,j_\Sigma}\, \phi =0,\qquad F_A +  (\mu^\kappa \circ \phi)\, \pi^* \omega_\Sigma =0 
\end{equation}
for a smooth $G$-equivariant map $\phi: P\rightarrow X$  and a connection $A$ in $P$ with curvature $F_A \in \Omega^2(P;\mathfrak{g})$. The connection $A$ can be seen as a $G$-equivariant splitting of vector bundles ${\rm T}P\cong \ker {\rm d}\pi \oplus \pi^*{\rm T}\Sigma$ over $P$, or as 1-form $A:{\rm T}P \rightarrow \mathfrak{g}$ corresponding to
the projection onto the first summand of that splitting, whereas the equivariant map $\phi$ can also be interpreted as a section $\phi:\Sigma \rightarrow E_X$ of the fibre bundle
\begin{equation}\label{bundleE}
\pi_X: E_X:=P\times_G X\rightarrow \Sigma
\end{equation}
with fibre $X$ associated to $P$ via the $G$-action on $X$. So there is a projection 
$p^A: {\rm T}E_X \cong \ker A \oplus \ker ({\rm d}\pi_X) \rightarrow \ker ({\rm d}\pi_X) \cong {\rm T}X$ specified by $A$, and this in turn determines a covariant derivative ${\rm d}^A$ on sections of $\pi_X$ by ${\rm d}^A \phi := p^A \circ {\rm d} \phi$. Then one use may the complex structures $j_\Sigma$ and $j_X$ to construct the  holomorphic structure operator
$$
\phi \mapsto  \bar \partial^A_{j_X,j_\Sigma}\phi := \frac{1}{2} \left({\rm d}^A \phi +j_X \circ {\rm d}^A \phi \circ j_\Sigma\right)
$$
appearing in the first equation in (\ref{vortex}). The kernel of $  \bar \partial^A_{j_X,j_\Sigma}$ specifies the holomorphic sections of $\pi_X$, and indeed sections 
$\phi: \Sigma \rightarrow E_X$ in this kernel can be regarded as holomorphic maps with respect to $j_\Sigma$ and a complex structure $J^A_{E_X}$ induced on $E_X$ by $j_X$ and $A$, see~\cite{Kob}.

When $\Sigma$ is compact, a section $\phi$ of  the bundle $\pi_X$ determines a $G$-equivariant 2-homology class $[\phi]^G \in H_2^G(X;\ZZ)$ which will play the role of 
topological invariant  in the moduli problem associated to the vortex equations (\ref{vortex}). To see how this invariant arises, we take a classifying map $f: \Sigma \rightarrow {\rm B}G$ for the principal $G$-bundle $P \cong f^*{\rm E}G$, and consider the map $\tilde f \times \phi: P\rightarrow {\rm E}G \times X $, where $\tilde f$ is the lift of $f$ 
and $\phi$ is  again regarded as $G$-equivariant map. Since $\tilde f \times \phi$ is $G$-equivariant, it descends to a map $\tilde \phi: \Sigma \rightarrow {\rm E}G\times_G X$ to the Borel construction for the $G$-action on $X$. Then we take the fundamental class $[\Sigma] \in H_2(\Sigma;\ZZ)$ and set 
\begin{equation} \label{phiG}
[\phi]^G:= \tilde \phi_*[\Sigma] \in H_2({\rm E}G\times_G X;\ZZ)=: H_2^G(X;\ZZ).
\end{equation}

There is also a natural $G$-equivariant $2$-cohomology class  $[\eta]_G \in H^2_G(X;\ZZ)$ determined by  the equivariantly closed form of degree 2
\begin{equation} \label{eta}
\eta(\xi) = \omega_X - \langle \mu, \xi  \rangle\,  \in\,  ({\rm Sym}(\mathfrak{g}^*) \otimes \Omega^*(X))^G,\qquad \xi \in \mathfrak{g}
\end{equation}
in the Cartan complex of the $G$-action on $X$ \cite{GuiSteSS}. To each $G$-connection $A$ in $P\rightarrow \Sigma$ we can associate 
the closed 2-form
\begin{equation}\label{etaA}
\eta(A) = \omega_X - {\rm d}(\mu, A) \in \Omega^2(P\times X)
\end{equation}
which can be seen to descend through the quotient $q: P \times X \rightarrow E_X$. i.e. $\eta(A)=q^*\eta_{E_X}(A)$. The cohomology class $[\eta_{E_X}(A)] \in H^2(E;\ZZ)$ is in fact independent of $A$, so the evaluation at the connection $A$ can be interpreted as a $G$-equivariant version of the Chern--Weil homomorphism~\cite{Mun}.
This construction can be used to model the pairing of the classes   $[\eta]_G \in H_G^2(X;\ZZ)$ and $[\phi]^G \in H^G_2(X;\ZZ)$ through the formula
\begin{equation} \label{equivpairing}
\langle [\eta]_G , [\phi]^G \rangle = \int_\Sigma \phi^* \eta_{E_X}(A).
\end{equation}
For a solution $(A,\phi)$ of the system of PDEs  (\ref{vortex}), the quantity (\ref{equivpairing}) admits a physical interpretation as total energy of the field configuration, which also
ensures that it must be nonnegative (see \cite{CGS, MunHK} and our discussion in Section~\ref{sec:susy}).

The pairs $(A,\phi)$ playing the role of variables in the system of equations (\ref{vortex}) form the infinite-dimensional manifold 
\begin{equation} \label{fields}
\mathcal{F}(P,X,\Sigma):= \mathcal A(P) \times \Gamma(\Sigma,E_X),
\end{equation}
 where the first factor (the space of $G$-connections in $P$) is an affine space over the vector space $\Omega^1(X,\mathfrak{g})$, and the second factor denotes smooth sections of $\pi_X$. This manifold supports an action of the infinite-dimensional Lie group $\mathcal{G}={\rm Aut}_{\Sigma}(P)$ by $(A,\phi) \mapsto ({\rm Ad}_g A-\pi^* (g^{-1} {\rm d}g), g\cdot \phi)$, $g\in \mathcal{G}$. Each tangent space
$$
{\rm T}_{(A,\phi)} \mathcal{F}(P,X,\Sigma) \cong \Omega^1(\Sigma,\mathfrak{g}) \times \Gamma(\Sigma,\phi^*( P \times_G {\rm T}X))
$$
receives an induced complex structure that can be written locally as
$$J: (\dot A,\dot \phi) \mapsto (*_{g_{\Sigma}}\dot A, (\phi^*j_X)\dot \phi ),$$ where $*_{g_\Sigma}$ is the Hodge star operator of the metric $g_\Sigma$, as well as a
Hermitian inner product  
\begin{equation} \label{L2metric}
\langle\!\langle(\dot A_1,\dot \phi_1) ,(\dot A_2,\dot \phi_2)\rangle\!\rangle_{L^2(\Sigma)}  := \int_\Sigma\left( \langle   \dot A_1 , \dot A_2 \rangle_{\kappa_{\mathfrak{g}},g_\Sigma} +
\langle \dot \phi_1, \dot \phi_2\rangle_{\phi^*g_X,g_\Sigma} \right) 
\end{equation}
which can be regarded as a generalisation  of the usual $L^2$-metric on spaces of functions. It is easy to check that these two geometric structures are compatible on the space of all fields.

The linearisation of each of the equations in $(\ref{vortex})$ around a solution $(A,\phi)$ defines a subspace of $
{\rm T}_{(A,\phi)} \mathcal{F}(P,X,\Sigma) $ which is invariant under the infinitesimal 
 $\mathcal{G}$-action --- so $\mathcal{G}$ also acts on the space of solutions. In physics, one is interested in the whole set of solutions much less than on the spaces
 of $\mathcal{G}$-orbits
 \begin{equation}\label{moduli}
 \mathcal{M}^X_{\mathbf{h}}(\Sigma):= \{ (A,\phi) \in\mathcal{F}(P,X,\Sigma) \,|\, \text{equations (\ref{vortex}) hold and }[\phi]^G = \mathbf{h} \}/\mathcal{G}
 \end{equation}
 for each $\mathbf{h} \in H_2^G(X;\ZZ)$. When non-empty, theses spaces (which are referred to as {\em moduli spaces of vortices}) are finite-dimensional and possess
 mild singularities. Moreover, their locus of regular points receives a K\"ahler structure, which can be formally understood as a symplectic reduction of the $L^2$-metric
 (\ref{L2metric}) as follows. First of all by looking at suitable completions, one first needs to interpret (\ref{L2metric}) as a K\"ahler metric. The first equation in (\ref{vortex}) is preserved under the complex structure $J$, thus it cuts out a complex submanifold of $\mathcal{A}(P) \times  \Gamma(\Sigma,E_X)$  which becomes a K\"ahler manifold with the pull-back of the ambient symplectic form. In turn, the left-hand side of the
 second equation in (\ref{vortex}) can be recast  as a moment map for the $\mathcal{G}$-action on the complex submanifold, with respect to this induced symplectic structure. Under these conditions, the definition  (\ref{moduli}) thus corresponds to an infinite-dimensional analogue of the Meyer--Marsden--Weinstein quotient in the context of
 finite-dimensional K\"ahler geometry.
 
A simple example of this construction, which bypasses part of the analysis required on the space of fields, is obtained when $X=\CC$ with standard K\"ahler metric and action of $G={\rm U}(1)$. The corresponding moduli spaces
$\mathcal{M}^\CC_\mathbf{h}(\Sigma)$ were first studied in~\cite{Tau} for $\Sigma=\CC$. In this example $\Sigma$ is not compact, and  $\mathbf{h}$ should be
interpreted in terms of a winding number for $\phi$  at infinity; a key result is that  $\mathcal{M}^\CC_{k}(\CC)\cong S^k \CC \cong  \CC^k$ for positive winding $k \in \NN$.
 The case where $\Sigma$ is compact was studied e.g.~in \cite{Bra,GP}, and the equivalent result is
 \begin{equation}\label{modspU1C}
 \mathcal{M}^{\CC}_{k}(\Sigma) \cong
 \left\{
 \begin{array}{cl}
 \emptyset & \text{ if } \frac{\tau}{4 \pi} {\rm Vol}(\Sigma) <  k , \\
 {\rm Pic}^k \Sigma & \text{ if } \frac{\tau}{4 \pi} {\rm Vol}(\Sigma) =  k, \\
 S^k \Sigma & \text{ if } \frac{\tau}{4 \pi} {\rm Vol}(\Sigma) >  k,
  \end{array}
 \right.
 \end{equation}
where ${\rm Vol}(\Sigma):=\int_\Sigma \omega_\Sigma$, ${\rm Pic}^k\Sigma$ is the Picard variety parametrising holomorphic line bundles of degree $k$ on $\Sigma$, and we write the function $\mu^\kappa$ as $x \mapsto -\frac{1}{2{\rm i}}(|x|^2 - \tau)$ with $\tau \in \RR$.
One can regard this linear example  as a toy model for the nonlinear situation we want to address in this paper --- more specifically, our focus will be in
the case where a stability condition analogous to $\frac{\tau}{4 \pi} {\rm Vol}(\Sigma) >  k$ is imposed. The corresponding K\"ahler metrics on $\mathcal{M}^{\CC}_{k}(\Sigma)$ 
are still poorly understood; but see e.g.~\cite{ManRom}  for a discussion of the  limit  $\frac{\tau}{4 \pi} {\rm Vol}(\Sigma) \searrow  k$.

\subsection{Generalities on K\"ahler toric targets} \label{sec:genKaehler}

From now on, we want to focus on the special case where $(X, j_X, \omega_X)$ is a K\"ahler structure on a compact  toric manifold with real torus $\TT$, and use $G=\TT$ as the structure group of the gauge theory.

The most convenient way of realising our toric manifold $X$ (see~\cite{CLS}) is perhaps by starting from a fixed free Abelian group $M\cong \ZZ^n$, which determines $\TT:={\rm Hom} (M,{\rm U}(1)) \cong {\rm U}(1)^n$, and specify a convex polytope $\Delta \subset M _\RR := M\otimes_\ZZ \RR $  with the following properties:
\begin{enumerate}
\item[(i)] At each vertex of $\Delta$, exactly $n$ {\em edges}  (i.e.~line segments between two vertices) meet.
\item[(ii)] All {\em edge directions} in $\Delta$ (i.e~one-dimensional subspaces of $M_\RR$ generated by a  difference of vertices) are rational, in the sense of admitting generators in
the lattice $M$.
\item[(iii)] One can choose generators for the $n$ edge directions associated to each vertex of $\Delta$ to form a basis of $M$.
\end{enumerate}
It is common practice to refer to such $\Delta$ as a {\em Delzant polytope}~\cite{Del,Gui}.
The {\em in}ward-pointing normal directions to the (closed) facets of $\Delta$ determine
{\em rays} $\rho$ of a complete {\em fan} in the dual space $N_\RR:=M_\RR^\vee \cong \RR^n$,  denoted ${\sf Fan}_\Delta$; we will write as $\mathbf{n}_\rho \in N$ the primitive generator of the semigroup $\rho \cap N$, where $N:={\rm Hom}_\ZZ(M,\ZZ)\subset N_\RR$ is the dual lattice to $M$.  
The most basic construction in toric geometry (see~\cite{CLS}, section 3.1) defines a complex variety $X$ from a fan such as ${\sf Fan}_\Delta$ by glueing  together $n$-dimensional affine varieties associated to the {\em cones} in the fan. Each of these affine pieces contains  the complex torus $\TT_\CC\cong (\CC^*)^n$, which is the piece corresponding to the zero cone in $N_\RR$. The restrictions we have put on $\Delta$ imply that the variety $X$ is smooth and projective with ${\rm dim}_\CC X=n$, so we will treat it as a compact complex $n$-manifold $(X,j_X)$. The complex torus $\TT_\CC={\rm Hom} (M, \CC^*)  \supset {\rm Hom} (M, {\rm U}(1))=:  \TT$ also acts on $X$, and in fact $X$ can be regarded as a completion of $\TT_\CC$ to which the action of $\TT_\CC$ on itself can be extended as a holomorphic action.  However, we want to emphasise that in
our gauge-theory setting it is the compact {\em real} torus $\TT$ that plays a more prominent role. We will always assume that the compact toric manifold in our discussion is specified by a Delzant polytope $\Delta$, but shall write it as $X$ rather than $X_\Delta$ or $X_{\mathsf{Fan}_\Delta}$ for short.

We denote by ${\sf Fan}_\Delta(1)$ the subset of {rays} (i.e.~one-dimensional cones) in the fan  of $\Delta$. Each
ray $\rho$ determines a ($\TT_\CC$- and in particular) $\TT$-equivariant divisor $D_\rho:= \overline{O(\rho)}\cong X_{{\rm Star}(\rho)}$ in $X$ as its {\em orbit closure}, which is a compact toric variety itself (see~\cite{CLS}, Theorem~3.2.6 and Proposition~3.2.7). The $\ZZ$-linear map $\alpha$ given by
 $$\alpha: \mathbf m\mapsto \sum_{\rho \in {\sf Fan}_\Delta(1) } \langle \mathbf m,\mathbf n_\rho\rangle \, D_\rho, \qquad \mathbf m\in M$$ induces a short exact sequence of Abelian groups
 (cf.~\cite{CLS}, Theorem~4.1.3)
 \begin{equation} \label{divclass}
0\longrightarrow M   \stackrel{\alpha} {\longrightarrow } {\rm CDiv}_\TT (X) \stackrel{\beta} {\longrightarrow} {\rm Cl}(X) \longrightarrow 0.
\end{equation}
Here, ${\rm CDiv}_\TT (X) \cong \bigoplus_{\rho \in {\sf Fan}_\Delta(1)} \ZZ D_\rho$ denotes the group of $\TT$-equivariant (Cartier) divisors, whereas
${\rm Cl}(X)$ stands for the divisor class group of $X$. The latter coincides with the Picard variety ${\rm Pic}(X)$, since $X$ is smooth (\cite{CLS}, Proposition~4.2.6). In this subsection we want to  state some  facts that, on one hand, revolve around the basic short exact sequence (\ref{divclass}), and on the other hand relate to aspects of the $\TT$-equivariant homology and cohomology of $X$ relevant to our subsequent discussion.

Let us start with some topological preliminaries. Since we are assuming that $X$ is compact, its fan is complete (see~\cite{CLS}, Theorem~3.1.19(c)), and so $X$ is simply connected
(Theorem 12.1.10 in \cite{CLS}, ); hence $H_1(X;\ZZ)=0=H^1(X;\ZZ)$ by Hurewicz and the universal coefficient theorem. There is an isomorphism  $H^2(X;\ZZ)\cong {\rm Pic}(X)$ (\cite{CLS}, Theorem~12.3.2), so Proposition~4.2.5 in~\cite{CLS} implies that the group $H^2(X;\ZZ)$ is free Abelian; it is also freely generated by virtue of (\ref{divclass}). By the universal coefficient theorem, we conclude that also $H_2(X;\ZZ)\cong {\rm Hom}(H^2(X;\ZZ),\ZZ)$ is  a finitely generated free Abelian group.

We observe that there is a spectral sequence converging to the $\TT$-equivariant cohomology of $X$:
\begin{equation} \label{specseq}
E_2^{p,q}:=H^p({\rm B}\TT; H^q(X;\ZZ) ) \Rightarrow  H^{p+q}(X\times_\TT {\rm E}\TT;\ZZ).
\end{equation}
In total degree two,  (\ref{specseq}) degenerates to the
short exact sequence
\begin{equation}\label{degses}
0 \longrightarrow H^2({\rm B}\TT;\ZZ) {\longrightarrow} H^2_\TT (X;\ZZ) \longrightarrow H^2(X;\ZZ)\longrightarrow 0.
\end{equation}
The two nontrivial maps in (\ref{degses}) are induced by the structure maps of the fibre bundle $X \hookrightarrow X\times_\TT {\rm E}\TT \rightarrow {\rm B}\TT$. Since $H^2(X;\ZZ)$ is free Abelian, the short exact sequence (\ref{degses}) splits; 
thus we have an abstract isomorphism
$$
H^2_\TT(X;\ZZ) \cong H^2({\rm B}\TT;\ZZ) \oplus H^2(X;\ZZ).
$$
Note that  $H^2({\rm B}\TT;\ZZ)\cong \ZZ^n$ by the K\"unneth formula. There is an identification of this group with the lattice $M$ under a natural  isomorphism between
(\ref{degses}) and the basic short exact sequence (\ref{divclass}), namely:

\begin{lemma} \label{ladder}
There is an isomorphism of short exact sequences
$$
\xymatrix{
0 \ar[r] & M \ar[r] \ar[d]_{\cong} & {\rm CDiv}_\TT (X)  \ar[r]^\beta  \ar[d]_{\cong}^{c_1^\TT} &  {\rm Cl}(X) \ar[r] \ar[d]_\cong^{c_1} & 0 \\
0 \ar[r] &H^2({\rm B}\TT;\ZZ) \ar[r]^{\tilde \alpha} &  H^2_\TT (X;\ZZ) \ar[r] & H^2(X;\ZZ) \ar[r] & 0
}
$$
\end{lemma}

Let us sketch how to understand the ladder diagram above (see~Proposition~4 in~\cite{BokRomTFB} for further details).
A  divisor in $X$  such as $D_\rho$ gives rise
to a line bundle $\xi_{D_\rho}$ over $X$ (whose isomorphism class depends only on the divisor class $\beta(D_\rho)$). The restriction of this bundle 
to $D_\rho$  is the normal bundle of the (Weil) divisor $D_\rho$;   so the first Chern class
$c_1(\xi_{D_\rho})$ coincides with the Poincar\'e dual of the 2-homology class determined by $D_\rho$. Since $D_\rho$ is $\TT$-invariant,  $\xi_{D_\rho}$ is actually
an equivariant bundle (see~\cite{CLS}, section 12.4), hence we also obtain a complex line bundle 
\begin{equation}\label{equivlb}
\nu_{D_\rho}:=\xi_{D_\rho} \times_\TT {\rm E}\TT \longrightarrow X \times_{\TT} {\rm E}\TT.
\end{equation}
The first Chern class  $c_1(\nu_{D_\rho})$ of this bundle is an element of 
$H^2( X \times_{\TT} {\rm E}\TT;\ZZ)$, so we obtain a map on $\TT$-invariant divisors $c_1^\TT:= c_1 \circ \nu$ which extends to a homomorphism  ${\rm CDiv}_\TT (X) \rightarrow H^2_\TT (X;\ZZ)$. This provides a lift of the more familiar
map  $c_1:{\rm Pic}(X) \cong {\rm Cl}(X) \rightarrow H^2(X;\ZZ)$  taking the first Chern class of line bundles over $X$ (up to isomorphism) --- note that by surjectivity of the
map $\beta$ in (\ref{divclass}), all isomorphism classes of line bundles on $X$ contain $\TT$-equivariant representatives. The homomorphism $c_1^{\TT}$
 can be interpreted (with some abuse of terminology) as the evaluation of an equivariant first Chern
class~\cite{Web}.

At this point, we will divert our discussion from topology to geometry.
We start by recalling that there is an alternative way of obtaining $(X,j_X) $ as a quotient, which goes as follows (see~\cite{CLS}, Section~5.1).
 Consider the affine space  $\bigoplus_{\rho \in {\sf Fan}_\Delta(1)} \CC \rho\cong  H^2_\TT (X;\CC)$ with polynomial
 coordinate ring  generated by variables $x_\rho$; for convenience, we denote this affine space by $\CC^r$ with $r:= |{\sf Fan}_\Delta(1)|$. A cone $\sigma \in {\sf Fan}_\Delta$ determines a monomial
 $x^{\hat \sigma}:= \prod_{\rho \not\subset \sigma} x_\rho$, and these generate the so-called {\em irrelevant ideal} $B( {\sf Fan}_\Delta ) :=\langle x^{\hat \sigma} \,|\, \sigma \in {\sf Fan}_\Delta \rangle$. Applying the functor ${\rm Hom}_\ZZ(\cdot ,\CC^*)$ to the short exact sequence (\ref{divclass}) exhibits the group of characters  ${\rm Hom}_\ZZ ({\rm Cl}(X),\CC^*)$ as a subgroup of the complex torus $(\CC^*)^r$ acting on $\CC^r$:
 \begin{equation}
 1 \longrightarrow {\rm Hom}_\ZZ ({\rm Cl}(X),\CC^*)  \longrightarrow  (\CC^*)^r {\longrightarrow}\TT_\CC \longrightarrow 1.
 \end{equation}
 This subgroup acts on the complement 
 $$\CC^r_\Delta := 
  \CC^r  \setminus V(B( {\sf Fan}_\Delta ))$$
 of the exceptional set $V(B( \mathsf{Fan}_\Delta ))$, the affine variety associated to the irrelevant ideal, and one can show (see \cite{CLS}, Theorem~5.1.1) that
\begin{equation} \label{Xasquotient}
 X \cong \CC^r_\Delta  /{\rm Hom}_\ZZ ({\rm Cl}(X),\CC^*).
\end{equation}

It turns out that this quotient construction endows $X$ with a family of symplectic structures. This is because one can be recast $X$ as a symplectic quotient of the affine
variety $\CC^r_\Delta $, seen as a Hamiltonian subspace of $\CC^r$ with standard ${\rm U}(1)^r$-action, i.e. a product  of $r$ copies of the
standard  ${\rm U}(1)$-action $({\rm e}^{{\rm i} \theta},x)\mapsto  {\rm e}^{{\rm i} \theta} x$ on  $(\CC,\frac{\rm i}{2}{\rm d}x \wedge {\rm d}\bar x )$. More precisely,
let us identify ${\rm Lie}({\rm U}(1)^r)^* = ({\rm i}\, \RR^r)^*\cong  \RR^r$ using the standard Euclidean inner
product, and denote by $\mu_\Delta:=(\beta_\RR \circ \mu^{\rm st})|_{\CC^r_\Delta} $  the restriction to $\CC^r_\Delta$ of the composition of the moment map $\mu^{\rm st}:\CC^r \rightarrow \RR^r $ for the standard product  torus action, with components $\mu^{\rm st}_\rho (x_{\rho_1},\ldots, x_{\rho_r})=  \frac{1}{2} |x_\rho|^2$, with the extension $\beta_\RR$ to real coefficients   of the quotient map $\beta$ in (\ref{divclass}). For any choice of  a class $\delta$ in the K\"ahler cone  $\mathcal{K}(X) \subset H^{1,1}(X;\RR)\cong {\rm Cl}(X)\otimes_\ZZ \RR$, one has an action of the real character group ${\rm Hom_\ZZ({\rm Cl}(X),{\rm U}(1))}$ on the pre-image 
$\mu_\Delta^{-1}(\delta) \subset \CC^r_\Delta$. The space of orbits for this action acquires a complex structure from the one of $\CC^r$ (see Proposition~4.2 in~\cite{Hit}), and  there is a map 
 \begin{equation} \label{quotient}
\mu_{\Delta}^{-1}(\delta)/{\rm Hom}_\ZZ({\rm Cl}(X),{\rm U}(1)) \rightarrow X
\end{equation}
which is a biholomorphism (cf.~Section~8.4 in~\cite{MumFogKir}). This process produces  a Marsden--Weinstein--Meyer symplectic form $\omega_\delta \in \Omega^2(X;\RR)$, and one can check that in fact $[\omega_\delta]=\delta$ (see~\cite{Cox}, p.~399).
The compact real torus $\TT$  is recovered as the quotient of the real torus ${\rm U}(1)^r \subset  (\CC^\times)^r$ acting on  $\CC^r$ by the  subtorus in the left-hand side of (\ref{quotient}), and its action on $(X,j_X,\omega_{\delta})$ is both Hamiltonian and holomophic. Thus for each $\delta$ in the K\"ahler cone, one can speak of a canonical K\"ahler structure on $X$. Under this construction, there is also a moment map $\mu$ on $(X,\omega_\delta)$ induced from $\mu^{\rm st}|_{\CC^r_\Delta}$, and it turns out that $\mu(X)=\Delta$
has an interpretation as space of $\TT$-orbits in $X$.

Under additional  assumptions, various Moser-type results ensure that a symplectic structure $\tilde \omega_X$ on $X$ with $[\tilde \omega_X]=\delta$ will be symplectomorphic to $\omega_\delta$ in the sense that $\varphi^*\tilde\omega_X=\omega_\delta$ for some $\varphi\in {\rm Diff}(X)$ (see Section~7.3 of~\cite{McDSal}), but the symplectomorphism $\varphi$
need not relate compatible complex structures.
A theorem of Delzant~\cite{Del} asserts that if $\tilde \omega_X$ is $\TT$-invariant, there exists one such symplectomorphism $\varphi$ which is $\TT$-equivariant
(this amounts to a classification of compact symplectic toric manifolds by Delzant polytopes up to translations in $M_\RR$). In turn, Abreu~\cite{Abr} showed that the K\"ahler structures $(X,\tilde \omega_X,\tilde \jmath_X)$ and $(X,\omega_\delta,j_X)$ can be related by a $\TT$-equivariant biholomorphism $(X,\tilde\jmath_X)\rightarrow (X,j_X)$, but this will not be a symplectomorphism in general. 

We now want to review very briefly some geometry of cones associated to the toric manifold $X$. This is needed to describe certain positivity conditions that arise in the context of the vortex equations. The relevant cones are contained in real extensions of the Abelian groups in the ladder diagram of Lemma~\ref{ladder}, e.g. ${\rm Cl}(X)_\RR := {\rm Cl}(X)\otimes_\ZZ \RR$. For more detail, we refer the reader to Section~3.3 of reference~\cite{BokRomTFB}, where we deal with the case of vortices on K\"ahler manifolds of arbitrary dimension.

Recall that the  {\em nef cone} of a toric variety $X$ (denoted ${\rm Nef}(X) \subset {\rm Cl}(X)_ \RR$ in~\cite{CLS}), is generated by classes of numerically effective (Cartier) divisors. In the case where $X$ is smooth, there is a very explicit description of the dual of this cone (known as the {\em Mori cone} and denoted $\overline{NE}(X)$) due to Batyrev~\cite{Bat}, in terms of  {primitive collections} associated to the normal fan $\mathsf{Fan}_{\Delta}$.  In our setting, the Mori cone is always strongly convex (this follows from Proposition~6.3.24 in~\cite{CLS}), and so the primitive relations associated to the normal fan of $X$ give minimal generators for  $\overline{NE}(X)$, which determines the K\"ahler cone of the manifold by duality.

Now the linear maps $\tilde \beta_\RR$ and $\tilde \beta_\RR^*$, obtained from extending $\tilde \beta := c_1\circ \beta \circ (c_1^\TT)^{-1}$ in Lemma~\ref{ladder} to real coefficients and duality, 
relate the basic cones $\mathcal{K}(X)$ and $ \overline{NE}(X)$ to natural cones in $\TT$-equivariant cohomology and homology, respectively. First of all, we have the following fact from Section~3.3 of \cite{BokRomTFB}:

\begin{proposition} \label{thmcones}
Any ray in the nef cone of $X$ admits a generator of the form $\beta \left(\sum_{\rho\in \mathsf{Fan}_{\Delta}(1)} a_\rho D_\rho\right)$ with all $a_\rho \in \NN_0$.
\end{proposition}

We consider the cone
 \begin{equation}  \label{BPScone}
 \mathcal{C}_{\rm BPS}(X):= \left(\sum_{\rho \in  \mathsf{Fan}_{\Delta}(1)} \RR_{\ge 0}\,  c_1^\TT(D_\rho) \right)^\vee \subset H_2^{\TT}(X;\RR),
 \end{equation}
 which is dual to the closed strictly convex polyhedral cone in $H^2_\TT(X;\RR)$ generated by $\TT$-equivariant first Chern classes  (see (\ref{equivlb}) and Lemma~\ref{ladder})
 $$c_1^\TT(D_\rho)= c_1(\nu_{D_\rho})\in H^2_\TT(X;\ZZ) $$
 associated to effective $\TT$-equivariant divisors  $D_\rho \subset X$. In the context of the present paper,   $\mathcal{C}_{\rm BPS}(X)$ can be identified with
  the {\em BPS cone} introduced in Section~4 of reference~\cite{BokRomTFB} (for vortex equations on bases of arbitrary dimension), because $H_2(\Sigma;\ZZ)$ is cyclic and
  the fundamental class $[\Sigma]$ provides a natural generator.
 The intersection of $\mathcal{C}_{\rm BPS}(X)$ with the image of $N_1(X)$ in $H^\TT_2(X;\RR)$ under $\tilde \beta_\RR^*$ is again a cone, and it is contained within the image of the Mori cone  $\overline{NE}(X)$.
 
 Given $\mathbf{h} \in H_2^\TT(X;\ZZ) \cap  \mathcal{C}_{\rm BPS}(X)$, the following two positivity properties are a direct consequence of   the definition~(\ref{BPScone}) and Proposition~\ref{thmcones}:
 \begin{itemize}
 \item[(P1)]  $k_\rho := \langle c_1^\TT(D_\rho),\mathbf{h} \rangle \in \NN_0$ for each ray
 $ \rho \in \mathsf{Fan}_\Delta(1)$;
 \item[(P2)] $ \langle \eta^{\omega_\delta}_X , \mathbf{h}\rangle \ge 0$ for each K\"ahler class $\delta$ of $X$.
 \end{itemize}
 
The assertion (P1) implies that whenever a class $\mathbf h \in H_2^\TT(X;\ZZ)$ is taken from the cone $\mathcal{C}_{\rm BPS}(X)$, one obtains nonnegative $k_\lambda$ in
the prescription (\ref{charges}). In the next subsection, we clarify why one should identify rays $\rho$ in the normal fan of $\Delta$ with colours $\lambda$, in the context of
divisor braid groups. Assertion (P2) is an energy positivity condition for vortex configurations $(A,\phi)$ with target $X$ and topological charge $[\phi]^\TT = \mathbf h$ contained in the cone $\mathcal{C}_{\rm BPS}(X)$, independently of 
the K\"ahler form $\omega_X$ prescribed on  $X$.

 \subsection{Vortices in compact toric fibre bundles} \label{sec:vortoric}

As mentioned in the Introduction, one topic of this paper is  the topology of a certain type of
configuration spaces associated to an oriented surface $\Sigma$ and a decorated graph $(\Gamma,\mathbf{k})$. In this subsection, we give a more general definition 
(see Definition~\ref{defconf}) before specialising to the configuration spaces that are most directly related to divisor braids (see Definition~\ref{CF}). 
Our main goal here is to elucidate how these constructions arise quite naturally from the study of moduli spaces of vortices in (\ref{moduli}) and their topology.

\begin{defn} \label{defconf}
Let  $\Lambda$ be a simplicial complex, $\mathbf{k}:{\rm Sk}^0(\Lambda) \rightarrow \NN_0, \lambda \mapsto k_\lambda$ a nonzero function on its set of vertices, and $\Sigma$ a Riemann surface. Let us denote an
$\ell$-simplex in $\Lambda$ by $[\lambda_1;\ldots;\lambda_\ell]$, where the $\lambda_i$ are distinct vertices (which we may also refer to as {\em colours}). The {\em space of  effective divisors of degree $\mathbf{k}$ on $\Sigma$ braiding by $\Lambda$} is the subset of
$S^\mathbf{k} \Sigma:=\prod_{\lambda \in {\rm Sk}^0(\Lambda)} S^{k_\lambda} \Sigma$
given by 
\begin{equation}\label{confGamma}
{\sf Div}^\kk_+(\Sigma,\Lambda):= \left\{  \mathbf{d} \in  S^{\mathbf{k}} \Sigma:  [\lambda_1; \ldots; \lambda_\ell] \notin \Lambda \;\Rightarrow \; \bigcap_{i=1}^\ell
{\rm supp}(d_{\lambda_i})  = \emptyset \right\}.
\end{equation}
We refer to $(\Lambda,\kk)$ as a {\em colour scheme}, and to $\mathbf{k}$ as its {\em coloured degree}. 
\end{defn}

In this definition, we interpret the nonzero  components $d_\lambda \in S^{k_\lambda} \Sigma $ of $\mathbf d$ as effective divisors of degree $k_\lambda$ on $\Sigma$, and denote
 their support by ${\rm supp}(d_{\lambda}) \subset \Sigma$.  
Observe that  $S^\mathbf{k} \Sigma$ is a manifold equipped with a complex structure induced by the one of $\Sigma$, since the same holds for each Cartesian factor; a short argument uses essentially Newton's theorem on symmetric functions, see e.g.~\cite[p.~18]{ACGH}.  Clearly,
${\sf Div}^\kk_+(\Sigma,\Lambda)$ is an open dense complex submanifold, its dimension being given by the {\em total degree}
$$
 |\mathbf{k}|:= \sum_{\lambda \in {\rm Sk}^0(\Lambda)} k_\lambda \;=\; {\rm dim}_\CC\, {\sf Div}^\kk_+(\Sigma,\Lambda).
$$

\begin{lemma} \label{connected}
If $\Sigma$ is connected, then ${\sf Div}_+^\kk (\Sigma,\Lambda)$ is connected for any colour scheme $(\Lambda,\kk)$.
\end{lemma}
\begin{proof}
The Riemann surface $\Sigma$ is locally path-connected, so it is path-connected.
For each colour $\lambda$, we pick a point $z_\lambda\in \Sigma$, such that the points
$z_\lambda$ are pairwise distinct. This defines a base point
$\mathbf{d}^*=[k_{\lambda_1}z_{\lambda_1};\ldots;k_{\lambda_\ell}z_{\lambda_\ell}]\in  {\sf Div}^\kk_+(\Sigma,\Lambda)$.
Given an arbitrary 
$\mathbf{d}\in  {\sf Div}^\kk_+(\Sigma,\Lambda)$ we want to find a path from $\mathbf{d}$ to $\mathbf{d}^*$.

Because $\Lambda$ is a simplicial complex, it is easy to find a path in ${\sf Div}^\kk_+(\Sigma,\Lambda)$ from $\mathbf{d}$ to $\mathbf{d}'$ such that
the support of $\mathbf{d}'$ is distinct from the support of $\mathbf{d}^*$.
So we can assume without loss of generality that $\mathbf{d}\cap \mathbf{d}^*$ is empty.

Pick a point $z$ from the divisor $\mathbf{d}$. Let us say that its colour is $\lambda$. We can find a path
in $\gamma:[0,1]\to \Sigma $ such that $\gamma(0)=z$, $\gamma(1)=z_\lambda$ and 
$\gamma(0,1)\in \Sigma \setminus ({\rm supp}\,  \mathbf{d}\cup {\rm supp}\, \mathbf{d}^*)$. This gives a path 
in ${\sf Div}^\kk_+(\Sigma,\Lambda)$ from $\mathbf{d}$ to a divisor where $z$ is replaced by $z_\lambda$. 
We do this inductively for all the points in $\mathbf{d}$, and obtain a path from 
$\mathbf{d}$ to $\mathbf{d}^*$ as required. 
\end{proof}

\begin{defn}
\label{def.composite}
We say that a coloured degree $\mathbf{k}$ is 
\begin{itemize}
\item
{\em composite} if its components are not of the form
$k_\lambda=\delta_{\lambda \lambda'}$ for some vertex $\lambda'$, where $\delta$ is the Kronecker delta;
\item
{\em effective} if $k_\lambda \ne 0$ for all vertices $\lambda$;
\item
{\em very composite} if $k_\lambda \ge 2$ for all  vertices $\lambda$.
\end{itemize}
\end{defn}

If  $k_\lambda=0$ for a vertex $\lambda$, one may eliminate the $\lambda$-component of the coloured degree, as well as all the simplices in $\Lambda$ incident to $\lambda$
(in particular, the vertex itself), without affecting the definition of $\mathsf{Div}^\kk_+ (\Sigma,\Lambda)$. In particular, we see immediately that ${\sf Div}^\kk_+ (\Sigma,\Lambda)=\Sigma$ if
$\mathbf k$ is not composite. The reason we allow for this apparent redundancy is that
it is natural to consider ${\sf Div}^\kk_+ (\Sigma,\Lambda)$ as a family of spaces depending on $|{\rm Sk}^0(\Lambda)|$ integers, and in some circumstances it is convenient
to allow these integers to be zero; but the discussion of this paper will be restricted to effective coloured degrees.

\begin{remark} \label{lovenothate}
In the particular case where $\Lambda$ is a graph, which we emphasise by writing $\Lambda= \Gamma$, the set (\ref{confGamma}) is  reminiscent of the generalised configuration spaces $\Sigma_\Gamma$ or ${\rm Conf}_\Gamma(\Sigma)$ introduced in references~\cite{EasHug, CeyMar}, which depend on a manifold (here $\Sigma$) and a graph $\Gamma$. The definition given in these references agrees with ours
provided that in our conventions the graph $\Gamma$ is replaced by its {\em negative} $\neg \Gamma$ (a graph with the same set of vertices but complementary set of edges) and one takes all $k_\lambda=1$. Note that $\neg(\neg \Gamma)= \Gamma$. We justify  taking the negative (see e.g.\ equation~(\ref{ConfDiv}) below) by the fact that, for a simplicial complex $\Lambda$ of dimension higher than one (such as $\Lambda=(\partial \Delta)^\vee$ in Theorem~\ref{conjecture} below for $n\ge 3$),
the obvious generalisation of $\neg \Gamma$ does not yield a simplicial complex.
\end{remark}

We now go back to the vortex equations (\ref{vortex}), in the setting where the target $(X,j_X,\omega_X)$ is a K\"ahler toric manifold defined by a given Delzant polytope $\Delta$, together with a choice of K\"ahler form $\omega_X$.
Note that the condition $\mu(X)=\Delta$ fixes the translational ambiguity of the moment map: $\Delta$ determines $\mu$ and vice versa. As in Section~\ref{sec:vorteq}, we assume that $\Sigma$ is
a compact and connected oriented Riemannian surface with K\"ahler structure $(\Sigma,j_\Sigma,\omega_\Sigma)$, and denote by
$$
[\omega_\Sigma]^\vee \in H_2(\Sigma;\RR)\cong \RR
$$
the {\em dual K\"ahler class}, on which  the K\"ahler class $[\omega_\Sigma] \in H^2(\Sigma;\RR)$ evaluates as unity.

Let us take a fixed class $\mathbf{h}$ in the semigroup $H_2^\TT(X;\ZZ) \cap \mathcal{C}_{\rm BPS}(X)$ defined in equivariant 2-homology by the cone  (\ref{BPScone}). This will determine the homotopy
class of a principal $\TT$-bundle $P\rightarrow \Sigma$ where we want to consider as base
for the vortex equations (\ref{vortex}). One way to understand this is as follows: $\mathbf h$ determines a unique homomorphism
$\vec{\mathbf{h}}: H_2(\Sigma;\ZZ) \rightarrow H^\TT_2(X;\ZZ)$ satisfying 
$$\vec{\mathbf{h}} ([\Sigma])= \mathbf{h}$$ 
where $[\Sigma]$ is the fundamental class; specifically, if $\mathbf{h} = {\phi}^\TT$ as in equation (\ref{phiG}), then $\vec{\mathbf{h}} = \tilde \phi_*$.
Composing with the homomorphism
$\tilde{\alpha}^*:H_2({\rm B}\TT;\ZZ) \rightarrow H_2^\TT(X;\ZZ)$ obtained by dualising $\tilde \alpha$ in the diagram of Lemma~\ref{ladder}, we get the map
\begin{equation}\label{1stChernclass}
\tilde \alpha^*  \circ \vec{\mathbf{h}} : H_2(\Sigma;\ZZ) \rightarrow H_2({\rm B}\TT;\ZZ)\cong \ZZ^{\oplus n},
\end{equation}
which can be interpreted as an element of $  H^2(\Sigma;\ZZ) \otimes_\ZZ H_2({\rm B}\TT;\ZZ)\cong H^2(\Sigma;\ZZ)^{\oplus n}  $. We then set
$$
c_1(P) = \tilde \alpha^* \circ \vec{\mathbf{h}},
$$
and by a well-known property of the first Chern classes this determines $P\rightarrow \Sigma$ up to homotopy. We shall denote by $(\tilde \alpha^*  \circ \vec{\mathbf{h}} )_\RR$
the extension of the $\ZZ$-linear map (\ref{1stChernclass}) to real coefficients.

The following result provides the main link connecting divisor braid groups to moduli spaces of vortices on Riemann surfaces.

\begin{thm} \label{conjecture}
Suppose that the charge $\mathbf{h} \in  H_2^\TT(X;\ZZ) \cap \mathcal{C}_{\rm BPS}(X)$ is such that
\begin{equation} \label{stability}
(\tilde \alpha^* \circ   \vec {\mathbf{h}})_\RR ([\omega_\Sigma]^\vee)\; \in\; {\rm int}\, \mu^\kappa (X).
\end{equation}
Then the moduli space $\mathcal{M}^X_\mathbf{h}(\Sigma)$ of the vortex equations (\ref{vortex}) on $(\Sigma,j_\Sigma, \omega_\Sigma)$ with target $(X, j_X, \omega_X, \TT, \mu^\kappa)$, defined in (\ref{moduli}), is nonempty. Setting 
$\mathbf{k} (\rho):=\langle c_1^\TT(D_\rho), \mathbf{h}\rangle$ for each $\rho \in \mathsf{Fan}_{\mu(X)}(1)$  as in (\ref{charges}),  there is a homeomorphism
\begin{equation}\label{eqconjecture}
\mathcal{M}^X_\mathbf{h}(\Sigma) \; \cong \; {\mathsf{Div}}_+^{\mathbf{k}}(\Sigma,(\partial \mu (X))^\vee).
\end{equation}
\end{thm}

In equation (\ref{eqconjecture}), the moduli spaces are identified with spaces of effective divisors on $\Sigma$ (as in Definition~\ref{defconf}) braiding by a simplicial complex constructed from the Delzant polytope
$\Delta = \mu(X)$ defining the $n$-dimensional toric target manifold $X$. This simplex is obtained by dualising the boundary of $\Delta$ (interpreted as an $(n-1)$-dimensional
spherical polyhedron). It is easily checked that condition (i) in the second paragraph of
Section~\ref{sec:genKaehler} implies that such a dual polyhedron forms a simplicial complex, for any Delzant polytope $\Delta$.

The assumption~(\ref{stability}) (where `int' denotes the interior) can be interpreted as a natural stability condition. The closed version of this equation,
$$
(\alpha_\RR^* \circ   \vec {\mathbf{h}}) ([\omega_\Sigma]^\vee)\; \in\;  \kappa(\Delta),
$$
is  a necessary condition for existence of vortex solutions --- this was shown by Baptista as Theorem~4.1 of~\cite{Ba1}, using essentially the convexity of $\mu(X)$; 
see~\cite{Gui}. In the toy example of Section~\ref{sec:vorteq} where $X=\CC$, condition (\ref{stability}) corresponds to the third alternative on the right-hand side of~(\ref{modspU1C}).

Theorem~\ref{conjecture} is a particular case of Theorem~4 in our companion paper~\cite{BokRomTFB}, where we consider vortex equations with compact K\"ahler toric targets $X$ on compact K\"ahler manifolds $Y$ of arbitrary dimension. This result positively answers a question/conjecture formulated by Baptista in Section~7 of~\cite{Ba1}. Previously, the
identification (\ref{eqconjecture}) had only been verified in the case where $X=\PP^n$ equipped with its Fubini--Study K\"ahler structure in~\cite{Ba1}. Early versions of this result
in the special case $n=1$ (i.e. $X=\PP^1$) had appeared independently
in~\cite{Mun} and~\cite{Sib2Yan}.

The proof of Theorem~\ref{conjecture} does provide further intuition about Abelian vortices.
The basic idea is that, under assumption (\ref{stability}), there is a one-to-one correspondence (modulo gauge equivalence) between solutions of the Abelian vortex equations (\ref{vortex}) on $\Sigma$ with nonlinear 
$n$-dimensional toric targets $X=X_\Delta$ and vortex solutions (also on $\Sigma$) with linear $r$-dimensional targets $\CC^r_\Delta$. This is in the same spirit of the more familiar correspondence between {\rm un}gauged nonlinear sigma-models with toric targets and linear sigma-models that is familiar in the physics
literature~\cite{WitP2d}, where the former appear as low-energy effective theories for the latter. For a further application of this idea to the study of the geometry of moduli spaces, we refer the reader to~\cite{RomSpe}.

When $X$ is given a canonical
symplectic structure $\omega_X=\omega_\delta$, Baptista showed that one can descend vortices with target $\CC^r_\Delta$ through the quotient (\ref{Xasquotient}); an important
ingredient was a simplified version of the the Hitchin--Kobayashi correspondence proved by Mundet in~\cite{MunHK}. However, the result by Abreu in~\cite{Abr} that we mentioned in Section~\ref{sec:genKaehler} and the fact that the complex structure $\omega_X$ does not appear in the PDEs (\ref{vortex}) allows to deform solutions within the same target K\"ahler class. Since the moduli of linear vortices are parametrised by the right-hand side of (\ref{eqconjecture}) and the descent map is injective, this argument allowed Baptista to
interpret ${\mathsf{Div}}_+^{\mathbf{k}}(\Sigma,(\partial \mu (X))^\vee)$ as a family of vortex moduli with target $X$.

The proof in~\cite{BokRomTFB} that this descent process is surjective is based on a construction of a lift of a solution $(A,\phi)$ of (\ref{vortex}) to a vortex with linear target $\CC^r_\Delta$.
First of all, one can define a principal $\TT \times {\rm Hom}_\ZZ ({\rm Cl}(X),{\rm U}(1))$-bundle on $\tilde P \rightarrow \Sigma$ and a smooth section an associated bundle with fibre
$\CC^r_\Delta$ lifting $\phi$ which is unique up to homotopy; in this step, it is crucial that the underlying structure groups are commutative. Then one lifts the connection $A$ in such
a way as to make the lifted section holomorphic. These two steps do not yet guarantee that the real vortex equation (involving the moment map) is satistfied, but once again
the Hitchin--Kobayashi correspondence of Mundet provides a complex gauge transformation to a full solution of the vortex equations with target $\CC^r_\Delta$.

In Section~\ref{sec:susy} below, we shall justify that the fundamental groups of the vortex moduli spaces $\mathcal{M}^X_\mathbf{h}(\Sigma)$  have direct physical interest. Given the description (\ref{eqconjecture}), the calculation of such groups is simplified by observing that one can model them on fundamental groups of simpler
spaces of effective braiding divisors on $\Sigma$, for which the braiding is determined by a suitable graph:

\begin{proposition} \label{onlySk1}
Let $\Sigma$ be a connected Riemann surface and $(\Lambda,\mathbf{k})$ a colour scheme as in Definition~\ref{defconf}. Then there is an isomorphism of fundamental groups
\begin{equation} \label{sk1enough}
\pi_1\, {\sf Div}^\kk_+(\Sigma,\Lambda) \cong \pi_1\, {\sf Div}^\kk_+(\Sigma,{\rm Sk}^1(\Lambda)).
\end{equation}
\begin{proof}
By Lemma~\ref{connected}, 
the  space of braiding effective divisors is connected, so we can base the fundamental group at a tuple $\mathbf{d}^* = (d^*_{\lambda})_{\lambda \in {\rm Sk}^0(\Lambda)}$ containing only reduced divisors $d^*_{\lambda} \in S^{k_\lambda} \Sigma$ such that $\bigcap_{\lambda \in {\rm Sk}^0(\Lambda)} {\rm supp}(d_\lambda) = \emptyset$, and interpret 
loops in ${\sf Div}^\kk_+(\Sigma,\Lambda)$  as closed coloured braids in $S^1\times \Sigma$; the colours of the strands are determined by the colours of their basepoints.
Then the main task is to show that, when one identifies all possible loops using the homotopies naturally specified by the simplicial complex $\Lambda$, intersections involving more than two strands occur in codimension higher than two, and thus do not give rise to further relations in the fundamental group. 
(In Section~\ref{sec:variegated}, we shall be more specific about the type of homotopies that we want to allow.)
A careful proof of this fact would involve a filtration argument,
taking advantage of the simplicial property of $\Lambda$ to resolve arbitrary intersections of strands into those involving only two colours. We refrain from spelling out the full argument in detail, since the same technique will also be employed to justify  {\em Claim~1} in the proof of Proposition~\ref{lambdaiso2} below.
\end{proof}
\end{proposition}

According to equation (\ref{sk1enough}), braiding by a graph (more precisely, by the 1-skeleton of the relevant simplicial complex, given in equation (\ref{eqconjecture})) is sufficient to describe $\pi_1 \mathcal{M}^X_\mathbf{h}(\Sigma)$. Graphs occurring as 1-skeleta of simplicial complexes such as $(\partial \mu (X))^\vee$ satisfy the following properties:
\begin{itemize}
\item[(i)]
they are simple, i.e.~do not contain multiple edges connecting two given vertices;
\item[(ii)]
they do not contain any {\em self-loops} (i.e.~edges beginning and ending at the same vertex).
\end{itemize}
It is clear that  these two properties define a much more general class of graphs  than those obtained as 1-skeleta of dualised boundaries of Delzant polytopes. This class is closed under the involution  $\neg$ (taking the negative of a graph) mentioned in Remark~\ref{lovenothate}.
Since a space of effective divisors braiding by such a graph $\Gamma$ is reminiscent of a configuration space, we shall write (see also Definition~\ref{CF})
\begin{equation}\label{DivConf}
{\sf Div}^\kk_+(\Sigma,\Gamma) =: {\sf Conf}_\kk(\Sigma, \neg \Gamma)
\end{equation}
to make contact with the usual notation --- which is recovered by setting all $k_\lambda=1$.

For an undirected graph $\Gamma$ satisfying properties (i) and (ii) above, a function $\mathbf{k}: {\rm Sk}^0(\Gamma)\rightarrow \NN$ and an orientable connected surface $\Sigma$,
we sketched the notion of {\em divisor braid group}  $\mathsf{DB}_\mathbf{k}(\Sigma,\Gamma)$ in the Introduction to this paper. In Section~\ref{sec:fundamental}, we shall add more precision to that first description (see Definition~\ref{DB}), and establish in Proposition~\ref{lambdaiso2} that there is an isomorphism
\begin{equation} \label{DBisoFG}
\mathsf{DB}_\mathbf{k}(\Sigma,\Gamma) \cong \pi_1\, {\sf Div}^\kk_+(\Sigma,\neg \Gamma),
\end{equation}
where the right-hand side is the fundamental group of the generalised configuration space ${\sf Conf}_\kk(\Sigma,\Gamma)$ (see Definition~\ref{CF}).
At present, the point we want to make, which relies on Theorem~\ref{conjecture}, is that there are many divisor braid groups $\DivBraid_\kk (\Sigma,\Gamma)$ that can be realised as fundamental groups of moduli spaces of vortices on $\Sigma$ valued in a  toric target $X=X_{{\sf Fan}_\Delta}$. Making use of Proposition~\ref{onlySk1}, one obtains namely
\begin{equation} \label{fundgroups}
\pi_1 \, \mathcal{M}_\mathbf{h}^{X}(\Sigma)  \cong  \mathsf{DB}_{\mathbf{k}} (\Sigma,\neg \, {\rm Sk}^1((\partial \Delta)^\vee))
\end{equation}
with coloured degree $\kk$ given as claimed in (\ref{charges}), i.e.
$$
\mathbf{k}(\rho)=\langle \mathbf{n}_\rho, \mathbf{h}\rangle \;\; \text{ for each }\;\; {\rho \in {\rm Sk}^0(\neg {\rm Sk}^1 ( (\partial \Delta)^\vee) ) =  {\sf Fan_\Delta (1)}}.
$$

 We shall come back to the gauge-theoretic viewpoint  (\ref{fundgroups})  in Section~\ref{sec.examples}, illustrating how certain groups of such
 ``realisable'' divisor braids turn out to be interesting and computable
 for simple targets $X=X_{\mathsf{Fan}_\Delta}$. Actually, the original motivation for the present work was to devise  basic tools for such computations.
 In the rest of this section, we explain briefly why results of this sort are relevant in the context of two-dimensional quantum field theory.

\subsection{Supersymmetric gauged sigma-models and vortex moduli} \label{sec:susy}

Let us briefly recall the definition of the $(1+2)$-dimensional   bosonic gauged nonlinear sigma-model on an oriented Riemannian surface  $(\Sigma,g_\Sigma)$ with K\"ahler target $(X,j_X,\omega_X)$, admitting a Hamiltonian holomorphic $G$-action with moment map $\mu$. Solutions of this model can be understood as paths $I \rightarrow \mathcal{F}(P,X,\Sigma)$, i.e.~maps from a real interval $I$ (parametrised by a variable $t$ representing time) to the space of fields defined in (\ref{fields}), 
where $P\rightarrow \Sigma$ is a principal $G$-bundle and $E_X\rightarrow \Sigma$ is as in (\ref{bundleE}). We look for paths satisfying the Euler--Lagrange equations of the action functional
\begin{equation} \label{sigmaaction}
\mathcal{S}_\xi (\tilde A,\phi)= - \| F_{\tilde A} \|_{L^2} ^2 + \| {\rm d}^{\tilde A}  \phi \|_{L^2}^2 - \xi \|\mu \circ \phi \|_{L^2}^2  
\end{equation}
with appropriate conditions on their values at one or maybe both ends of $I$. These equations are second-order PDEs on the manifold $I \times \Sigma$.
In (\ref{sigmaaction}), $\tilde A := A_t {\rm d}t + A(t)$ is a connection (with curvature $F_{\tilde A}$) on the pull-back $p_2^*P$ under the
projection  $p_2:\RR \times \Sigma\rightarrow \Sigma$, while the component $ \phi$ can be interpreted as a path $I \rightarrow C_G^\infty(P,X)$ of $G$-equivariant maps; $\| \cdot \|_{L^2}$ is used to denote
the $L^2$-seminorms on $I\times \Sigma$ determined by the
Lorentzian metric ${\rm d}t^2-g_\Sigma$, by the K\"ahler metric $g_X$ on $X$, and by $G$-invariant bilinear forms $\kappa_\mathfrak{g}$ and $\kappa_{\mathfrak{g}^*}$ associated to a given nondegenerate $G$-equivariant map $\kappa: \mathfrak{g}^* \rightarrow \mathfrak{g}$; and $\xi$ is a positive real parameter.

The integrands in (\ref{sigmaaction}) are invariant under paths to the unitary gauge group $\mathcal{G}:={\rm Aut}_\Sigma(P)$. We are interested in the gauge equivalence classes of solutions to the Euler--Lagrange equations of the functional, and we will argue that this problem simplifies considerably for the value $\xi=1$, which is  referred to as the {\em self-dual} point (or regime) of the sigma-model.

The Lagrangian, or time-integrand in (\ref{sigmaaction}), can be recast as the difference between a non-negative kinetic energy and a 
potential energy given by the positive functional
\begin{equation}\label{potential}
\mathcal{V}_\xi (A(t),\phi(t))=\int_\Sigma  \left( \left|F_{A(t)}\right|_{\kappa_{\mathfrak{g}},g_\Sigma}^2 + \left|{\rm d}^{A(t)} \phi(t)\right|_{g_X,g_\Sigma}^2 + \xi |\mu \circ \phi(t)|_{\kappa_{\mathfrak{g}^*},g_\Sigma}^2\right).
\end{equation}
Assuming $\Sigma$ compact for simplicity, one can rewrite the integral in (\ref{potential}) as in~\cite{Mun}
\begin{eqnarray}
\mathcal{V}_\xi (A(t),\phi(t)) &=& \pm \int_\Sigma \phi(t)^*(\eta_E(A(t))) + {(\xi -1)} \int_\Sigma |\mu\circ \phi(t)|_{\kappa_{\mathfrak{g}^*},g_\Sigma}^2 \nonumber \\
&& +  \int_\Sigma  \left|F_{A(t)} \pm (\mu^\kappa \circ \phi(t))\, \omega_\Sigma\right|_{\kappa_{\mathfrak{g}},g_\Sigma}^2   \label{Bog} \\
&& +  \int_\Sigma \left|{\rm d}^{A(t)}\phi(t) \pm j_X\circ {\rm d}^{A(t)}\phi (t)\circ j_\Sigma\right|_{g_X,g_\Sigma}^2 ,  \nonumber
 \end{eqnarray}
where $\eta_E(A)$ denotes a 2-form on $E=P\times_G X$ pulling back to (\ref{etaA}), as before; this rearrangement is sometimes called the Bogomol'ny\u\i\ trick. 

The two possible choices of signs above give distinct ways of rewriting equation~(\ref{potential}), and the choice that renders the first term on the right-hand side of (\ref{Bog}) non-negative provides a useful estimate on the energy of the system. Recalling equation (\ref{equivpairing}) we observe that, once this choice is made, one can write the first term
as $|\langle [\eta]_G,[\phi(t)]^G\rangle|$. This quantity only depends on the homotopy class of the path $t \mapsto \phi(t)$, i.e.~it is a constant determined by the connected component
of $\mathcal{F}(P,X,\Sigma)$ where the path $t\mapsto (A(t),\phi(t))$ takes values.
To simplify things further, we also assume that $\xi=1$, so that the second term in~(\ref{Bog}) vanishes; then we conclude that
$\mathcal{V}_1 (A(t),\phi(t))\geq |\langle [\eta]_G,[\phi]^G\rangle| $. 

Now suppose that, for each $t \in I$,  the field configuration $(A(t),\phi(t))$ is such that the last two integrands of (\ref{Bog}) for the choice of sign made above vanish. 
By definition, the kinetic energy vanishes for a {\em static} field, so if in addition this $ t \mapsto (A(t),\phi(t))$ is constant, then it minimises not only $\mathcal{V}_1$ 
with value $ |\langle [\eta]_G,[\phi]^G\rangle|$, but also the energy of the system, and it follows that it represents a critical point for
the functional $\mathcal{S}_1$. On the other hand, a field can only be {\em stable} (i.e.~minimise the total energy) if their  kinetic energy  vanishes
and $\mathcal{V}_1 (A(t),\phi(t))$ attains its minimum value within the given homotopy class. Thus the equations defined by the vanishing of the last two integrands
(pointwise on $\Sigma$) define the static and stable field configurations of the model. The choice of upper sign in (\ref{Bog}) corresponds to the vortex equations (\ref{vortex}),
whereas the lower sign defines anti-vortices. There may still exist genuinely dynamical solutions to the Euler--Lagrange equations of $\mathcal{S}_1$, in particular in components where no vortex or antivortex exist.

Ultimately, we are interested in field configurations in $\mathcal{F}(P,X,\Sigma)$ modulo the action of the gauge group.   
There is a version of Gau\ss{}'s law, obtained as the Euler--Lagrange equation of (\ref{sigmaaction}) corresponding to $A_t$, expressing that
the path $t\mapsto (A(t),\phi(t))$ is orthogonal to  the  $\mathcal{G}$-orbits in the spatial $L^2$-norm; this equation has the special status of
a constraint on the space of fields, just as the component $A_t$ is to be regarded as an auxiliary field (its time derivative does not appear in the integrand of (\ref{sigmaaction})),
and it ensures that dynamics of the action $\mathcal{S}_\xi$ descends to a dynamical system on $\mathcal G$-orbits.

If the boundary conditions incorporate the vortex equations (\ref{vortex}) in some way, one may try (following Manton~\cite{ManSut})
 to approximate a time dependent field solving the Euler--Lagrange equations of the sigma-model by a time-dependent
solution to the vortex equations. The hope is to approximate the kinetic terms in the action (\ref{sigmaaction}) by the squared norm 
in the metric on the moduli space of vortices described in Section~\ref{sec:vorteq}; 
this is sensible as far as the trajectory $(A(t),\phi(t))$ remains close to vortex configurations. 
Since the potential energy for vortices at $\xi = 1$ within a given class $\mathbf{h}$ remains constant, 
this proposal amounts to approximating dynamics in the sigma-model
by geodesic motion on the moduli space $\mathcal{M}_{\mathbf{h}}^X(\Sigma)$. Rigorous study of the simplest models 
(for $\Sigma=\CC$, $X=\CC$ and $G={\rm U}(1)$ in~\cite{Stu}, and for $\Sigma=T^2$, $X=\PP^1$ and trivial $G$ in~\cite{Spe}) revealed that this approximation is sound for a Cauchy problem with initial velocities that are small and tangent to the moduli space, and that it is even robust for small values of $|\xi-1|$, provided that a  potential energy term (corresponding to the restriction of second integral in~(\ref{Bog}) to the moduli space) is included in the moduli space dynamics. In other words: at slow speed, the classical $(1+2)$-dimensional field theory with target $X$ and action (\ref{sigmaaction}) is well described by a $(1+0)$-dimensional sigma-model with target $\mathcal{M}_{\mathbf{h}}^X(\Sigma)$.

One should hope to take a further step, and try to understand the `BPS sector'  of the quantisation of these sigma-models, for 
$X$ a K\"ahler toric manifold and $G=\TT$ its embedded torus,
via geometric quantisation of the truncated classical phase space
$$\coprod_{{\mathbf{h}\in H_2^\TT(X;\ZZ)\cap \mathcal{C}_{\rm BPS}(X)}}  {\rm T}^*\mathcal{M}_{\mathbf{h}}^X(\Sigma) $$
with the canonical symplectic structure. This semiclassical regime should capture the physics at energies that are close to the ground states of the model in each topological sector within
the BPS cone (with an obvious extension for the anti-BPS cone).

An important feature of the sigma-model described by $\mathcal{S}_1$ is
that it admits a supersymmetric extension  --- more precisely, for the topological A-twist this extension is described in detail by Baptista in~\cite{BapTSM}, which yields an $N=(0,2)$ 
supersymmetric topological field theory whose quantum observables localise to vortex moduli spaces
as Hamiltonian Gromov--Witten invariants~\cite{CGMS}. This is also the case for 
the effective slow-speed approximation defined above --- in fact, it is  known~\cite{HV} that one-dimensional sigma-models onto a K\"ahler target such as  
$\mathcal{M}_{\mathbf{h}}^X(\Sigma)$ extend to models with $N=(2,2)$ supersymmetry, and thus by adding fermions to our effective (0,1)-dimensional model one could even accommodate the amount of local supersymmetry present in the classical theory. 

According to Witten~\cite{WitSMT}, the ground states in the
effective supersymmetric quantum mechanics on the moduli space should be described by complex-valued harmonic (wave)forms on each component $\mathcal{M}_{\mathbf{h}}^X(\Sigma)$. The relevant Laplacian is the one associated to the natural K\"ahler metric (\ref{L2metric}), which is defined from the kinetic terms of the parent bosonic model (\ref{sigmaaction}). The supersymmetric parity of such states will be given by their degree as forms reduced mod 2. However, multi-particle
quantum states, corresponding equivariant homotopy classes $\mathbf h$ with composite coloured degrees $\mathbf k$ (in the sense of Definition~\ref{def.composite}) should also be interpreted in terms of individual
solitons. This leads to the expectation that the Hilbert spaces corresponding to the quantisation of each component $\mathcal{M}_{\mathbf{h}}^X(\Sigma)$ with coloured
degrees (\ref{charges}) should split nontrivially into sums of tensor products of elementary Hilbert spaces corresponding to individual particles. The individual vortex species corresponding
to different facets of the Delzant polytope $\Delta$ determining the target $X$ should be among these individual particles, but there could well be more types of particles needed to
account for the observed ground states. A case-study where extra (composite) particles occur, which still fit into a common framework~\cite{BokRomP}, is considered in~\cite{RomWeg}.

It is natural to extend the semiclassical approximation slightly by allowing nontrivial holonomies (i.e.~anyonic phases) of the multi-particle waveforms in supersymmetric quantum mechanics, as in the Aharonov--Bohm effect for electrically charged particles. This is implemented by advocating, as in~\cite{Wit-QFS},  that the waveforms be valued in local systems over the moduli space. It is well-known that the different choices involved are parametrised by representations of $\pi_1\, \mathcal{M}_{\mathbf{h}}^X(\Sigma)$, each representation corresponding to a flat connection. The whole picture is familiar from the quantisation~\cite{OliWes} of dyonic particles in $\RR^3$ where, in addition to a topological magnetic charge $k\in \NN$ (specifying a moduli space ${M}^0_k$ of centred BPS $k$-monopoles~\cite{AtiHit}), one needs to choose a representation of $\pi_1 ({M}^0_k) \cong \ZZ_k$ specifying the electric charge of the dyon.

As we will see (in contrast with the case of BPS monopoles), the fundamental group $\pi_1\, \mathcal{M}_{\mathbf{h}}^X(\Sigma)$ in the context of vortices turns out to be infinite, leading to a continuum of representations. In this situation, it is physically more natural to deal with distributions of flat connections over a measurable subset of the representation variety
rather than a specific representation. An alternative~\cite{BokRomP,BokRomWeg} to dealing with a collection of bundles over $\mathcal{M}_{\mathbf{h}}^X(\Sigma)$ is to perform the quantisation over a cover $ \widetilde{\mathcal{M}}_{\mathbf{h}}^X(\Sigma)$ of the moduli space where the relevant local systems trivialise --- of course, this will always be the case for the universal cover. Then one views 
linear combinations (or wavepackets) of waveforms representing each quantum multi-particle state over a distribution of representations (see also~\cite{AtiEO}) as elements of 
the $L^2$-completion
\begin{equation} \label{BPSsector}
\bigoplus_{{\mathbf{h}\in H_2^\TT(X;\ZZ)\cap {\mathcal{C}}_{\rm BPS}(X)}} \!\!\!\!\!\!\!\!\!\!\!\!\!\!\!
L^2\Omega_c^*\left( \widetilde{\mathcal{M}}_{\mathbf{h}}^X(\Sigma), \partial_0  \widetilde{ \mathcal{M}}_{\mathbf{h}}^X(\Sigma) ;\CC \right)
\end{equation}
of a space of compactly supported forms. Here, we admit that  the the covering spaces $ \widetilde{\mathcal{M}}_{\mathbf{h}}^X(\Sigma)$ may not be cocompact, and we do not assume that their metrics (which are just the pull-backs of (\ref{L2metric}) from the quotients) are necessarily complete; but that one can distinguish a component  $\partial_0  \widetilde{ \mathcal{M}}_{\mathbf{h}}^X(\Sigma)$ of each boundary mapping surjectively onto the subset of the boundary of  $\mathcal{M}_{\mathbf{h}}^X(\Sigma)$ that is accessible to finite-time geodesic flow. These assumptions take into account the most recent analytic work on the geometry of the moduli of nonlinear vortices~\cite{RomSpe}, which revealed non-completeness and finiteness of volume over compact subsets of the surface $\Sigma$. Then
it is physically sensible  to impose vanishing conditions for the waveforms over  $\partial_0  \widetilde{ \mathcal{M}}_{\mathbf{h}}^X(\Sigma)$, and we incorporate this in the notation (\ref{BPSsector}).

The BPS sector (\ref{BPSsector}) of the quantum Hilbert space is an infinite-dimensional vector space, but it admits an action of
the group of deck transformations, which will be the whole fundamental group $\pi_1  \mathcal{M}_{\mathbf{h}}^X(\Sigma)$ if we work directly on the universal cover. In order to obtain information about interesting subspaces (where this group acts), 
we need to proceed statistically in the sense of Murray--von Neumann dimensions~\cite{MurVNeu,Lue} over the group von Neumann algebra~(see \cite{BroOza}, p.~43)
\begin{equation}\label{gvNa}
\mathcal{N}(\pi_1  \mathcal{M}_{\mathbf{h}}^X(\Sigma)):=  \mathcal{B}(\ell^2(\pi_1  \mathcal{M}_{\mathbf{h}}^X(\Sigma)))^{\pi_1  \mathcal{M}_{\mathbf{h}}^X(\Sigma)}.
\end{equation}
In particular, for the study of the ground states (where one should restrict to harmonic forms), this problem reduces to computing analytic $L^2$-Betti numbers~\cite{AtiEO,Lue} of the covering space equipped with the group action. In the usual setting of geometric quantisation, where quantum states take values in line bundles, it is sufficient to use the maximal Abelian cover and
representations of the fundamental group at rank one,
as in~\cite{BokRomP}. Note that this is equivalent to working with the Abelianisation $H_1( \mathcal{M}_{\mathbf{h}}^X(\Sigma) ;\ZZ)$. As we shall see, in our situation of Abelian vortices this a free Abelian group, so the machinery of $L^2$-invariants  has the flavour of classical Fourier analysis in this setting.

For a classical phase space whose fundamental group is nonabelian, one may also go beyond line bundles, and construct nonabelian waveforms valued in
local systems of higher rank. Such waveforms provide examples of {\em nonabelions}~\cite{MooRea}, which have been sought after in QFT model-building. They have been proposed to describe the phenomenology of correlated  electrons (for which $\phi$ would play the role of order parameter) in certain contexts of interest for
condensed-matter physics~\cite{NSSFD}. The results of
this paper demonstrate that some of the gauge theories we described above, with K\"ahler toric targets $X$, determine moduli spaces $ \mathcal{M}_{\mathbf{h}}^X(\Sigma) $
having {nonabelian} fundamental groups.
Perhaps surprisingly, this means that specific quantum {\em Abelian} supersymmetric gauged sigma-models (as we shall illustrate in Section~\ref{sec.examples}) are good candidates for effective  field theories whose quantum  states braid with {\em nonabelionic} statistics.

\section{A  fundamental-group interpretation of $\DivBraid_{\mathbf{k}}(\Sigma, \Gamma)$}
\label{sec:fundamental} 

In this section we provide a more rigorous definition of the divisor braid groups $\DivBraid_{\mathbf{k}}(\Sigma, \Gamma)$ described in our Introduction. We shall also justify the
isomorphism (\ref{DBisoFG}) that allows to interpret them as fundamental groups of certain generalised  configuration spaces that we denote as $\CF_{\mathbf{k}}(\Sigma, \Gamma)$. 

\subsection{Basic definitions}

For the rest of the paper, it will be more convenient to work with the following convention for the combinatorial data specifying intersections. 

\begin{defn}
  A {\em negative colour scheme} is a pair $(\Gamma,\mathbf{k})$ consisting of an undirected graph $\Gamma$ with no self-loops and no multiple edges, and a
  function $\mathbf{k}: {\rm Sk}^0 (\Gamma) \rightarrow \NN$ on the set of vertices (i.e. 0-skeleton) of $\Gamma$. The vertices  of the graph $\Gamma$ will sometimes be referred to as {\em colours},
 each image $k_\lambda := \kk(\lambda)$ as the  {\em degree} in colour $\lambda$,  and the map $\kk$ as the {\em coloured degree}.
\end{defn}
 
 Some of the main results of this paper concern coloured degrees $\kk$ for which $k_\lambda \ge 2$ for all colours $\lambda$; recall that we refer to such coloured degrees, or to any object defined from them, as {\em very composite} (see Definition~\ref{def.composite} for nomenclature concerning $\kk$). 
We will often simplify the notation by introducing a total order on the set of vertices, which will identify them with consecutive integers --- e.g. $1,\dots,r$. 
Accordingly, we can represent an effective coloured degree $\mathbf k$ by the tuple of degrees $k_\lambda$ in each colour, i.e. by an element of $\NN^r$, where $r$ is the number of colours. 

As anticipated by the notation, the configuration spaces we are interested in will be associated to a negative colour scheme  $(\Gamma,\mathbf{k})$ together with a smooth surface $\Sigma$. 
For each colour $\lambda$, we consider the symmetric product $S^{k_\lambda}\Sigma$. A point of this symmetric power can be considered as an effective divisor of degree $k_\lambda$ on $\Sigma$. We shall write
$$\Sigma^\kk := \prod_{\lambda \in {\rm Sk}^0 (\Gamma)} \Sigma^{k_\lambda }  \qquad \text{and} \qquad
S^\kk\Sigma := \prod_{\lambda\in {\rm Sk}^0 (\Gamma)} S^{k_\lambda}\Sigma.$$ 
It is evident that the symmetric group 
\begin{equation} \label{symmgrp}
\mathfrak{S}_\kk:=\prod_{\lambda \in {\rm Sk}^0 (\Gamma)} \mathfrak{S}_{k_\lambda}
\end{equation}
acts on $\Sigma^\kk$ with quotient $S^\kk\Sigma$. Once again we note that that this quotient is always smooth, even though the action of $\mathfrak{S}_\kk$ on $\Sigma^\kk$ is not free in general.
 
\begin{defn} \label{CF}
The space $\CF_\kk (\Sigma,\Gamma)$ of configurations of points on $\Sigma$ with negative colour scheme $(\Gamma,\kk)$ is the set
\[
\left\{
(d_1,\dots,d_r)\in S^\kk\Sigma\, \left| \, 
\begin{array}{c}
\text{If an edge connects the vertices $\lambda$ and $\lambda'$ in $\Gamma$,}\\
\text{ then the supports of $d_\lambda$ and $d_{\lambda'}$ in $\Sigma$ are disjoint.}
\end{array}
 \right. \right\}.  
\]  
\end{defn}

Referring back to our more general Definition~\ref{defconf}, we can alternatively write
\begin{equation}\label{ConfDiv}
{\sf Conf}_\kk (\Sigma,\Gamma) = {\sf Div}^\kk_+ (\Sigma,\neg \Gamma)
\end{equation}
as in  (\ref{DivConf}), recovering the negative graph $\neg \Gamma$ as the 1-dimensional simplicial complex of a (genuine) colour scheme in the sense of that definition;
this also makes contact with the established nomenclature for configuration spaces (see Remark~\ref{lovenothate}).

Recall that if $\Sigma$ is given the structure of a Riemann surface, then each $S^{k_\lambda} \Sigma$ receives an induced complex structure. Via the inclusion $\CF_\kk (\Sigma,\Gamma) \subset S^\kk \Sigma$ as a dense open set, the configuration space becomes a complex manifold; its complex dimension is the {\em total degree} $|\kk|$ already defined in equation (\ref{normk}),
$$
|\kk| := \sum_{\lambda \in {\rm Sk}^0 (\Gamma)} k_\lambda = \dim_\CC  S^\kk \Sigma = \dim_\CC \CF_\kk (\Sigma,\Gamma) .
$$

If the coloured degree $\kk$ is the function $\mathbf{1}$ with constant value 1 and $\Gamma$ is the complete graph with $r$ vertices, then 
the configuration space $\CF_{\mathbf{1}} (\Sigma,\Gamma)$  is the classical configuration space  of
$r$ ordered, distinct points on $\Sigma$.  It is well know that the fundamental group of $\CF_{\mathbf 1} (\Sigma,\Gamma)$  is the group of
pure braids ${\sf PB}_r(\Sigma)$ on $\Sigma$.  We will  focus next on how to handle the other end of the spectrum, i.e. when higher symmetric powers of the surface $\Sigma$ occur
as components of the ambient space $S^\kk\Sigma \supset \CF_{\mathbf{k}} (\Sigma,\Gamma)$. (These symmetric powers occur in {\em all} components when $\kk$ is very composite.)

\subsection{The monochromatic case} \label{sec:monochrom}

By the Dold--Thom theorem~\cite{DolTho}, if $k\geq 2$ then the fundamental group of the symmetric product $S^k\Sigma$ equals $H_1(\Sigma;\ZZ)$. As a warm-up to other constructions, we want to interpret this object
as a group of braids on $\Sigma$.
Since we know the group $H_1(\Sigma;\ZZ)$ perfectly well, this venture may sound like a pointless pursuit. However, another point of view on the following material, which might be more instructive, is that it concerns the obvious surjective map from the group of (not necessarily pure) braids on $\Sigma$ to the first homology group of $\Sigma$.
Our goal is to exhibit explicit elements of the braid group that normally generate the kernel of this map.

For convenience, we will assume that $\Sigma$ has a Riemannian metric. However, our constructions will not depend on this metric in any essential way. There is a projection map $p:\Sigma^k\to S^k\Sigma$ corresponding to the quotient by $\mathfrak{S}_k$, which in classical algebraic geometry is written as $(z_1,\ldots, z_k) \mapsto \sum_{i=1}^k z_i$. We fix a basepoint $z\in\Sigma$, and similarly basepoints $z^k=(z,z,\dots,z)\in \Sigma^k$ and $\mathbf{z}=p(z^k)= k\, z\in S^k \Sigma $.  We want to interpret the fundamental group of $\Sigma$ as a  generalised braid group.

We define the appropriate generalised braids in the following way. A braid consists of  a differentiable  map
\[
\gamma:[0,1]\to \Sigma^k
\]
with the property that for all $i\not=j$ and $t\in (0,1)$ we have that $\gamma_i(t)\not= \gamma_j(t)$. In particular, we cannot have that $\gamma(0)=z^k$ or $p(\gamma(0))=\zz$ for some $z\in \Sigma$. Instead, we consider a disc centred at $z$, pick arbitrarily $k$ distinct points $z_1, \dots,z_k$ inside this disc, and define $\zz^\prime$ to be the image  (i.e. equivalence class) $\sum_{i=1}^k z_i$ in $S^k\Sigma$ of $(z_1,\dots,z_k)\in \Sigma^k$ under $p$. Up to homotopy, there is a unique path from $z$ to each $z_i$ in the chosen disc, providing a canonical identification of the fundamental groups $\pi_1(S^k\Sigma)$ defined using those two base points. To simplify our notation, we will usually suppress the dependence of the fundamental group on the choice of $\zz^\prime$, and even the difference between $\zz$ and its `reduced' version $\zz^\prime$.

For braids $\gamma$ which satisfy $\gamma(0)=(z_1,\dots,z_k)$, we shall impose the restriction  $p\circ\gamma(1)=\zz^\prime$.  Because the $z_i$ are distinct points, there will be a well-defined permutation $\sigma \in \mathfrak{S}_k$ such that $\gamma(1)=(z_{\sigma(1)},\dots,z_{\sigma(k)})$. This permutation  will obviously depend on $\gamma$. There is also a well-defined way of composing two such braids $\gamma^{(1)}$ and $\gamma^{(2)}$ by glueing at the end points. Precisely, if $\sigma$ is the permutation determined by $\gamma^{(1)}$, then $\gamma^{(1)}*\gamma^{(2)}$  is the composition of the path $\gamma^{(1)}$ with the path $t\mapsto \left(\gamma^{(2)}_{\sigma(1)}(t),\dots,\gamma^{(2)}_{\sigma(k)}(t)\right)$.

Now the braid $\gamma$ determines an element $[p\circ \gamma]\in \pi_1(S^k\Sigma)$. We introduce an equivalence relation on such braids, generated by two types of relations.   First, if two braids are homotopic through braids, they are equivalent. Secondly, we express that we allow any strands of this braid to pass through each other.

Suppose that $\gamma_1,\gamma_2:[0,1]\to \Sigma$ are two continuous paths. Together, they represents a path $p\circ (\gamma_1,\gamma_2)$
in $S^2\Sigma$. Suppose that $\gamma_1(1/2)=\gamma_2(1/2)$. Then, we can can consider
\[
\gamma_1'(t)=
\begin{cases}
\gamma_1(t)& 0\leq t\leq 1/2\\
\gamma_2(t)& 1/2\leq t\leq 1\\
\end{cases} 
\qquad
\gamma_2'(t)=
\begin{cases}
\gamma_2(t)& 0\leq t\leq 1/2\\
\gamma_1(t)& 1/2\leq t\leq 1\\
\end{cases} 
\]
The path $p\circ (\gamma_1',\gamma_2')$ equals the path
$p\circ (\gamma_1,\gamma_2)$. This will produce a ``crossing relation'' in the fundamental group of $S^2\Sigma$ which we have to take into account.

Let $\gamma$ and $\tilde \gamma$ denote two braids in the sense considered before. Let $z\in \Sigma$ be a point, $\epsilon >0$ such that the exponential map associated to our fixed metric is a homeomorphism on a ball of radius $\epsilon$ in ${\rm T}_z(\Sigma)$. In particular, the metric ball $B_\epsilon(z):=\exp_z(B_\epsilon(0))$ is a disc. Let $\tau>0$ and $0< t_0< 1$. Now assume that $i_1\not= i_2$ are such that
\begin{itemize}
\item $\gamma_i(t)=\tilde\gamma_i(t)$ unless $i\in \{i_1,i_2\}$.
\item If $t< t_0-\tau$, then $\gamma_{i_1}(t)=\tilde \gamma_{i_1}(t)$ and  $\gamma_{i_2}(t)=\tilde \gamma_{i_2}(t)$.
\item  If $t> t_0+\tau$, then $\gamma_{i_1}(t)=\tilde \gamma_{i_2}(t)$ and  $\gamma_{i_2}(t)=\tilde \gamma_{i_1}(t)$.
\item If $\vert t-t_0\vert\leq \tau$ and $i\not =i_1, i\not= i_2$, then $\gamma_i(t)\not \in B_\epsilon(z)$. 
\item If $\vert t-t_0\vert\leq \tau$, and either  $i=i_1$ or $i=i_2$, then $\gamma_i(t)\in B_\epsilon(z)$.
\end{itemize} 
If these conditions are satisfied for some choice of $z,\epsilon,i_1,i_2$, we say that $\gamma$ and $\tilde\gamma$ are related by a {\em crossing move}.

 It is quite easy to check that the equivalence relation generated by homotopies and crossing moves is compatible with composition of braids, so we also get a composition on equivalence classes. The constant braid $\gamma_i(t)=z_i$ gives an identity element for this composition, and the reverse of braids (in the usual sense) is its inverse. It follows that  the equivalence classes form a group.

\begin{defn}
The group of monocromatic divisor braids $\DivBraid_k(\Sigma)$ is the group of equivalence classes of braids modulo homotopies and crossing moves.
\end{defn}

Our goal is to relate this group with the first homology of the surface $\Sigma$, which is the fundamental group of $S^k\Sigma$.

\begin{lemma}
\label{maptofundamental1}
If $\gamma$ and $\tilde\gamma$ are braids in the same equivalence class, then 
$[p\circ\gamma]=[p\circ \tilde\gamma]\in \pi_1(S^k\Sigma)$. 
\begin{proof}
If $\gamma$ and $\tilde\gamma$ are homotopic, the assertion is immediate. Now suppose they are related by a crossing move. We indicate how to construct a homotopy from $p\circ \gamma$ to $p\circ \tilde \gamma$. This homotopy will be a composition of two homotopies. 
The first homotopy only changes the strands $\gamma_{i_1}(t)$ and $\gamma_{i_2}(t)$, and only for $\vert t-t_0\vert<\tau$. It moves  
$\gamma$ to a path $\gamma^{(1)}$ which satifies that $\gamma_{i_1}^{(1)}(t_0)=\gamma_{i_2}^{(1)}(t_0)$. This path is not a braid (since two strands intersect), but $p\circ \gamma^{(1)}$ still defines an element of 
$\pi_1(S^k\Sigma)$, and actually 
the same element as $p\circ \gamma$. Now consider the path $\gamma^{(2)}:[0,1]\to \Sigma^k$ defined by
\[
\gamma^{(2)}_i(t)=
\begin{cases}
\gamma^{(1)}_i(t)&\text{ if $i\not\in \{i_1,i_2\}$,}\\
\gamma^{(1)}_i(t)&\text{ if $i\in \{i_1,i_2\}$, and $t\leq t_0$.}\\
\gamma^{(1)}_{i_2}(t)&\text{ if $i=i_1$ and $t\geq t_0$.}\\
\gamma^{(1)}_{i_1}(t)&\text{ if $i=i_2$ and $t\geq t_0$.}\\
\end{cases}
\]
We see that $p \circ \gamma^{(2)}=p \circ \gamma^{(1)}$. But it is also easy to see that we have a second homotopy, similar to the first one,  from
$\gamma^{(2)}:[0,1]\to\Sigma^k$ to $\tilde\gamma:[0,1]\to\Sigma^k$. We have now proved the lemma, since $p\circ \gamma^{(2)}$ defines the same element in $\pi_1(\Sigma^k)$ as $p\circ \tilde\gamma$. 
\end{proof}
\end{lemma}

Lemma~\ref{maptofundamental1} provides us with a group homomorphism
\begin{equation}\label{xi}
\xi:\DivBraid_k(\Sigma)\to \pi_1(S^k\Sigma).
\end{equation}
\begin{lemma}
\label{lambdaiso1}
$\xi$ is an isomorphism.
\begin{proof}
Recall that from partitions of $k$ one obtains a filtration of $S^k\Sigma$ after how many of the components of each pre-image $k$-tuple $(z_1,\dots,z_k)$ are distinct. That is,  we take
$$F^i:=\{[(z_1,\dots ,z_k)]\vert \text{ at most $i$ of the points  $z_j$ are distinct}\}.$$ In the filtration
\[
\emptyset = F^0 \subset \Sigma \cong F^1\subset F^2\subset \dots\subset  F^k=S^k\Sigma
\]  
the (real) codimension of the space $F^i$ is $2(k-i)$. In general, the spaces $F^i$ have singularities.  However, each $F^i$ is closed in $F^{i+1}$, so the complements $U^i:=S^k\Sigma \setminus F^i$ form a filtration of open subspaces
\[
\emptyset=U^k  \subset U^{k-1} \subset \dots U^0  = S^k\Sigma .
\]
The set $U^{k-1}$ is the configuration space of $k$ distinct unordered  points on $\Sigma$. 

\emph{Claim 1:}  The difference $U^{i-1}\setminus U^{i}=F^i\setminus F^{i-1}$ is a closed submanifold of $U^{i-1}$ of dimension $2i$. 

Closure is immediate: since $F^i$ is closed in $S^k\Sigma$, we have that $F^i \setminus F^{i-1}$ is closed in $S^k\Sigma \setminus F^{i-1}$. We need to check that it is a submanifold. This question is local, so we can as well assume that $\Sigma=\CC$. 

We can identify $\CC^k/\mathfrak{S}_k$ with $\CC^k$ through the map $\sum_{i=1}^k z_i \mapsto  (s_1,\dots,s_k)$, where the $s_i$ denotes the $i$-th elementary symmetric function in $z_1,\dots, z_k$. The image in $\CC^k/\mathfrak{S}_k$ of the (total) diagonal subspace of $\CC^k$ is a copy of $\CC$, and its inclusion followed by the identification is the map $\phi_k(z)= \left(\binom k1 z,\binom k2 z^2,\dots, \binom kk z^k\right)$. This is clearly the inclusion of a submanifold.

Now suppose $\mathbf{z}\in  F^i\setminus F^{i-1}$ for some $i$. That is, 
\[
\mathbf{z} =p(\underbrace{z_1,\dots,z_1}_{\gets k_1 \to},\underbrace{z_2,\dots,z_2}_{\gets k_2 \to},\dots,\underbrace{z_i,\dots,z_i}_{\gets k_i \to}),
\] 
where the points $z_1, z_2, \ldots, z_i$ are pairwise distinct.
We get a chart for $S^k\Sigma$ in a neighbourhood of $\mathbf{z}$ by using the elementary symmetric functions on the first $k_1$ variables, then the elementary symmetric functions on the next $k_2$ variables, and so on. In a neighbourhood of $\mathbf{z}$, the subset $F^i$ is given by $i$ points, and the inclusion is locally given by $(\phi_{k_1},\phi_{k_2},\dots,\phi_{k_i})$. This map is a product of $i$ embeddings of $\CC$, so it is the inclusion of a submanifold. The claim is proved.
 
\emph{Claim 2:} $\xi$ is surjective.

By transversality, a map $S^1\to U^0= S^k\Sigma$ is homotopic, by an arbitrarily small homotopy, to a map whose image avoids $F^{0}$. Using transversality inductively we can avoid all $F^i$ for $i<k$. It follows that there is an arbitrarily small homotopy that moves $\gamma$ to a map whose image is in $U^{k-1}$. This map can obviously be lifted to a braid $[0,1]\to \Sigma^k$, which proves Claim 2. Actually, this lift is unique and compatible with composition. 

\emph{Claim 3:} $\xi$ is injective.

Assume that  $\gamma^{(0)}$ and  $\gamma^{(1)}$ are two braids and $\Phi:[0,1]\times[0,1]\to S^k\Sigma$ is a homotopy from $p\circ\gamma^{(0)}$ to $p\circ \gamma^{(1)}$. 
We can make the homotopy transversal to the filtration $F^*$. In particular, we can assume without restriction that its image does not meet $F^{k-2}$, and that it will meet $F^{k-1}$ transverally in finitely many points at times $t_1,\dots , t_n $. At each $t_i$ the homotopy will perform a crossing move, so that 
we can transform $\gamma^{(0)}$ to $\gamma^{(1)}$ by a finite number of homotopies and a finite number of crossing moves.
\end{proof}
\end{lemma}

\subsection{Braids of many colours} \label{sec:variegated}

Suppose that $\Sigma$ is an oriented surface and $(\Gamma, \kk)$ is a given negative colour scheme in $r>1$ colours. To render the notation somewhat less cumbersome, we will write $k:= |\kk|$ for the total degree.

Observe that the configuration space $\CF_\kk (\Sigma,\Gamma)$ of Definition~\ref{CF} can be considered as a quotient of an open subset $U \subset \Sigma^\kk$. 
There is a restricted action of $ \mathfrak{S}_{\mathbf k}$ on this $U$, and the space we are interested in is the quotient 
$${\CF}_\kk (\Sigma,\Gamma):=U /\mathfrak{S}_{\mathbf k} \subset S^\kk \Sigma. $$ 
Let $p: U  \to S^\kk \Sigma$ once again denote the quotient map. 

Let $z_1,\dots,z_r$ be arbitrarily chosen, disjoint points in $\Sigma$. These points determine a point $(z_1^{k_1},\dots, z_r^{k_r}) \in U$, where as in Section~\ref{sec:monochrom} we are writing
$z_\lambda^{k_\lambda} := (z_\lambda,\cdots, z_\lambda) \in \Sigma^{k_\lambda}$.
 We choose the image of this point under the quotient map $p$ as a base point in $S^\kk \Sigma$. 

For $1 \leq \lambda \leq r$, we chose $r$ disjoint metric discs $D_\lambda \subset \Sigma $ where $z_\lambda \in D_\lambda$. Inside each disc $D_\lambda$, we chose
distinct points $z_{\lambda,j}$ where $1\leq j\leq k_\lambda$. In particular, the points $z_{\lambda,j}$ (when all the labels are taken) are distinct. Let $Z_\lambda:=\{z_{\lambda,j}\in \Sigma \, \vert\,  1\leq j\leq k_\lambda\}$ and $Z:=\bigcup_{\lambda=1}^r Z_\lambda$. 

We first consider the full set ${\BraidSet}_{\kk}$ of {\em colour-pure} braids on $k$ strands in $\Sigma$, starting and ending at $Z$.
Precisely, an element $\gamma$ of this set consists of a family of smooth maps $\gamma_{\lambda,j}:[0,1]\to \Sigma$ satisfying the following conditions:
\begin{itemize}
\item If  $(\lambda,j) \not=(\lambda^\prime,j^\prime)$, then $\gamma_{\lambda,j}(t)\not=\gamma_{\lambda^\prime,j^\prime}(t)$  for all $t\in[0,1]$.
\item $\gamma_{\lambda,j}(0)\in Z_\lambda$ and $\gamma_{\lambda,j}(1)\in Z_\lambda$.
\end{itemize}

We can think of such a braid as built out of $k$ strands $\gamma_{\lambda,j}$, such that every strand is coloured by the first index $\lambda$. That is, we say
that $\gamma_{\lambda,j'}$ has the same colour as $\gamma_{\lambda,j''}$ for any $j',j''$, and the collection $$\gamma_\lambda:=\left\{ \gamma_{\lambda,j}: [0,1] \rightarrow \Sigma\right\}_{j=1}^{k_\lambda}$$ forms the
sub-braid of colour $\lambda$ in $\gamma$.

Now we introduce an equivalence relation on this set. As in the monochromatic case, the equivalence we are interested in is generated by two types of relations. The first type is that colour-pure braids which are homotopic through braids satisfying the same conditions are equivalent. The second is that 
two braids are equivalent if they are related by a crossing move involving two strands $\gamma_{\lambda,j}$ and $\gamma_{\lambda',j^\prime}$ that are either: (i) of the same colour (i.e. $\lambda=\lambda'$), or (ii) of colours $\lambda,\lambda'$ corresponding to vertices in $\Gamma$ that are {\em not} connected by an edge of the graph.

We can generalise the definition of divisor braids to the many-colour case as follows.

\begin{defn}[Divisor braid groups] \label{DB}
A {\em divisor braid} of negative colour scheme $(\Gamma, \kk)$  is an equivalence class of elements in $\BraidSet_\kk$ under homotopy of colour-pure braids and the crossing moves of type (i) and (ii) described above. The set of equivalence classes of such divisor braids is denoted $\DivBraid_\kk(\Sigma,\Gamma)$. One can compose divisor braids by concatenation as usual, and it is easy to see that they form a group. 
\end{defn}

As in the case of one colour, a braid $\gamma_{\lambda,j}$ determines a closed path $p\circ \gamma_{\lambda,j}(t)$, which can be interpreted as an element 
$\xi(\gamma)\in \pi_1(S^\kk \Sigma)$
for the obvious generalisation 
\begin{equation} \label{DBaspi1}
\xi: \DivBraid_\kk(\Sigma,\Gamma) \rightarrow \pi_1 (\CF_\kk(\Sigma,\Gamma))
\end{equation}
of the group homomorphism (\ref{xi}).
\begin{lemma}
\label{maptofundamental2}
If $\gamma^{(1)}$ and $\gamma^{(2)}$ are equivalent, then $\xi(\gamma^{(1)})=\xi(\gamma^{(2)})$.
\begin{proof}
If the braids are homotopic, they clearly define the same element. If the braids are related by a crossing move involving braids of the same colour, they determine the same element by the same argument as in the proof of Lemma~\ref{maptofundamental1}.
\end{proof}
\end{lemma}

 As in the monochromatic case, we have the following result:
\begin{proposition} 
\label{lambdaiso2}
The map (\ref{DBaspi1})   is a well-defined group isomorphism.
\begin{proof}
We follow the proof of Lemma~\ref{lambdaiso1}.
The main ingredient in the proof is a filtration of $S^\kk \Sigma$. We partially order the set $(l_1,\dots,l_r)$ of $r$-tuples of nonnegative numbers by 
$(l_1,\dots,l_r)\leq (l_1^\prime,\dots,l_r^\prime)$ if $l_\lambda \leq l^\prime_\lambda $ for all $\lambda$. For $\ll\leq \kk$ we define
$F^\ll\subset S^\kk(\Sigma)$ as the set of points $(z_1,\dots,z_r)$ such that $z_\lambda \in S^{k_\lambda }(\sigma)$ has at most $l_\lambda$ distinct elements. If $\ll^\prime \leq \ll$, clearly $F^{\ll^\prime}\subset F^{\ll}$ is a closed subset. Now consider
\[
F^{<\ll}:=\bigcup_{\ll^\prime<  \ll}F^{\ll^\prime}\subset F^\ll.
\] 
This is a closed subset. The complement $U^\ll=F^\kk\setminus F^{<\ll}$ is an open set in the $2k$- dimensional manifold $F^\kk$, and $V^\ll=F^\ll\setminus F^{<\ll}$ is an open set in $F^\ll$. 

\emph{Claim 1:} $V^\ll$ is a submanifold of $U^\ll$ of (real) dimension  $2(\sum_{\lambda=1}^r l_\lambda)$. 

There are two conditions to check. First, we have to check that $V^\ll$ is closed in $U^\ll$.  But since $F^\ll$ is closed in $F^\kk$, it follows immediately that $V^\ll = F^\kk\cap U^\ll$ is closed in $U^\ll$. Next, for each $y \in V^\ll$ we need to find a submanifold chart on a neighborhood of $y$ in $U^\ll$. 
We then construct a chart around the image $\mathbf{y}=p(y)=(\mathbf{y}_1,\ldots,\mathbf{y}_r)$ proceeding around each component $\mathbf{y}_\lambda$ exactly  as in the proof of Lemma~\ref{lambdaiso1}. 

\emph{Claim 2:} $\xi$ is an isomorphism.
  
If $\gamma:[0,1]\to \CF_\kk(\Sigma,\Gamma)$, we can do an arbitrarily small pertubation of $\gamma$ to a map whose image is disjoint from $F^{<\kk}$. But this means that the element it represents in $\pi_1(\CF_\kk(\Sigma, \Gamma))$ is in the image of $\xi$. Similarly, if $F$ is a homotopy between two braids $\gamma^{(1)}$ and $\gamma^{(2)}$, we can make the homotopy transversal to all $F^\ll\setminus F^{<\ll}$. By counting dimensions, we see that this means (using transversality) that  the image of the homotopy is disjoint from $F^\ll$ unless $l_\lambda=k_\lambda$ except for $\lambda=\lambda_0$, and $l_{\lambda_0}=k_{\lambda_0}-1$. It follows that the homotopy intersects $F^{<\kk}$ in isolated points. Arguing as in the monochromatic case, we get that the braids $\gamma^{(1)}$ and $\gamma^{(2)}$ are equivalent. 
\end{proof}
\end {proposition}

\begin{remark}
\label{Hurewicz}
A divisor braid gives a 1-cycle in $\Sigma$ for each colour $\lambda$, namely the sum of the $k_\lambda$  trajectories on $\Sigma$ described by the basepoints of the $\lambda$-coloured strands. If there are $r$ colours, it is easy to see that this defines a group homomorphism
\begin{equation}\label{Hurew}
h: \DivBraid_\kk(\Sigma,\Gamma) \to H_1(\Sigma;\ZZ)^{\oplus r}.
\end{equation}
In the case of only one colour ($r=1$), this map can be identified with the Hurewicz map $\pi_1(\Sigma) \to \pi_1(S^k\Sigma)\cong H_1(\Sigma;\ZZ)$. 
\end{remark}

\section{Some computations using pictures}
\label{sec:pictures}

The group of braids ${\DivBraid_\kk(\Sigma,\Gamma)}$  defined in the last section is a quotient group of a certain subgroup of the full braid group ${\sf B}_k(\Sigma)$  on $k=|\kk | := \sum_{\lambda \in {\rm Sk}^0(\Gamma)} k_\lambda$ strands on $\Sigma$. Let $\sigma : {\sf B}_k(\Sigma) \to \mathfrak{S}_k$ be the usual map from the full braid group to the symmetric group, given by picking an ordering of the set $Z=\{z_{\lambda,j}\}_{\lambda,j=1}^{r,k_\lambda}$ and extracting from every braid (not necessarily colour-pure) the associate bijection mapping the starting points to the endpoints of its strands.  The Cartesian product $\mathfrak{S}_{\kk}$ in (\ref{symmgrp}) corresponds to the subgroup of permutations that leaves each of the $r$ subsets $Z_\lambda=\{z_{\lambda,j}\ \vert\,  1\leq j\leq k_\lambda\}$ invariant.

The subgroup ${\sf B}_\kk (\Sigma)= \sigma^{-1}(\mathfrak{S}_\kk)\subset {\sf B}_{|\kk|}(\Sigma)$ is the group of braids coloured by the $r$ colours subject to the relations of homotopy. According to Definition~\ref{DB}, we obtain the group of divisor braids $\DivBraid_\kk(\Sigma,\Gamma)$ from this group by introducing the additional relations generated by the crossing moves of both types (i) and (ii). 

Let ${\sf PB}_{k}(\Sigma)$ be the usual group of pure braids on $k$ strands on $\Sigma$.
There is a diagram whose top row is exact, and where the vertical map is surjective:
\[
\begin{CD}
1@>>>{\sf PB}_{|\kk|}(\Sigma) @>>> {\sf B}_{\kk}(\Sigma) @>>> \mathfrak{S}_{\kk}@>>>1\\
@.@. @VVV  @. \\
@.@. \DivBraid_\kk(\Sigma,\Gamma) @.
\end{CD}
\]  
\begin{lemma}
\label{surjectivity}
The composite map $$\Lambda: {\sf PB}_{|\kk|} (\Sigma)\to {\sf B}_{\kk}(\Sigma) \to \DivBraid_\kk(\Sigma,\Gamma)$$ is surjective.
\begin{proof}
By the above diagram, it is sufficient to find a set of elements $\alpha_\lambda$ in the kernel of the map ${\sf B}_\kk (\Sigma) \to \DivBraid_\kk (\Sigma,\Gamma)$ such that their images
 in $\mathfrak{S}_\kk$ generate this group. Now, let $\alpha$ be related to the trivial element by a single crossing move,  involving the strands starting at $z_{\lambda,j}$ and $z_{\lambda,j^\prime}$. By the definition of the equivalence relation, these elements are in the kernel of the map ${\sf B}_\kk (\Sigma) \to \DivBraid_\kk(\Sigma,\Gamma)$. On the other hand, its image in the symmetric group is the transposition of $z_{\lambda,j}$ and $z_{\lambda,j^\prime}$, so the family of such elements generate $\mathfrak{S}_\kk$. This proves the lemma.
\end{proof}
\end{lemma}

It follows from Lemma~\ref{surjectivity} that a generating set for the pure braid group will provide us with a set of generators for $\DivBraid_\kk(\Sigma,\Gamma)$. We intend to argue geometrically, so we want to describe the generators by suitable pictures. 

For the rest of this subsection, we make two simplifying assumptions:
\begin{itemize}
\item $\Sigma$ is compact and  orientable; we shall denote its genus by $g$;
\item  $\kk$ is very composite (i.e. $k_\lambda \geq 2$ for every colour $\lambda$).
\end{itemize}

There are several presentations of the pure braid group of an orientable surface available in the literature. 
For higher genus surfaces, the first one seems to be given in~\cite{BirOBG}. Actually, we are only interested in specifying a set of generators.
We are going to describe a presentation given in reference~\cite{ GM}. Represent an orientable surface of genus $g$ as a regular $4g$-gon, where opposite sides have been identified respecting the orientation. Choose an auxiliary oriented line $L$ passing through two opposite vertices of this $4g$-gon, and assume that $z_1,\dots, z_k$ are points lying on this line, such that the ordering of the labels of the points respects the the linear order on the line.   

We will be interested in a particular class of pure braids. Consider a closed path $\gamma_i:[0,1]\to \Sigma$ such that $\gamma_i(0)=\gamma_i(1)=z_i$,  and furthermore such that $\gamma_i(t)\not=z_j$ for all $t\in [0,1]$ and $i\not=j$. We define a braid $\Phi(\gamma_i)$ by the paths in $\Sigma^k$ with components
\[
\Phi(\gamma_i)_j(t)=
\begin{cases}
z_j&\text{ if $j\not=i$,}\\
\gamma_i(t)&\text{ if $j=i$,}
\end{cases}, \qquad(j=1,\ldots, k).
\]
We call an object of this type a {\em monic braid}. We shall give generators for the divisor braid groups which are monic braids.

We define $a_{i,\ell}$ for $1\leq i\leq k$ and $1\leq \ell \leq 2g$ to be the monic braid  given by letting the path $\gamma_i(t)$ run along a homotopic deformation of a straight line from $\gamma_i(0)=z_i$ to the middle point of an edge labelled $\ell$ in the $4g$-gon (Figure~1). After that, it runs from the middle of the opposite edge back to to $z_i$.

\begin{figure}
\label{fig.a}
\vspace{5mm}
\includegraphics[width=5cm,angle=0]{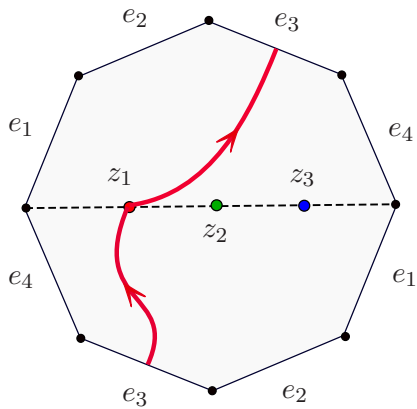} \\[-0pt] 
\vspace{-40mm}\hspace{-50mm}$e_1$\\
\vspace{-19mm}\hspace{-20mm}$e_2$\\
\vspace{-4mm}\hspace{20mm}$e_3$\\
\vspace{10mm}\hspace{50mm}$e_4$\\
\vspace{15mm}\hspace{-50mm}$e_4$\\
\vspace{11mm}\hspace{-20mm}$e_3$\\
\vspace{-5mm}\hspace{22mm}$e_2$\\
\vspace{-20mm}\hspace{51mm}$e_1$\\
\vspace{-18mm}\hspace{-24mm}$z_1$ \\
\vspace{3mm}\hspace{1mm}$z_2$ \\
\vspace{-12mm}\hspace{24mm}$z_3$ \\
\vspace{32mm}

\caption{The path $a_{1,3}$.}
\end{figure}

\begin{figure}
\label{fig.t}
\vspace{5mm}
\includegraphics[width=5cm,angle=0]{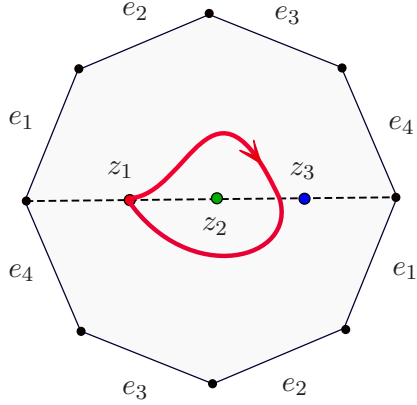} \\[-0pt] 
\vspace{-40mm}\hspace{-50mm}$e_1$\\
\vspace{-19mm}\hspace{-20mm}$e_2$\\
\vspace{-4mm}\hspace{20mm}$e_3$\\
\vspace{10mm}\hspace{50mm}$e_4$\\
\vspace{15mm}\hspace{-50mm}$e_4$\\
\vspace{11mm}\hspace{-20mm}$e_3$\\
\vspace{-5mm}\hspace{22mm}$e_2$\\
\vspace{-20mm}\hspace{51mm}$e_1$\\
\vspace{-18mm}\hspace{-24mm}$z_1$ \\
\vspace{3mm}\hspace{1mm}$z_2$ \\
\vspace{-12mm}\hspace{24mm}$z_3$ \\
\vspace{32mm}

\caption{The path $t_{1,2}$.}
\end{figure}
Further, for each  pair $(i,j)$ with $1\leq i < j \leq k$, we choose a path $t_{i,j}$ starting at $z_i$, staying in the interior of the $4g$-gon, surrounding each of the the points $z_{i+1},z_{j}$ and intersecting $L$ in exactly one point besides $z_i$. Call the corresponding monic braid $t_{i,j}$ (Figure~2). 
It is useful for us to look for a slightly different set of generators for the pure surface braid group. For
$1\leq i ,j \leq k$, let us consider the path that starts at $z_i$, goes through the upper half of the $4g$-gon to a neighbourhood of $z_j$, runs once around $z_j$ in the positive sense, and then follows the same path back to $z_i$. We denote the corresponding monic braid by $b_{i,j}$. 

\begin{thm}
\label{PureBraidGroup}
Let $\Sigma$ be a compact oriented surface. If the genus  $g\geq 1$, the pure braid group is
generated by the classes $a_{i,j}$ and $b_{i,j}$. If $\Sigma=S^2$, the pure braid group is generated
by the classes $t_{i,j}$.   
\begin{proof}

First consider the case $g\geq 1$.
 According to \cite[\S 4]{GM},
the homotopy equivalence classes of the  $a_{i,\ell}$ and $t_{i,j}$ thus defined generate the pure braid group  ${\sf PB}_k(\Sigma)$. The list of relations presented there is rather long and complicated, but we will not need the relations.
Clearly, we can express $t_{i,j}$ as a product of the elements $b_{i,j}$ with $i<j\leq k$, and it follows that the classes of $a_{i,\ell}$ and $b_{i,j}$ must also generate the pure surface braid group.   
In the case $\Sigma \cong S^2$, the map induced by inclusion from the pure braid group of the plane to the pure braid group of the
sphere is surjective. 
The pure braid group of the plane is generated by braids $A_{i,j}$ (see~\cite{Artin}) which maps to the braids $t_{i,j}$ in the 
braid group of $S^2$. 
\end{proof} 
\end{thm}

It is a consequence of Theorem~\ref{PureBraidGroup} and Lemma~\ref{surjectivity} that the images  
$\Lambda(a_{i,\ell})$ and $\Lambda(b_{i,j})$ (respectively $\Lambda(t_{i,j})$ for $\Sigma=S^2$) generate $\DivBraid_\kk(\Sigma,\Gamma)$.
We want to produce relations between these elements in $\DivBraid_\kk(\Sigma,\Gamma)$. 
We start with two general lemmas that ensure that certain elements of $\DivBraid_\kk (\Sigma,\Gamma)$ commute with each other.

\begin{lemma}
\label{moveA} 
Let $\gamma$ and $\gamma^\prime$ be two paths in $\Sigma$ such that $\gamma(0)=\gamma(1)=z_i$ and  $\gamma'(0)=\gamma'(1)=z_{i^\prime}$. Further, assume that the images of the two paths in $\Sigma$ that do not intersect, that is, for all $t,t^\prime$ we have that $\gamma(t)\not=\gamma^\prime(t^\prime)$.  Then the divisor braids $\Phi(\gamma)$ and $\Phi(\gamma^\prime)$ commute.
\begin{proof}
The $k$ paths in the concatenation $\Phi(\gamma)* \Phi(\gamma^\prime)$ only have two nonconstant strands, namely 
\begin{align*}
(\Phi(\gamma) * \Phi(\gamma^\prime))_i(t)&=
\begin{cases}
\gamma(2t) &\text{ if $t\leq 1/2$,}\\
z_i &\text{ if $t\geq 1/2$.}
\end{cases}\\
(\Phi(\gamma') * \Phi(\gamma))_{i^\prime}(t)&=
\begin{cases}
z_{i^\prime} &\text{ if $t\leq 1/2$,}\\
\gamma^\prime(2t-1) &\text{ if $t\geq 1/2$.}
\end{cases}\\
\end{align*} 
We define a homotopy $F_s$ between these divisor braids.
For $0\leq s\leq 1$  we define each strand of $F_s$ to be constant, except for  
 \begin{align*}
(F_s)_i(t)&=
\begin{cases}
\gamma(2t-s) &\text{ if $s/2 \leq t\leq (s+1)/2$,}\\
z_i &\text{ if either $s/2 \geq t$ or $t \geq (s+1)/2$.}
\end{cases}\\
(F_s)_{i^\prime}(t)&=
\begin{cases}
z_{i^\prime} &\text{ if either $t \geq 1-s/2$ or $t\leq (1-s)/2$,}\\
\gamma^\prime(2t+s-1) &\text{ if $1-s/2\geq t\geq (1-s)/2$.}
\end{cases}\\
\end{align*}
This is a homotopy from $F_0=\Phi(\gamma) * \Phi(\gamma^\prime)$ to
$F_0=\Phi(\gamma') * \Phi(\gamma)$.
\end{proof}
\end{lemma}

\begin{lemma}
\label{commutingbraids}
 Let $\gamma$ and $\gamma^\prime$ be paths in $\Sigma$ such that $\gamma(0)=\gamma(1)=z_i$ and $\gamma^\prime(0)=\gamma^\prime(1)=z_{i^\prime}$. Suppose that the  that $z_i$ and $z_{i^\prime}$ have either the same colour or colours corresponding to vertices in $\Gamma$ which are not connected by an edge. Then the monic braids $\Phi(\gamma)$ and $\Phi(\gamma^\prime)$ commute.
\begin{proof}
As in the proof of the previous lemma, we construct a homotopy from $\Phi(\gamma)*\Phi(\gamma^\prime)$ to  $\Phi(\gamma^\prime)*\Phi(\gamma)$ consisting of a family of maps $F_s:[0,1]\to \Sigma^k$. In this case we do not know that $F_s(t)_i\not=F_s(t)_{i^\prime}$. But we can make $F_s$ transversal to the fat diagonal in $\Sigma^k$. This can be done by an arbitraily small pertubation, constant at the ends of the homotopy. We obtain a new homotopy $\tilde F_s$. Except for finitely many values of $s$, this is indeed a braid. When $s$ passes through one of the exceptional points, the braid changes by a crossing move. This proves the lemma.
\end{proof}
\end{lemma}

\begin{corollary}
\label{homologyfactorization}
Let $Z_\lambda':=\bigcup_{\lambda' \not = \lambda}Z_{\lambda'}$ be the union of all the basepoints that do not have colour $\lambda$.
Let $\gamma$ and $\gamma^\prime$ be loops of the same colour $\lambda$ in $\Sigma\setminus Z_\lambda'$. 
Suppose that they determine the same homology class in $H_1(\Sigma\setminus Z_\lambda';\ZZ)$. Then $\Phi(\gamma)=\Phi(\gamma^\prime)$.
\begin{proof}
Since $H_1(\Sigma;\ZZ)$ is the Abelianization of $\pi_1(\Sigma)$, it suffices to show that the monic braid constructed from a commutator vanishes. But this follows immediately from Lemma~\ref{moveA}.  
\end{proof} 
\end{corollary}

A consequence of Corollary~\ref{homologyfactorization} is that the classes $\Lambda(a_{i,\ell})$ and $\Lambda(b_{i,j})$ only depend on the colour of their basepoints. For each colour $\lambda$, we choose one of the basepoints of this colour. Let us say that it is indexed by the subscript $i_\lambda$.  For each number $\ell$, $1\leq \ell \leq 2g$ and colours $\lambda, \mu$,  we define 
\begin{equation}\label{alphabeta}
\alpha_{\lambda,\ell} := \Lambda(a_{{i_\lambda},\ell}) \qquad \text{and} \qquad \beta_{\lambda,\mu}:=\Lambda(b_{i_\lambda,i_\mu}).
\end{equation}

\begin{lemma} 
\label{generating.set}
The classes $\alpha_{\lambda,\ell}$ and $\beta_{\lambda,\mu}$ generate $\DivBraid_\kk (\Sigma,\Gamma)$.
The classes $\beta_{\lambda,\mu}$ satisfy that $\beta_{\lambda,\lambda}=e$ and $\beta_{\lambda,\lambda^\prime}=\beta_{\lambda^\prime,\lambda}$.

\begin{proof}
Since the collection of all classes $a_{i,\ell}$ and $b_{i,j}$ generate the pure braid group ${\sf PB}_k(\Sigma)$, it follows from Lemma~\ref{surjectivity} that the classes $\Lambda(a_{i,\ell})$ and
$\Lambda(b_{i,j})$ generate $\DivBraid_\kk(\Sigma,\Gamma)$.  It follows from Corollary~\ref{homologyfactorization} that the classes $\Lambda(a_{i,\ell})$ only depend on the colours of the points $z_i$ and the edge $e_\ell$, and similarly that the classes $\Lambda(b_{i,j})$ only depend on the colour of the points $z_i$ and $z_j$. So
$\DivBraid_\kk (\Sigma,\Gamma)$ is generated by the classes $\alpha_{\lambda,\ell}$ and $\beta_{\lambda,\mu}$.
The relations asserted for the classes $\beta_{\lambda,\mu}$ follow from elementary manipulation.
\end{proof}
\end{lemma}

In order to manipulate divisor braids, it is very convenient to have a graphical representation of them, as well as a repertoire of moves that do not change the equivalence class. This will be our next task.

A divisor braid can be considered as a family 
of coloured paths in $\Sigma \times [0,1]$, satisfying the additional condition that paths of colours associated to vertices connected by an edge in the graph $\Gamma$  never intersect. We choose to  draw the images of the braids in $\Sigma$. In this picture, there will quite likely be some points where the images of strands of different colours intersect. Accordingly, we will adhere to the convention that if $\gamma(t)=\gamma^\prime(t^\prime)\in\Sigma$ and $t<t^\prime$, then we will draw $\gamma$ as undercrossing and $\gamma^\prime$ as overcrossing. This is quite similar to the ordinary representation of braids through a picture that is supposed to be three-dimensional but is drawn on a plane. The difference is that, in those pictures, you are usually projecting the surface $D^2$ onto an interval, and preserving the time coordinate. We are now projecting away from the time coordinate. We can assume that this is done in such a way that the projection is a number of immersed oriented curves in the plane with a finite number of intersections. We call such projections allowable. To improve legibility, we also allow that the constant curve is to be represented by a single point in the projection. Note that such a representation completely determines the braid up to homotopy, and that a homotopy of projections can be lifted to a homotopy of braids. However, we might have homotopies of braids which do not correspond to a homotopy of allowable projections.

In Figure~3, we sketch how to gain a better understanding of the homotopy equivalences involved by comparing the projections of  the same divisor braids in a disc $D^2$ in two different directions. The lower row depicts the projections we discussed above, the upper row corresponds to a projection $D^2\to D^1$. Of course, we are just making use of the technique of drawing three-dimensional objects that came to be called ``descriptive geometry''.
 
The most fundamental braid is the one represented by one point circling another one in a small disk.
In Figure~3, this corresponds to the lower middle picture. Let us call this picture $\mathcal X$. 
There are certain obvious variations of this picture. We can perform the operation of reversing the orientation of the
blue curve. This yields a new picture that we will call $B\mathcal{X}$. We can also reverse the orientation of the red curve, which gives us the
two new pictures that we denote by $R\mathcal{X}$ and  $RB\mathcal{X}$. Finally, we can change the overcrossing to an undercrossing, and the
undercrossing to an overcrossing. This gives us four new pictures, which we shall refer to as $C\mathcal{X},CB\mathcal{X},CR\mathcal{X}$ and $CRB\mathcal{X}$. 

\begin{lemma}
\label{illusion}
The pictures $C\mathcal{X}$ and $RB\mathcal{X}$ are not projections of braids. 
 The braids corresponding to the projections $\mathcal{X},CB\mathcal{X}$ and $CR\mathcal{X}$ are homotopic. The braids
corresponding to the projections $B\mathcal{X},R\mathcal{X}$ and $CRB\mathcal{X}$ are also homotopic, and represent the
inverse of the braid represented by $\mathcal X$. 
\begin{proof}
The proof is by elementary manipulation. The best way of convincing oneself of the truth of the lemma is probably to experiment with a real-world three-dimensional string model. 

Let us first discuss why $C\mathcal{X}$ cannot be the projection of a braid. Suppose it were. The blue
strand would be represented by a map $r_B:[0,1]\to D^2\times [0,1]$, and similarly with the red strand.
The lower crossing would represent a point $(t_x^B,x)$ on the blue strand of $C\mathcal{X}$ and a point $(t_x^R,x)$ on the red strand,
and similarly the upper crossing would represent points $(t_y^B,y)$ and $(t_y^R,y)$. 
Because we have now inverted the crossings,  we have that $t_x^B<t_x^R$ and $t_y^R<t_y^B$. On the other hand,
the orientation of the strands shows that $t_x^R <t_y^R$ and $t_y^B <t_x^B$, so we get
$t_x^R<t_y^R<t_y^B<t_x^B<t_x^R$, which is a contradiction. 
 
The three pictures in  Figure~3 represent braids homotopic to $\mathcal X$.
As is easily seen in the upper row pictures, reversing time replaces a braid by its inverse.
This introduces an involution of the situation. The effect on the pictures in the lower row is
that the crossings are inverted and also that the orientations of both strands are inverted.
In particular, this means that $CRB\mathcal{X}$ is indeed the projection of a braid, and actually by the 
inverse of the braid projecting to $\mathcal X$. Applying the involution to the picture  $C\mathcal{X}$, we get
$BR\mathcal{X}$. Since $C\mathcal{X}$ cannot be the projection of a braid, neither can $BR\mathcal{X}$ be.

Applying the involution to the lower right picture, we see that inverting the orientation of the blue strand 
inverts the corresponding braid. This shows that $B\mathcal{X}$ and $R\mathcal{X}$ represent the inverses of $\mathcal X$. 
Similarly, $CBR\mathcal{X}$ represents the inverse of both $CR\mathcal{X}$ and $CB\mathcal{X}$. Since we already proved that
$CBR\mathcal{X}$ represents the inverse of $\mathcal{X}$, we are done with the  proof of the lemma.
\end{proof}
\end{lemma}

\begin{figure}[ht]
\label{fig.descriptive}
\vspace{5mm}
\includegraphics[width=13cm,angle=0]{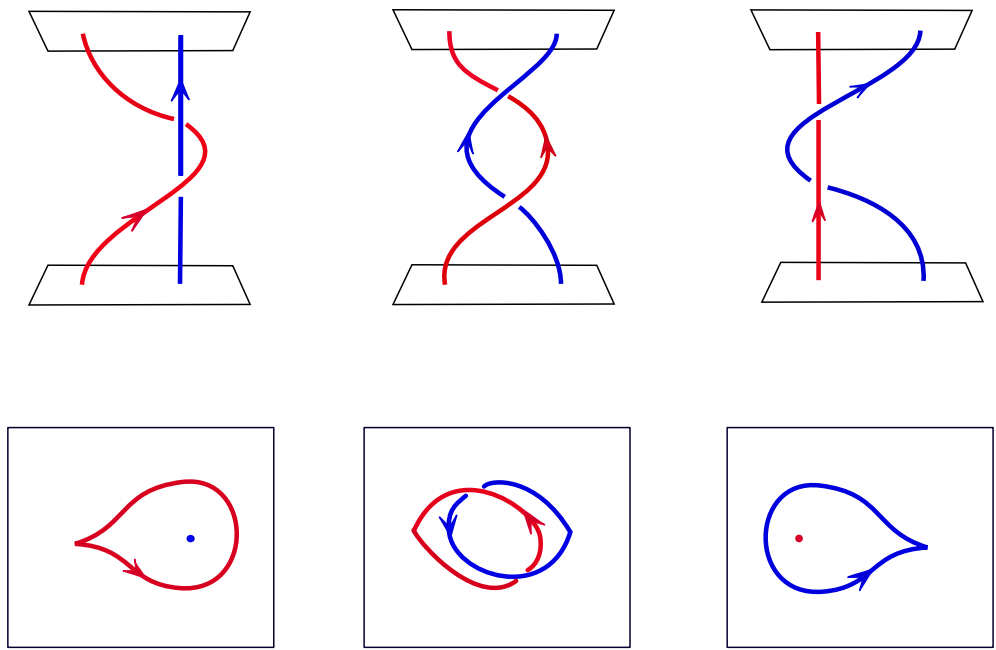} \\[-0pt]

\vspace{5mm}
\caption{The two pictures in each column are two different projections of the same divisor braid. The braid on the left side of the picture is homotopic to the braid on the right side.}
\end{figure}

\begin{figure}[hb]
\label{fig.b}
\vspace{5mm}
\includegraphics[width=5cm,angle=0]{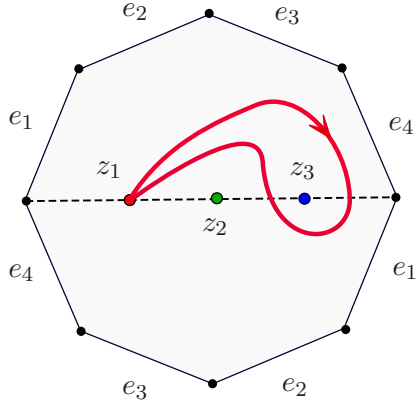} \\[-0pt] 
\vspace{-40mm}\hspace{-50mm}$e_1$\\
\vspace{-19mm}\hspace{-20mm}$e_2$\\
\vspace{-4mm}\hspace{20mm}$e_3$\\
\vspace{10mm}\hspace{50mm}$e_4$\\
\vspace{15mm}\hspace{-50mm}$e_4$\\
\vspace{11mm}\hspace{-20mm}$e_3$\\
\vspace{-5mm}\hspace{22mm}$e_2$\\
\vspace{-20mm}\hspace{51mm}$e_1$\\
\vspace{-18mm}\hspace{-27mm}$z_1$ \\
\vspace{3mm}\hspace{1mm}$z_2$ \\
\vspace{-12mm}\hspace{24mm}$z_3$ \\
\vspace{32mm}
\caption{The path $\beta_{1,3}$.}
\end{figure}

\begin{lemma}
\label{commutator}
Let $\lambda, \mu$ be two different colours, and $1\le \ell, \ell' \le 2g$. Then 
\[
[\alpha_{\lambda,\ell},\alpha_{\mu,\ell'}]=
\beta_{\lambda,\mu}^{-1}
\]
\begin{proof}
The left-hand side of the equation above is the commutator of two monic braids, built from paths that cross each other in exactly one point.  We draw the situation in Figure~5. Using the argument from  Lemma~\ref{illusion}, we obtain the assertion.
\end{proof}
\end{lemma}

\begin{figure}
\label{fig.commutator}
\vspace{5mm}
\includegraphics[width=5cm,angle=0]{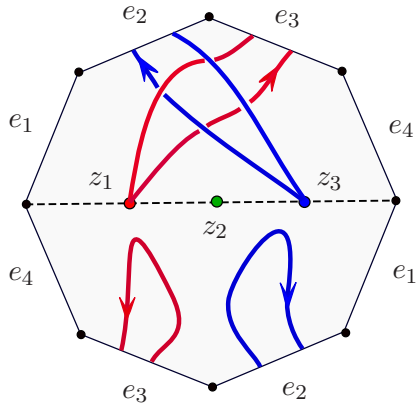} \\[-0pt] 
\vspace{-40mm}\hspace{-50mm}$e_1$\\
\vspace{-19mm}\hspace{-20mm}$e_2$\\
\vspace{-4mm}\hspace{20mm}$e_3$\\
\vspace{10mm}\hspace{50mm}$e_4$\\
\vspace{15mm}\hspace{-50mm}$e_4$\\
\vspace{11mm}\hspace{-20mm}$e_3$\\
\vspace{-5mm}\hspace{22mm}$e_2$\\
\vspace{-20mm}\hspace{51mm}$e_1$\\
\vspace{-17mm}\hspace{-29mm}$z_1$ \\
\vspace{2mm}\hspace{1mm}$z_2$ \\
\vspace{-11mm}\hspace{31mm}$z_3$ \\
\vspace{32mm}

\caption{The commutator $[\alpha_{1,3},\alpha_{3,2}]$.}
\end{figure}

\begin{corollary}
Provided that the genus of $\Sigma$ is positive, the classes $\alpha_{\lambda,\ell}$ generate $\DivBraid_\kk (\Sigma,\Gamma)$. In the case $\Sigma \cong S^2$, the classes $\beta_{\lambda,\mu}$ generate $\DivBraid_\kk (\Sigma,\Gamma)$.
\begin{proof}
The first statement follows from Lemma~\ref{generating.set} and Lemma~\ref{commutator}.

In the genus-zero case  $\Sigma\cong S^2$, the group $\DivBraid_\kk(\Sigma,\Gamma)$ is generated by the classes $\beta_{\lambda,\mu}$, because they are the images of Artin's generators~\cite{KasTur} for the pure braid group. 

\end{proof}
\end{corollary}

\begin{lemma}
\label{central}
The elements $\beta_ {\lambda,\mu}\in \DivBraid_k(\Sigma,\Gamma)$ are central.
\begin{proof} Consider first the case $g\geq 1$. Since the classes $\alpha_{\lambda,\ell}$ generate the group, 
we only need to show that $\alpha_{\lambda,\ell}$ commutes with $\beta_{\mu,\nu}$ for all choices of $\lambda,\mu,\nu,\ell$.
We are free to chose any of a number of equivalent braids to represent $\beta_{\mu,\nu}$.
First, since $k_\mu\geq 2$, we can find a special point $z_e$ of color $\mu$, which is different from 
the point $z_{i_\lambda}$ where $\alpha_{\lambda, \ell}$ starts out.  

According to Lemma~\ref{homologyfactorization}, 
if we pick any loop $\gamma$ starting at $z_e$ which represents the homology class in 
$H_1( \Sigma\setminus \{z_j\, \vert\,  j\not=i_\mu\};\ZZ)$ equal to a small circle around $z_{i_\nu}$, then the braid $\Phi(\gamma)$ will represent $\beta_{\mu,\nu}$. It is easy to see that we can always find such a  
curve  $\gamma$ that does not intersect $a_{\lambda, \ell}$. By Lemma~\ref{commutingbraids}, the statement follows.
\end{proof}
\end{lemma}

Let  now $E_\kk (\Gamma)$ be the Abelian subgroup of  $\DivBraid_\kk(\Sigma,\Gamma)$ generated by the classes $\beta_{\lambda,\mu}$. These classes were constructed from  strands projecting to a disc in $\Sigma$, so it is evident that this group does not depend on the genus of the surface; this is why we choose to suppress the dependence on $\Sigma$.   Recall the map $h$ in equation (\ref{Hurew}), which provided a generalisation of the Hurewicz homomorphism in
Remark~\ref{Hurewicz}.
\begin{thm}
\label{centralextension}
The group $\DivBraid_\kk (\Sigma,\Gamma)$ sits in  a central extension
\begin{equation} \label{centrext}
 E_\kk(\Gamma) \to \DivBraid_\kk (\Sigma,\Gamma) \xrightarrow{h} H_1(\Sigma;\ZZ)^{\oplus r}.
\end{equation}
\begin{proof}
Assume that $g\geq 1$.
Since   
$E_\kk (\Gamma)$ is generated by the commutators $\beta_{\mu ,\nu}$, it is certainly contained in the 
kernel of $h$.
The group $E_\kk (\Gamma)$ is a central subgroup of
$\DivBraid_\kk (\Gamma)$. The quotient group $A := \DivBraid_\kk (\Sigma,\Gamma)/E_\kk (\Gamma)$ is Abelian, generated by the equivalence classes
$\alpha_{\lambda, \ell}$. The map $h$ factors over the quotient, and we have to show that the  
induced homomorphism $\bar h:A \to H_1(\Sigma;\ZZ)^{\oplus r}$ is an isomorphism.
By the definition of $\alpha_{\lambda,\ell}$, the group $H_1(\Sigma;\ZZ)$ is generated by the cycles 
$h(\alpha_{\lambda,\ell})$,  where $1\leq \ell \leq 2g$. It follows that $\bar h$ is surjective. 
Since   $H_1(\Sigma;\ZZ)^{\oplus r}$ has rank greater or equal to the rank of $A$, the map $\bar h$
has to be an isomorphism. 
This finishes the proof of the theorem in positive genus.

Finally, suppose that the genus $g$ is zero. In this case,  $\DivBraid_\kk(\Sigma,\Gamma)$ is equal to $E_\kk (\Gamma)$, so the statement in the theorem is trivially true. 
\end{proof}
\end{thm}

\begin{remark}
This result is in the same spirit of the presentations given in~\cite{BelGodGua} for the so-called {\em mixed braid groups}  defined in \cite{AnKo}. We emphasise that our divisor braid
groups are distinct from such mixed braid groups. One relation between the two is that there is an obvious epimorphism from the mixed braid group of $\Sigma$ in
$(k_1,k_2)$-strings $\mathbf{B}_{k_1,k_2}$ to the particular divisor braid group ${\sf DB}_{k_1,k_2}(\Sigma; \bullet\!\! \!-\!\!\!-\!\circ)$ that we shall study in more detail in Section~\ref{sec:twocolours}, and this morphism is not injective in general.
\end{remark}

We close this section by summarising the properties of the classes $\beta_{\lambda,\mu}$.

\begin{proposition}
\label{Erelations}
The group $E_\kk (\Gamma)$ is generated by the classes $\beta_{\lambda ,\mu}$, $1 \leq \lambda,\mu \leq k$.
They satisfy the following relations:
\begin{itemize}
\item $\beta_{\lambda,\mu}=0$ if there is no edge in $\Gamma$ from $\lambda$ to $\mu$
\item $\beta_{\lambda,\mu}=\beta_{\mu,\lambda}$
\item 
$
\sum_{\mu \not=\lambda} k_\mu \beta_{\lambda,\mu}=0.
$
\end{itemize}
\begin{proof}
The only nontrivial statement left is the last relation.
Pick a disc $D^2$ containing all the points $z_{i_\lambda}$. Let $\gamma$ be a path stating in $z_{i_\lambda}$, going to the boundary of the disc, follow the whole boundary of the $4g$-gon once and then returning to $z_{i_\lambda}$. This path is homotopic to a product of commutators in $\pi_1(\Sigma\setminus \bigcup_{j\not=i_\lambda} \{z_j\},z_{i_\lambda})$. It follows that the monic braid $\Phi(\gamma)$ yields the trivial element of $\DivBraid_\kk(\Sigma,\Gamma)$. On the other hand, it circles each point $z_{j}$ for $j\not = i_\lambda$ exactly once, so it does represent the sum in the formula. 
\end{proof}
\end{proposition}

\section{Linking numbers  and two colours}
\label{sec:linking}

In the previous section, we obtained a set of generators for our divisor braid groups, together with some relations. 
We claim that there are no further relations, so that we have actually obtained a presentation of the groups.
To prove this, one has to
find tools to detect that a braid is non-trivial.  One idea is to use  invariants for the links (or `closed braids') in the closed three-manifold $S^1 \times \Sigma$ obtained from colour-pure braids on $\Sigma$.  In this section, we will first discuss link invariants for disjoint cycles in such coloured closed braids, and then apply them to the presentation of divisor braid groups in a crucial situation: very composite colour schemes whose graph consists of a single edge and two vertices.

\subsection{Linking numbers of cycles}

We start with a fairly general situation.
Let $R$ be a ring, $n\in \NN$, and $M$ a compact $(2n+1)$-dimensional manifold with an $R$-orientation. Assume that $\alpha_1, \alpha_2$ are chains in $C_n(M;R)$ which are boundaries.
Also assume that these cycles $\alpha_i$ are disjoint. Our goal is to define an homotopy-invariant linking number in $R$, which will be able to detect whether the $n$-cycles $\alpha_1$ and $\alpha_2$ are linked. 

If $\alpha_1$ and $\alpha_2$ are contained in an open set homeomorpic to $\RR^{2n+1}$, we can consider the 
usual linking number in this open set. Think of this number as a local linking number. If this local linking number is different from zero, the $n$-chains are linked locally. However, even if the local linking number is different from zero, they might get unlinked by some global move. So if we can define a global linking number for links on $M$    
it will have to be a weakening of the local linking number. In general, even if the local linking number is nonzero, the global linking number might still vanish. 

By assumption, there is a chain $A_i\in C_{n+1}(M;R)$ such that $ \partial A_i=\alpha_i$. Consider the chain $A_1\times \alpha_2\in C_{2n+1}(M \times M ;R)$.
Since the cycles $\alpha_i$ are disjoint, we obtain a boundary $\partial(A_1\times \alpha_2)\in C_{2n}(M\times M\setminus M_\Delta;R)$, where $M_\Delta$ is the usual diagonal. It follows that $A_1\times \alpha_2$ defines a cycle in $C_{2n+1}(M\times M, M\times M\setminus M_\Delta;R)$. 
\begin{defn}
The {\em linking number} $L(\alpha_1,\alpha_2) \in R$ of the disjoint $n$-cycles $\alpha_1,\alpha_2$ in $M$ is the intersection of the homology classes $[A_1\times \alpha_2]$ and $[M_\Delta]$ in $M\times M$. 
\end{defn}
\begin{defn} A {\em continuous family of singular chains} parametrised by $t$ is a formal sum $\sum_i c_i\, S_i(t)$, where $c_i\in R$ are constants and each $S_i(t)$ is a singular simplex, continuously parametrized by $t$.
\end{defn}

We now collect basic facts about this  linking number.

\begin{proposition} 
\label{chaininvariant}
The linking number $L(\alpha_1,\alpha_2)$ enjoys the following properties:
\begin{itemize}
\item[(i)]
It does not depend of the choice of $A_1$. 
\item[(ii)]
It is graded skew-symmetric.
\item[(iii)]
It is $R$-bilinear in the appropriate sense.
\item[(iv)]
It is constant in a continuous family of disjoint pairs of boundaries $(\alpha_1(t),\alpha_2(t))$.
\end{itemize}
\begin{proof}
(i) Two different choices $A_1, A_1'$ differ by a cycle $z\in Z_{2n}(M;R)$. The corresponding homology classes $[A_1 \otimes \alpha_2], [A_1'\otimes \alpha_2]\in H_{2n+1}(M\times M,M\times M\setminus M_\Delta;R)$ differ by $[z\times \alpha_2]$. This is actually a cycle in $C_{2n+1}(M\times M;R)$, and the corresponding homology class  vanishes already there since $[\alpha_2]=0\in H_{n}(M;R)$.

(ii) Skew-symmetry follows from observing that, if $\partial A_2=\alpha_2$, then $$\partial(A_1\otimes A_2)=(-1)^n\alpha_1\otimes A_2+A_1\otimes \alpha_2,$$ so that the classes $$[A_1 \otimes \alpha_2], (-1)^{n+1}[A_2\otimes \alpha_1]\; \in\;  H_{2n+1}(M\times M, M \times M\setminus M_\Delta;R)$$
agree. 

(iii) Linearity in the second argument is immediate from the definition. The caveat `in the appropriate sense' simply means that 
$$L(\alpha_1,\alpha_2^\prime+\alpha_2^{\prime\prime})=L(\alpha_1,\alpha_2^\prime)+L(\alpha_1,\alpha_2^\prime)$$ holds under the condition that both $\alpha_2^\prime$ and
$\alpha_2^{\prime\prime}$ are disjoint from $\alpha_1$. 
 Linearity in the first argument follows from this equality and skew-symmetry.

(iv) By a compactness argument, every continuous path of pairs of disjoint cycles is composed by paths of cycles for which either the first or the second chain is held constant.  
By skew-symmetry, it is enough to prove homotopy invariance in the case where $\alpha_2(t)$ is constant. 
The homotopy $\alpha_1(t)$ gives a chain $A$ which is disjoint from $\alpha_2$ and such that $\partial A=\alpha_1(1)-\alpha_1(0)$. By definition,
$L(\alpha_1(1)-\alpha_1(0),\alpha_2)$ is the intersection of $M_\Delta$ with $A\times \alpha_1$. But this intersection is empty, so $L(\alpha_1(1)-\alpha_1(0),\alpha_2)=0$.
Therefore, by linearity in the first argument,
\[
L(\alpha_1(0),\alpha_2)=L(\alpha_1(1)-\alpha_1(0),\alpha_2)+L(\alpha_1(0),\alpha_2)=L(\alpha_1(1),\alpha_2),
\]
which proves the homotopy invariance.
\end{proof}
\end {proposition}

The example we will now focus on is basic, and it brings us back to the  stage where $M=S^1 \times \Sigma$ for a compact oriented surface $\Sigma$, and thus $n=1$. 

Let $k_1,k_2 \in \NN$ and $\{z_i \, | \, 1\leq i\leq k_1 + k_2 \}$ be a set of distinct points of $\Sigma$. We denote by $\gamma_i$  the constant path at the point $z_i$. Let $f:D^2\to \Sigma$ be an oriented  chart, such that $f(0,0)=z_1$ and $f(1,0)=z_{k_1+1}$, and whose image  does not contain any point $z_i$ for $i\not =1,k_1+1$.  Let $\gamma_1^\prime$ be the restriction of such a map $f$ to the boundary. Then  we can reinterpret $\gamma_1'$  as a loop on $\Sigma$, starting at $z_{k_1+1}$, winding once around $z_1$ in the positive direction, and returning back to $z_{k_1+1}$.

Consider the disjoint cycles $\alpha_1,\alpha_2 \in C_1(S^1\times \Sigma;\ZZ/{{\rm gcd}(k_1,k_2)})$ given by the families of maps $\{\gamma_1^\prime,\gamma_2,\gamma_3,\dots,\gamma_{k_1}\}$ and   $\{\gamma_{k_1+1},\dots,\gamma_{k_1+k_2}\}$, respectively. 

\begin{lemma}
\label{simpleloop}
$L(\alpha_1,\alpha_2) \equiv 1\, ({\rm mod}\, {\rm gcd}(k_1,k_2))$.
\begin{proof} 
Let $F:S^1\times [0,1]\to S^1\times \Sigma$ be the map $F(t,u)=(t,f((1-u)t))$. This map defines a chain $A_1$ whose boundary is the difference between the chains represented by $\gamma_1$ and $\gamma_1^\prime$. To compute the linking number $L(\alpha_1,\alpha_2)$ we first need to make $A_1$ transversal to $\alpha_2$. We can do this by moving $\alpha_2$ to another constant path nearby.  The intersection number is then the number of points  $((t_1,u),(t_2))\in (S^1 \times [0,1])\times S^1)$ such that $F(t_1,u)=(t_2,w)$ where $w$ is a point close (but not equal to) to $z_1$, let us say $w = f(\epsilon,0)$. The unique solution is  $t_1 = t_2 = 1-\epsilon$ and $u= (1,0)$. We also have to keep track of the sign, which with the usual orientation conventions is positive in this case, so that the intersection number is $1\, ({\rm mod}\, {\rm gcd}(k_1,k_2))$ as claimed.
\end{proof}    
\end{lemma}

\subsection{Application to  divisor braids in two colours} \label{sec:twocolours}
In this section we focus on the case of a very composite negative colour scheme with graph $\Gamma = \bullet\!\! \!-\!\!\!-\!\circ$; thus we deal with $r=2$ colours.
We want to determine in this situation the Abelian group $E_{(k_1,k_2)}(\Sigma,\Gamma)$ defined in the paragraph before Theorem~\ref{centralextension}. 

The first step is to construct an invariant of divisor braids.  Suppose that $\gamma:[0,1]\to \Sigma^{k_1+k_2}$ represents a divisor braid. We consider the maps $\bar\gamma_i:[0,1]\to S^1\times \Sigma$ where $\bar\gamma_i(t)=(t\,({\rm mod}\, 1),\gamma_i(t))$. (Here, we identify $S^1$ with $\RR/\ZZ$.) The first $k_1$ of the maps $\gamma_i:[0,1]\to \Sigma$ determine a 1-cycle $\alpha_1=\sum_{j=1}^{k_1}[\bar\gamma_j]$ and the last $k_2$ determine a 1-cycle $\alpha_2=\sum_{j=k_1+1}^{k_1+k_2}[\bar\gamma_j]$ . By assumption, these two cycles can be taken disjoint in $S^1\times \Sigma$.

Now assume triviality of the homology classes $h(\alpha_\lambda)=0\in H_1(\Sigma;\ZZ)$ for both $\lambda=1,2$. Then the homology class 
$$h(\bar\alpha_\lambda)\; \in \; H_1(S^1 \times \Sigma;\ZZ)\cong H_1(S^1;\ZZ) \oplus H_1(\Sigma;\ZZ)$$
 equals $k_\lambda[S^1]$, for $\lambda=1,2$. In particular, if we consider homology with coefficients in the ring $\ZZ/{\rm gcd}(k_1,k_2)$, then the corresponding homology classes are trivial. 
 
 We can now define the {\em link invariant} $L(\gamma)$ of the divisor braid in two colours $\gamma$ to be the linking number $$L(\alpha_1,\alpha_2)\; \in\;  \ZZ/{\rm gcd}(k_1,k_2)$$ introduced above.
According to Proposition~\ref{chaininvariant}, this is both well defined and homotopy invariant. It is also easy to see that if two braids are related by a crossing move, then they have the same invariant (the most obvious argument consists of subdividing the chains at the crossing point, and using that a chain is chain-homotopic to its subdivision).

\begin{example} Let $\gamma$ be a trivial braid of constant paths. The 2-cycle $A_1$ does not intersect $\alpha_2$, so the braid invariant $L(\gamma)$ is zero.
\end{example}

\begin{example} \label{excalculate} 
We can compute the link invariant  $L(\beta_{1,2})$ for the divisor braid defined in (\ref{alphabeta}) and associated to the negative colour scheme 
$(\Gamma ,\kk)= (\bullet\!\! \!-\!\!\!-\!\circ, (k_1,k_2))$. This braid consists of one point of colour 1 starting at $z_1$ and tracing a small circle in positive direction around a point $z_2$ of colour 2.  Using Lemma~\ref{simpleloop},  we see that the link invariant equals
$1\, ({\rm mod}\, {\rm gcd}(k_1,k_2))$. More generally, if $n\in \ZZ$ then  $$L(n\, \beta_{1,2})\;  \equiv\;  n\, ({\rm mod}\, {\rm gcd}(k_1,k_2)).$$ 
\end{example}

The following result can be regarded as the first nontrivial computation of a genuine divisor braid group.

\begin{thm}
\label{twocolors}
The reduction ${\rm mod}\, {\rm gcd}(k_1,k_2)$ of the intersection pairing of 1-cycles
$$H_1(\Sigma;\ZZ)\otimes_\ZZ H_1(\Sigma;\ZZ)\to \ZZ/\gcd(k_1,k_2)$$ induces a 
central extension   of $H_1(\Sigma;\ZZ)^{\oplus 2}$ by $\ZZ/\!\gcd(k_1,k_2) $ which is isomorphic to the divisor braid group $\DivBraid_{(k_1,k_2)}(\Sigma, \bullet\!\! \!-\!\!\!-\!\circ)$.

\begin{proof}
According to Proposition~\ref{Erelations}, the centre $E_\kk( \bullet\!\! \!-\!\!\!-\!\circ)$ is generated as an Abelian group by $\beta_{1,2},\beta_{2,1}$ under the relations $\beta_{1,2}-\beta_{2,1}=0$, $k_2\, \beta_{1,2}=0$ and $k_1\,\beta_{2,1}=0$. Thus $E_{(k_1,k_2)}(\bullet\!\! \!-\!\!\!-\!\circ)$ is generated by $\beta_{1,2}$, with the 
relation $\gcd(k_1,k_2)\, \beta_{1,2}=0$ holding true. However, there could conceivably be more relations.

Let $\widetilde{E}_{(k_1,k_2)}( \bullet\!\! \!-\!\!\!-\!\circ)$ be the Abelian group generated by the symbols $b_{1,2},b_{2,1}$ 
under the relations $b_{1,2}-b_{2,1}=0, k_2b_{1,2}=0$ and $k_1b_{2,1}=0$. This is a cyclic group of order
$\gcd(k_1,k_2)$, generated by $b_{1,2}=b_{2,1}$. It will be handy to write also $b_{1,1}=b_{2,2}=0$ for the zero element of this group.

For any 1-cycle $a \in H_1(\Sigma;\ZZ)$ we shall write $a_1=(a,0)$ and $a_2=(0,a)$ for the two inclusions in the direct sum $\bigoplus_{\lambda=1}^2 H_1(\Sigma;\ZZ)$.
Let $\widetilde{\DivBraid}_{(k_1,k_2)}(\Sigma,\bullet\!\! \!-\!\!\!-\!\circ)$
be the extension of $H_1(\Sigma;\ZZ)^{\oplus 2}$ by $\widetilde{E}_{k_1,k_2}( \bullet\!\! \!-\!\!\!-\!\circ)$ 
corresponding to the 2-cocycle in $Z^2(H_1(\Sigma;\ZZ)^{\oplus 2}, \widetilde{E}_{k_1,k_2}( \bullet\!\! \!-\!\!\!-\!\circ))$ given by
\begin{equation} \label{2cocycle}
(a_\lambda,a'_{\lambda'})\mapsto \sharp (a,a')\, b_{\lambda,\lambda'},
\end{equation}
where $\sharp (\cdot,\cdot)$ is the intersection pairing on $H_1(\Sigma;\ZZ)$ (see \cite[p.~827]{Lan} for background on group extensions and group cohomology). From the way we constructed $\widetilde{E}_{k_1,k_2}( \bullet\!\! \!-\!\!\!-\!\circ)$, it is
clear that the bilinear map (\ref{2cocycle}) only depends on the values of the intersection pairing modulo congruence by ${\rm gcd}(k_1,k_2)$.

It follows from Theorem~\ref{centralextension} that there is 
a ladder diagram of groups with surjective vertical arrows
and exact rows:
\[
\begin{CD}
0 @>>> \widetilde{E}_{(k_1,k_2)}( \bullet\!\! \!-\!\!\!-\!\circ)  
@>>> \widetilde{\DivBraid}_{(k_1,k_2)}(\Sigma,\bullet\!\! \!-\!\!\!-\!\circ)
@>>> H_1(\Sigma;\ZZ)^{\oplus 2} @>>> 0 \\
@. @VVV @VVV @| @. \\
0 @>>> E_{(k_1,k_2)}( \bullet\!\! \!-\!\!\!-\!\circ)  
@>>> \DivBraid_{(k_1,k_2)}(\Sigma,\bullet\!\! \!-\!\!\!-\!\circ)
@>>> H_1(\Sigma;\ZZ)^{\oplus 2} @>>> 0 
\end{CD}
\]
Note that the groups in the middle column are not necessarily commutative. 
The lemma follows if we can show that the left vertical map is an isomorphism.
We already know that the map is surjective, so we have to show injectivity.

We claim that the kernel of the vertical homomorphism on the left, determined by $b_{1,2} \mapsto \beta_{1,2}$,  is trivial. That is, if $\gcd(k_1,k_2)$ does not divide $m$, then $m\, \beta_{1,2}\not=0\in \DivBraid_{(k_1,k_2)}(\Sigma,\bullet\!\! \!-\!\!\!-\!\circ)$. But this follows immediately from the computation of the link invariant in Example~\ref{excalculate} above. \end{proof}
\end{thm}

\section{Linking numbers for general negative colour schemes}
\label{sec:colors}

Now we consider the case of more than two colours, dealing with negative colour schemes $(\Gamma,\kk)$ based on a general graph $\Gamma$ without self-loops. We shall
still restrict our attention to the case where $\kk$ is very composite, as in the last section and most of the previous one.
Our main aim is to compute the centre $E_\kk( \Gamma)$ of the group $\DivBraid_\kk(\Sigma, \Gamma)$ presented as central extension in Theorem~\ref{centralextension}.

Let $D_\kk(\Gamma)$ be the free Abelian group generated by the symbols $b_{\lambda,\mu}$ for all pairs of colours $(\lambda,\mu)$ with  $1\leq \lambda<\mu \leq r$ connected by an edge in $\Gamma$. According to Proposition \ref{Erelations},  there is a surjective map $\lp :D_\kk(\Gamma) \to E_\kk(\Gamma)$ given by $\lp (b_{\lambda,\mu})=\beta_{\lambda,\mu}$. The kernel of this map contains the image of the map 
$P_{(\Gamma,\kk)}:\ZZ^{\oplus r}\to D_\kk(\Gamma)$ given on generators $e_	\lambda$ by the formula $P_{(\Gamma,\kk)}(e_\lambda)=\sum_{\mu\not=\lambda} k_\mu b_{\lambda,\mu}$. 

The problem is that there could well be further relations. The main purpose of  this section is to prove that there are none. To do this, we have to show that certain elements of $\DivBraid_\kk(\Sigma,\Gamma)$ are nontrivial. Our basic strategy is to follow the method that we successfully used in the case of two colours. We shall show how to construct a suitable generalisation of the linking number for cycles, depending on each graph $\Gamma$, and then use it to define invariants of elements of  $\DivBraid_\kk(\Sigma, \Gamma)$. 

\subsection{The $\Gamma$-linking number}

Let $K$ be an Abelian group.  
We shall define linking numbers for families of cycles with values in $K$ for a given negative colour scheme $(\Gamma,\kk)$ in $r$ colours.

Let $M$ be an oriented  manifold. We fix a chain  $\basic \in C_*(M;\ZZ)$ and call it the {\em basic chain} in $M$. The main example we will be interested is when $M = S^1\times \Sigma$ for $\Sigma$ a connected surface, and  $\basic$ is   $S^1 \times \{ z \}$ for any chosen point $z \in \Sigma$. The linking number that we shall introduce below will depend on the choice of $\basic$. 

We say that two chains $\xi,\eta\in C_*(M;\ZZ)$ are disjoint if there exist  disjoint open sets $U,V\subset  M$, such that $\xi$ is in the image of the chain map $C_*(U;\ZZ)\to C_*(M;\ZZ)$ induced by the inclusion, whereas $\eta$ is in the image of $C_*(V;\ZZ)\to C_*(M;\ZZ)$. 

\begin{defn} Fix a graph $\Gamma$ without self-loops and a basic chain $\basic$ in $M$.

(i)
A  {\em $\Gamma$-adapted pair} $(a,b)$ consists of two collections of chains  $a= \{a_\lambda \}_{1\leq \lambda \leq r}$ and $b= \{b_\mu \}_{1\leq \mu \leq r}$ in $C_*(M;\ZZ)$ indexed by the vertices $\lambda$ of $\Gamma$ such that whenever there is an edge in $\Gamma$ between the vertices $\lambda$ and $\mu$, the chains $a_\lambda$ and $b_\mu$ are disjoint.

(ii) Let two negative colour schemes $(\Gamma,\kk)$ and $(\Gamma,\kk')$ be given.
A {\em $(\Gamma,\kk,\kk')$-allowable pair} consists of two collections of chains  $a= \{a_\lambda \}_{1\leq \lambda \leq r}$ and $b= \{b_\mu \}_{1\leq \mu \leq r}$  satisfying the following conditions:
\begin{itemize}
\item Each $a_\lambda$ and $b_\mu$ is a cycle.
\item The pair $(a,b)$ is $\Gamma$-adapted.
\item For every colour, we have an equality of homology classes 
$[a_\lambda] = k_\lambda [\basic]$ and $[b_\mu] = k'_\mu[\basic]$, respectively.
\end{itemize}
The collection $a$ is a called {\em  $(\Gamma,\kk)$-allowable} if $(a,a)$ is a $(\Gamma,\kk,\kk)$-allowable pair.
\end{defn}

A continuous one-parameter family of allowable pairs of chains $\{ (a(t),b(t))\,|\,t \in [0,1] \}$ is called a homotopy
between the allowed pairs of chains $(a(0),b(0))$ and $(a(1),b(1))$. 

We want to find homotopy invariants of  linking type for $(\Gamma,\kk,\kk')$-allowable pairs.  
The first step is to choose, for each colour $\lambda$,
a chain $A_{\lambda}\in C_*(M;\ZZ)$ such that
\begin{equation}\label{chainA}
\partial A_{\lambda}=a_\lambda-k_\lambda\basic.
\end{equation}
Next, we try to find  linear combinations of 
the chains $A_{\lambda}\otimes b_\mu$
with the property that their boundary is contained in the 
sub-chain complex $C_*(M\times M \setminus M_\Delta;K)$, where $M_\Delta$ is the diagonal.
For this purpose, we consider linear combinations
\begin{equation}\label{SY}
S_{\class Y}({a,b})=\sum_{\lambda,\mu}A_{\lambda}\otimes b_\mu \otimes  Y_{\lambda,\mu}  \in
C_*(M;\ZZ)\otimes C_*(M;\ZZ)\otimes K,
\end{equation}
where $\class Y := \{Y_{\lambda,\mu}\}_{\lambda,\mu=1}^r$ is a certain family of elements in $K$ (which we will refer to as a {\em $K$-collection} for short), which
we will specify more precisely later.
The boundary of this chain is
\begin{eqnarray*}
\partial  S_{\class Y}(a,b)&=&\sum_{i,j}(a_\lambda - k_\lambda\basic) \otimes b_\mu
\otimes Y_{\lambda,\mu} \\
&=&\sum_{\lambda,\mu} a_\lambda \otimes b_\mu \otimes Y_{\lambda,\mu}
-\sum_{\lambda,\mu} \basic \otimes b_\mu \otimes k_\lambda Y_{\lambda,\mu} .
\end{eqnarray*}

Now we shall assume that  the $K$-collection $\class Y$ satisfies the following 
\emph{$(\Gamma,\kk,\kk')$-allowability conditions}:
 
  \begin{itemize}
\item[(A1)] $\qquad$ $\displaystyle \sum_{\lambda=1}^r k_\lambda  Y_{\lambda,\mu} = 0\;$ and $\; \displaystyle \sum_{\mu=1}^r k_\mu' Y_{\lambda,\mu}=0 \; $ in  $K$.\\[-3 pt]
\item[(A2)] $\qquad$ If $(\lambda,\mu)$ is not an  edge of $\Gamma$, then also $Y_{\lambda,\mu}=0 \in K$.
  \end{itemize}

If $\class Y$ is the $K$-collection $\{Y_{\lambda,\mu}\}_{\lambda,\mu=1}^r$, we define 
its transpose $\class Y^T$ to be the $K$-collection $\{Y_{\mu,\lambda}\}_{\lambda,\mu=1}^r$, where the colour indices $\lambda$ and $\mu$ have been interchanged.  Clearly,  $\class Y$ is  $(\Gamma,\kk,\kk')$-allowable if and only if  $\class Y^T$ is  $(\Gamma,\kk',\kk)$-allowable. We say that a $K$-collection is $(\Gamma,\kk)$-allowable in the case
$\kk'=\kk$.

Note that if $\class Y$ is  $(\Gamma,\kk,\kk')$-allowable, then by (A1)
\[
\partial  S_{\class Y}(a,b)
=\sum_{\lambda,\mu} a_\lambda \otimes b_\mu \otimes Y_{\lambda,\mu}.
\]

Recall the Eilenberg--Zilber map  from~\cite{EilZil}. This map is a chain homotopy equivalence
\[
{\rm EZ}_{A,B}:C_*(A)\otimes C_*(B)\to C_*(A\times B)
\] 
which is functorial in $A$ and $B$. 
In particular, if $\xi,\xi'$ are two disjoint chains in $C_*(M;\ZZ)$, it 
follows from said functoriality that 
$${\rm EZ}_{M,M}(\xi\otimes \xi')\in C_*(M\times M\setminus M_\Delta;\ZZ).$$

Assume that $(a,b)$ is a $(\Gamma,\kk,\kk')$-allowable pair, and that $\partial A_{i} = a_i -k_i\basic$ as above. 
Fix a  family $\class Y= \{Y_{\lambda,\mu}\}_{\lambda,\mu=1}^r$ satisfying the $(\Gamma,\kk,\kk')$-allowability conditions (A1) and (A2).

\begin{lemma} \label{IndependentA}
 With the hypotheses stated, the homology class of $S_{\class Y}({a,b})$ in (\ref{SY}) does not depend on the choices of
the chains $A_{\lambda}$. 
\begin{proof}
The  difference between two choices of the chains $A_{\lambda}$ is a cycle
$c_{\lambda}\in C_*(M;\ZZ) $. The difference between the two corresponding
definitions of $S_{\class Y}({a,b})$ is the homology class
\begin{eqnarray*}
[\sum_{\lambda,\mu} c_{\lambda}\otimes b_\mu \otimes Y_{\lambda,\mu}]&=&
\sum_{\lambda,\mu} [c_{\lambda}] \otimes [b_\mu] \otimes Y_{\lambda,\mu}\\
&=&\sum_{\lambda,\mu}[c_{\lambda}] \otimes [\basic]  \otimes k_\mu Y_{\lambda,\mu},
\end{eqnarray*}
which vanishes by (A2), and this establishes the claim.
\end{proof}
\end{lemma}

Let now $$j_*:C_*(M\times M;\ZZ) \to C_*(M \times M,M\times M\setminus M_\Delta;\ZZ)$$ be the canonical  quotient map.
We consider the chain
\[
j_*{\rm EZ}_{M,M} (S_{\class Y}(a,b) )\,  \in\,  C_*(M \times M,M\times M\setminus M_\Delta;\ZZ).
\]
The boundary of this chain can be written as
\[
j_*{\rm EZ}_{M,M} (\partial S_{\class Y}(a,b)) =
\sum_{\lambda,\mu} j_*{\rm EZ}_{M,M}(a_\lambda\otimes b_\mu)\otimes Y_{\lambda,\mu}
\]

Each nontrivial term in the sum above is contained in $C_*(M\times M\setminus M_\Delta;K)$.
If there is an edge in $\Gamma$ connecting the colours $\lambda$ and $\mu$, then 
${\rm EZ}_{M,M}(a_\lambda \otimes b_\mu)\in C_*(M\times M\setminus M_\Delta;K)$, because $(a,b)$ is a
$(\Gamma,\kk,\kk')$-adapted pair and we are assuming condition (A2) above.
It follows that
\[
j_*{\rm EZ}_{M,M}( S_{\class Y}(a,b)) \; \in \; C_*(M \times M,M\times M\setminus M_\Delta;\ZZ)
\]
is a cycle. 

The orientation of $M$ specifies a dual generator $u\in H^n(M\times M,M\times M\setminus M_\Delta;\ZZ)\cong \ZZ$, 
and we define  the linking invariant $\psi_{\class Y}(a,b)\in K$ by imposing that the formula
\begin{equation} \label{Gammalink}
u(\psi_{\class Y}(a,b)):=u (j_*{\rm EZ}_{M,M} (S_{\class Y}(a,b))) \;\in K
\end{equation}
shall hold. Note that if the condition (\ref{Gammalink}) is satisfied for the orientation $u$, it is also satisfied for the opposite orientation $-u$, and thus $\psi_{\class Y}(a,b)$ is defined independently of the orientation.

\begin{lemma}
\label{le:psi}
The  operation $\psi_\class Y$  defined from the $(\Gamma,\kk,\kk')$-allowable $K$-collection $\class Y$ satisfies the following properties:
\begin{itemize}
\item[(i)] 
\label{Symmetry}
 $\psi_\class Y(a,b)$= $(-1)^{\dim M+\deg \basic}\psi_{\class Y^T}(b,a)$.

\item[(ii)]
\label{Linearity} 
$\psi_\class Y$ is linear in the second argument, in the following sense:  if $(a,b)$, $(a,b')$ and $(a,b'+b)$ all are all $(\Gamma,\kk,\kk')$-allowable pairs, then
$$\psi_\class Y(a,b+b')=\psi_\class Y(a,b)+\psi_\class Y(a,b').$$ Similarly, $\psi_\class Y$ is linear in the first argument.

\item [(iii)]
\label{Homotopy}
If the $(\Gamma,\kk,\kk')$-allowable pair $(a,b)$ is homotopic to the $(\Gamma,\kk,\kk')$-allowable pair $(a',b')$ through
$(\Gamma,\kk,\kk')$-allowable pairs, then 
$\psi_\class Y(a,b)=\psi_\class Y(a',b')$. 
\end{itemize}
\begin{proof}
Assume that $(a,b)$ is a $(\Gamma,\kk,\kk')$-allowable pair.  Put
$a_\lambda - k_\lambda \basic =\partial A_{\lambda}$ and $b_\mu - k'_\mu \basic = \partial B_\mu$. Then

\begin{align*}
  \partial (\sum_{\lambda,\mu}A_{\lambda}\otimes B_{\mu}\otimes Y_{\lambda,\mu})
&=
\sum_{\lambda,\mu}\partial A_{\lambda}\otimes B_{\mu}\otimes Y_{\lambda,\mu}+
(-1)^{\deg \basic+1}\sum_{\lambda,\mu}A_{\lambda}\otimes \partial B_{\mu}\otimes Y_{\lambda,\mu}\\
&=
\sum_{\lambda,\mu}(a_\lambda-k_\lambda \basic)\otimes B_{\mu}\otimes Y_{\lambda,\mu} \\
& \qquad +
(-1)^{\deg \basic+1}\sum_{\lambda,\mu}A_{\lambda}\otimes (b_\mu- k'_\mu \basic)\otimes Y_{\lambda,\mu}\\
\label{symmetry}&
=
\sum_{\lambda,\mu} a_\lambda\otimes B_{\mu}\otimes Y_{\lambda,\mu}+
(-1)^{\deg \basic+1}\sum_{\lambda,\mu}A_{i}\otimes b_j\otimes Y_{\lambda,\mu}.
\end{align*}

It follows that the homology class $$[j_*{\rm EZ}_{M,M}(\sum_{\lambda,\mu}A_{\lambda}\otimes b_\mu \otimes Y_{\lambda,\mu})]\; \in\; C_*(M\times M,M\times M\setminus M_\Delta;\ZZ)$$ 
equals the homology class 
$$(-1)^{\deg \basic}[j_*{\rm EZ}_{M,M}(\sum_{\lambda,\mu} a_\lambda \otimes B_{\mu}\otimes Y_{\lambda,\mu})].$$

 Let $\tau:M\times M \to M\times M$ be the map 
swapping the two factors. Then
\begin{align*}
u(j_*{\rm EZ}_{M,M}(\sum_{\mu,\lambda} B_\mu \otimes a_{\lambda}\otimes Y_{\mu,\lambda}))&=\tau^*(u)(j_*{\rm EZ}_{M,M}(\sum_{\lambda,\mu} a_\lambda \otimes B_{\mu}\otimes Y_{\lambda,\mu}))\\
&=(-1)^{\dim M}u(j_*{\rm EZ}_{M,M}(\sum_{\lambda,\mu} a_\lambda \otimes B_{\mu}\otimes Y_{\lambda,\mu}))\\
&=(-1)^{\dim M + \deg \basic}u(\sum_{\lambda,\mu} A_\lambda\otimes b_{\mu}\otimes Y_{\lambda,\mu})
\end{align*}
and claim (i) follows.

That $\psi_\class Y$ is linear in the second argument is immediate from the definition.
Linearity in the first variable follows from linearity in the second variable and assertion (i).
This finishes the proof of assertion (ii).

Finally, a homotopy between allowable pairs can be approximated by a composition of homotopies, each of which fixes either $a$ or $b$, so to prove (iii) it is enough to show that such homotopies preserve $\psi_\class Y(a,b)$. This follows from Lemma~\ref{IndependentA}, together with the usual argument for homotopy invariance of homology. 
\end{proof}
\end{lemma}

\subsection{Application to  link invariants of divisor braids}

Let $(\Gamma,\kk)$ be a negative colour scheme on $r$ colours with $k_\lambda \ge 2$ for every colour $\lambda$; recall that we use the wording `very composite' to refer to this situation. We want to construct a presentation of the divisor braid group 
   $\DivBraid_\kk(\Sigma , \Gamma)$, where $\Sigma$ is the orientable Riemann surface. 
   
 Recall that Theorem~\ref{centralextension} established the existence of a central extension
\[
E_\kk(\Gamma)\to \DivBraid_\kk(\Sigma, \Gamma) \to H_1(\Sigma;\ZZ)^{\oplus r}
\]
where $E_\kk(\Gamma)$ is the subgroup generated by the homotopy classes $\beta_{\lambda,\mu}$ defined in (\ref{alphabeta}). 

An element $\alpha\in E_\kk (\Gamma)$ can be represented by $r$ sets of maps
$$\tilde  \alpha_{\lambda,j}: S^1\to M:=S^1\times \Sigma,\qquad j=1,\ldots, k_\lambda;$$ 
each such set defines a cycle $\tilde\alpha_\lambda $ in $M$,  representing the closed sub-braid of colour $\lambda$. 
Using the K\"u{}nneth theorem, we have a natural isomorphism $H_1(M;\ZZ)\cong H_1(S^1;\ZZ)\oplus H_1(\Sigma;\ZZ)$. 
The image of $[\tilde\alpha_\lambda]$ in the first factor is determined by the degree $k_\lambda$ of the divisor braid (or number of strands) in colour $\lambda$, that is, it equals $k_\lambda[S^1]$.
The image of 
$[\tilde\alpha_\lambda]$ in the second factor $H_1(\Sigma;\ZZ)$ is trivial, since each curve representing $\beta_{\lambda,\mu}$ yields the
trivial homology class in $H_1(\Sigma;\ZZ)$. 
It follows that the homology class of $[\tilde\alpha_\lambda]$ equals the homology class of $k_\lambda[S^1]\in H_1(S^1\times \Sigma;\ZZ)$, 
so that we can consider link invariants with respect to a basic chain 
$\basic \in C_1(M;\ZZ)$ representing $[S^1]$. 

\begin{defn} 
Let $\class Y$ be a $(\Gamma,\kk)$-allowable collection in an Abelian group $K$; then we define the {\em $\class Y$-linking} of the central divisor braid $\alpha \in E_\kk (\Gamma)$ to be
\[
\theta_{\class Y}(\alpha):=\psi_{\class Y}(\tilde \alpha,\tilde \alpha).
\] 
\end{defn}

The following result allows the calculation of $\class Y$-linkings with respect to a given $(\Gamma,\kk)$-allowed collection $\class Y$.

\begin{lemma}
\label{le:invariants}
 The quantity $\theta_\class Y(\alpha)$ is a well-defined element of $K$. The $\class Y$-linking yields a group  homomorphism
$$\theta_\class Y:E_\kk(\Gamma)\to K.$$ On the generators $\beta_{\lambda,\mu}$, it is given simply by the formula
\[
\theta_\class Y(\beta_{\lambda,\mu})=Y_{\lambda,\mu}.
\]
\begin{proof}
 If $\alpha$ and $\alpha'$ are homotopic braids, then the corresponding cycles $\tilde\alpha_\lambda$
and $\tilde\alpha'_\lambda$ are homotopic, through homotopies that
respect the underlying colour scheme $(\Gamma,\kk)$. It follows from Lemma~\ref{le:psi} that 
$\Psi_\class Y(\alpha)$ is homotopy invariant, 
meaning that it is well defined.

Consider braids  $\alpha',\alpha^{\prime\prime}$ representing elements in $E_\kk (\Gamma)$, and let 
$\alpha := \alpha'*\alpha^{\prime\prime}$ be their concatenation.
For each colour $\lambda$,   then $\tilde \alpha$ 
is the composition of two paths in $S^1\times \Sigma$, corresponding to
$\tilde \alpha'$ respectively $\tilde\alpha^{\prime\prime}$. 

It is convenient to introduce some normalisation in this situation. For instance, we can assume that the paths $\alpha'$ and
$\alpha''$ are constant paths close to $\{1\}\times \Sigma$, so that the composed path  
is constant close to $\{1\}\times \Sigma$ and $\{-1\}\times \Sigma$. There are  
2-chains $A'_\lambda,A''_\lambda$ such that $\partial A_\lambda' = \tilde \alpha_\lambda' - k_\lambda \basic$ and
$\partial A_\lambda'' = \tilde \alpha_\lambda'' - k_\lambda \basic$. We can further normalise these so that
$A_\lambda'$ and $A_\lambda''$ agree
close to $\{1\}\times \Sigma$. Then they patch up to a chain $A_\lambda$ in $S^1\times \Sigma$ 
whose boundary is $\tilde \alpha_\lambda - k_\lambda \basic$. With some abuse of notation, we can think of
$\tilde\alpha$ as the union of $\tilde\alpha_\lambda'$ and $\tilde \alpha_\lambda''$ over the common boundary, 
and similarly consider $A_\lambda$ as a union of $A'_\lambda$ and $A''_\lambda$ over their common boundary.
Figure~6 gives a schematic representation of this setup.

\begin{figure}[h]
\label{fig.flag}
\vspace{5mm}
\includegraphics[width=8cm,angle=0]{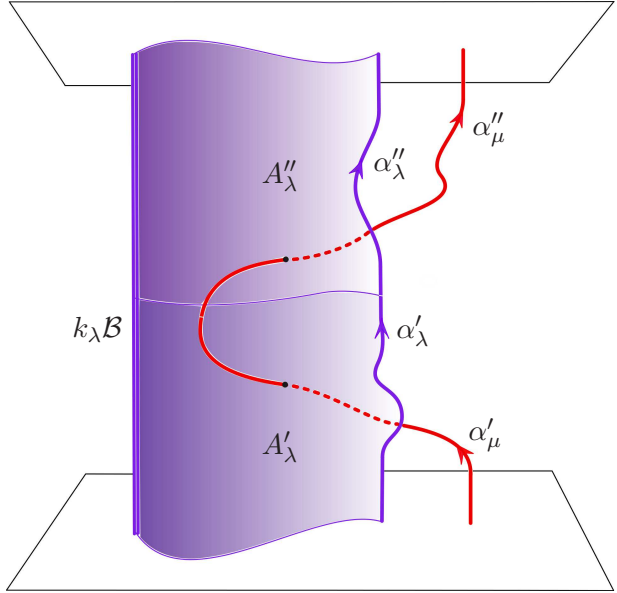} \\[-0pt] 
\vspace{-65mm}\hspace{47mm}$\alpha''_\mu$\\
\vspace{0mm}\hspace{21mm}$\alpha''_\lambda$\\
\vspace{-3mm}\hspace{-8mm}$A''_\lambda$\\
\vspace{16mm}\hspace{27mm}$\alpha'_\lambda$\\
\vspace{9mm}\hspace{46mm}$\alpha'_\mu$\\
\vspace{-3mm}\hspace{-8mm}$A'_\lambda$\\
\vspace{-20mm}\hspace{-56mm}$k_\lambda\basic$\\
\vspace{37mm}
\caption{The normalised setup in the proof of Lemma~\ref{le:invariants}.}
\end{figure}

We now need to compute the intersection number of $A_\lambda$ with $\tilde \alpha_\mu$. 
There will be no intersections between   $A^\prime_{\lambda}$ and $\alpha_\mu^{\prime\prime}$, or between   $A^{\prime\prime}_{\lambda}$ and $\alpha_{\mu}^{\prime}$. 
It follows that
\begin{align*}
\theta_\class Y(\alpha'*\alpha^{\prime\prime})&=
\psi_\class Y(\tilde \alpha,\tilde\alpha)\\
&=\psi_\class Y(\alpha',\alpha^{\prime})+\psi_\class Y(\alpha^{\prime\prime},\alpha^{\prime\prime})\\
&=\theta_\class Y(\alpha')+\theta_\class Y(\alpha^{\prime\prime}).
\end{align*}
This shows that $\theta_{\class Y}$ is a homomorphism.

To compute $\theta_{\class Y}(\beta_{\lambda,\mu})$, we first note that, for $x_\mu \in \Sigma$, the difference
$\tilde \beta_{\lambda,\mu}-x_\mu$ consists of the difference between a path of colour $\mu$ winding around
a constant path of colour $\lambda$ and a constant path of colour $\mu$. 
We chose $A_\lambda$ to be the chain described in 
Example~\ref{simpleloop}. Using the result in this example, we see that
the intersection between $A_\lambda$ and the constant path $x_\mu$
is the congruence class of 1. 
So we obtain
$$\theta_{\class Y}(\beta_{\lambda,\mu})=\psi_{\class Y}(\tilde \beta_{\lambda,\mu},\tilde \beta_{\lambda,\mu})=Y_{\lambda,\mu}$$
as claimed.
\end{proof}
\end{lemma}

For a given  $(\Gamma, \kk)$, let  $F_\kk(\Gamma)$ denote the free Abelian group generated by 
the symbols $b_{\lambda,\mu}$ for $ 1\leq \lambda,\mu \leq r$. Consider the subgroup $R_\kk(\Gamma)\subset F_\kk(\Gamma)$  generated by the following elements:
\begin{enumerate}
\item  $b_{\lambda,\mu}-b_{\mu,\lambda}$, $1\le \lambda < \mu \le r$;
\item $b_{\lambda,\mu}$ if there is not an edge between $\lambda$ and $\mu$ in $\Gamma$;
\item $P_\mu :=\sum_{\lambda \not=\mu} k_\lambda b_{\lambda,\mu}$ for $1\leq \mu \leq r$.
\end{enumerate}

We define a map $\lp': F_\kk(\Gamma)\to E_\kk(\Gamma) \subset {\sf DB}_\kk(\Sigma,\Gamma)$ 
by setting $\lp'(b_{\lambda,\mu})=\beta_{\mu,\nu}$ on generators.
From the definition of $E_\kk(\Gamma)$ given after Lemma~\ref{central}, it follows immediately that
this is a surjective homomorphism.

From Proposition~\ref{Erelations}, we know that
$\lp'$ factors over a  surjective quotient map 
\begin{equation} \label{lp}
\lp:F_\kk(\Gamma)/R_\kk(\Gamma)\to E_\kk(\Gamma).
\end{equation}

\begin{thm}
\label{thm:RationalIso}
  The map $\lp$ is an isomorphism.
  \begin{proof}

Since we already know that the map (\ref{lp}) is surjective, we only have
to check that it is injective. For this, it suffices to find
a right-inverse. 

Let us take $Y_{\lambda,\mu} = [b_{\lambda,\mu}]\in  F_\kk(\Gamma)/R_\kk(\Gamma)$. Then we see that
$\sum k_\lambda Y_{\lambda,\mu}=\sum k_\mu Y_{\lambda,\mu}=0$; moreover, if there is no edge in $\Gamma$
between $\lambda$ and $\mu$, then also $Y_{\lambda,\mu}=0$. Thus $\mathbf{Y}= \{ Y_{\lambda,\mu}\}_{\lambda,\mu}$ is a $(\Gamma,\kk)$-allowable $F_\kk(\Gamma)/R_\kk(\Gamma)$-collection, and it determines a $\class Y$-linking  $\theta_{\class Y}: 
E_\kk (\Gamma)\to  F_\kk (\Gamma)/R_\kk(\Gamma)$.
   
It follows from Lemma~\ref{le:invariants} that 
$\lp \circ \theta_{\class Y}$ is the identity.
\end{proof}
\end{thm}

We now sum up what we have proved so far.
Let $(\Gamma,\kk)$ be a very composite negative colour scheme, that is, $k_\lambda\geq 2$ for all colours $\lambda$. Let $\Sigma$ be a compact, oriented surface of genus $g$. 
We are going to describe the group $\DivBraid_\kk(\Sigma, \Gamma)$  using a simple set of generators and relations. There are two sets of generators.
For a pair of colours $1\leq \lambda,\mu\leq r$ we define a generator $b_{\lambda,\mu}$. Choose a set of generators $a_\ell$ ($1\leq \ell \leq 2g$) for
the free Abelian group $H_1(\Sigma;\ZZ)$. For each $\ell$  and each $\lambda$, we introduce a generator $a_{\lambda,\ell}$ (a copy of the 1-cycle $a_\ell$ in colour $\lambda$).
\begin{thm} \label{present}
  The divisor braid group  $\DivBraid_\kk(\Sigma, \Gamma)$  is isomorphic to the group generated by the symbols $b_{\mu,\lambda}$, $a_{\lambda,\ell}$ with  indices as described above, subject
to the following relations:
\begin{enumerate}
\item[(i)] $b_{\lambda,\mu}b_{\lambda',\mu'}=b_{\lambda',\mu'}b_{\lambda,\mu}$
\item[(ii)] $b_{\lambda,\mu}=e$ (the neutral element)  if there is no edge in $\Gamma$ connecting $\lambda$ and $\mu$
\item[(iii)] $b_{\lambda,\mu}=b_{\mu,\lambda}$
\item[(iv)] 
$
\prod_{\mu \not=\lambda} b_{\lambda,\mu}^{k_\mu}=e
$
\item[(v)] $b_{\lambda,_\mu}a_{\nu, \ell}=a_{\nu,\ell}b_{\lambda_,\mu}$
\item[(vi)] $a_{\lambda,\ell}a_{\mu,\ell'}a_{\lambda,\ell}^{-1}a_{\mu,\ell'}^{-1}
=b_{\lambda,\mu}^{\sharp(a_\ell,  a_{\ell'})}$. \label{lastrel} 
\end{enumerate} 
In the last relation, $\sharp (\cdot,\cdot)$ denotes the  intersection pairing in $H_1(\Sigma;\ZZ)$.
\begin{proof}

According to Theorem \ref{centralextension}, there is a central extension
\begin{equation}
 E_\kk(\Gamma) \to \DivBraid_\kk (\Sigma,\Gamma) \xrightarrow{h} H_1(\Sigma;\ZZ)^{\oplus r}.
\end{equation}
The group $H_1(\Sigma;\ZZ)^{\oplus r}$ is a free Abelian group with $2gr$ generators. Let $G$ denote the group presented in the statement of the theorem.
We define a homomorphism $\bar \lp:G\to \DivBraid_\kk (\Sigma,\Gamma)$ by setting
$$
\bar\lp(a_{\lambda, \ell})=\alpha_{\lambda,\ell} \quad \text{and} \quad
\bar\lp(b_{\lambda, \mu})=\beta_{\lambda,\mu}.
$$
We have to check that $\bar\lp $ respects the relations (i)--(vi). The first four relations are
respected according to Lemma \ref{Erelations}. 
According to Lemma \ref{central}, the fifth relation is satisfied.
By Lemma \ref{commutator}, the last relation is also satisfied.
We obtain a commutative diagram of not necessarily commutative groups with exact rows:
\[
\begin{CD}
@. F_\kk(\Gamma)/R_\kk(\Gamma)
@>f>> G @>>> G/\langle b_{\lambda,\mu}\rangle_{\lambda,\mu=1}^r @>>> 0\\
@. @V\lp VV @V{\bar \lp}VV @VVV @. \\
0 @>>> E_{(\kk)}(\Gamma)  
@>>> \DivBraid_{\kk}(\Sigma,\Gamma)
@>>> H_1(\Sigma;\ZZ)^{\oplus r} @>>> 0 
\end{CD}
\]
According to Theorem \ref{thm:RationalIso} the map $\lp$ is an isomorphism. From this and from the diagram it follows that the map $f$ is injective. It is easy to see that the right vertical map takes
\begin{align*}
  a_{\lambda,\ell} \mapsto (0,\dots,0,a_\ell,0,\dots,0), 
\end{align*}
with the $a_\ell$ inserted into the summand corresponding to the colour $\lambda$, and is thus an isomorphism. It now follows from the five-lemma that $\bar \lp$ is an isomorphism. 
\end{proof}
\end{thm}

\section{The centre of a very composite divisor braid group}
\label{sec:center}

According to Theorem~\ref{thm:RationalIso}, if each $k_\lambda\geq 2$ (i.e.\ the negative colour scheme $(\Gamma,\kk)$ is very composite),
the center $E_\kk(\Gamma)$ of a divisor braid group ${\sf DB}_\kk(\Sigma,\Gamma)$ is isomorphic to the Abelian group
\begin{equation} \label{groupD}
D_\kk(\Gamma):=F_\kk(\Gamma)/R_\kk(\Gamma),
\end{equation}
whose generators and relations were defined after Lemma~\ref{le:invariants}.
Note that this group is also defined if some $k_\lambda=1$. 
In this section, we study in more detail how the quotient (\ref{groupD}) depends on
$\Gamma$ and $\kk$. 

There are a few general statements. 
Consider maps of graphs $f:\Gamma' \to \Gamma$
with the property that if $\lambda,\mu\in {\rm Sk}^0(\Gamma')$ and there is no edge between
$f(\lambda)$ and $f(\mu)$ in $\Gamma$, then there is no edge between $\lambda$ and $\mu$ in $\Gamma'$.
In particular, this condition is satisfied for the inclusion of a subgraph $\Gamma' \hookrightarrow \Gamma$. Another example is provided by transformations of
graphs of the form
\begin{equation}\label{assgraph}
\Gamma = \neg \, {\rm Sk}^1((\partial \Delta)^\vee)
\end{equation}
(associated to a Delzant polytope $\Delta$) induced by the blow-up of
a fixed point of the torus action of $X=X_{{\sf Fan}_\Delta}$; we shall provide some concrete examples at the end of this section.
Recall that the assignment (\ref{assgraph}) is motivated by equation~(\ref{fundgroups}).

Given such a map $f$ and a negative colour scheme $(\Gamma,\kk)$, we can define an induced coloured degree
$\kk'= f^*(\kk)$ for the source graph $\Gamma'$ by setting $k'_\lambda = k_{f(\lambda)}$.
There is an induced homomorphism map $D_{\kk'}(\Gamma') \to D_\kk(\Gamma)$ given on generators by
$b_{\lambda,\mu}\mapsto b_{f(\lambda),f(\mu)}$. 

\begin{remark} \label{rk:connect}
Suppose that $\Gamma =\Gamma' \cup \Gamma''$ as a disjoint union.
Then $(\Gamma,\kk)$ determines $\kk'$ (respectively $\kk''$) by restriction 
to $\Gamma'$ (respectively $\Gamma''$), and it is easy to see that
the maps induced by the inclusions together define an isomorphism 
\begin{equation}
\label{eq:sum}
D_\kk(\Gamma) \cong D_{\kk'}(\Gamma') \oplus D_{\kk''}(\Gamma'').
\end{equation}
\end{remark}

From now on, and based on (\ref{eq:sum}), we shall restrict our attention to connected graphs $\Gamma$ unless explicitly stated.

\begin{example}
Here is a very concrete illustration of the presence of torsion in the divisor braid group. The configuration space of three labelled points on the two-sphere is homotopy equivalent to ${\rm SO}(3)$. 
To see this, use that three distinct points on $S^2$ are on a unique plane in $\RR^3$. In particular, the pure braid group ${\sf PB}_3(S^2)$ on three strands on $S^2$ is
isomorphic to $\ZZ_2$. It follows that if  $\Gamma$ is the complete graph with  $r=3$ vertices and the genus of $\Sigma$ is zero, then we have $\beta_{\lambda,\mu} = \beta$ for each pair of distinct colours, a relation
$2\beta=0\in D_{(1,1,1)}(\Gamma)$ and
${\sf DB}_{(1,1,1)}(\Sigma,\Gamma)\cong \ZZ_2$. The negative colour scheme in this example is not very composite, but $D_\kk(\Sigma)\cong E_\kk(\Sigma)$ still holds.
\end{example}

\subsection{The rank of the centre} \label{sec_rankcentre}

The free part of $D_\kk(\Gamma)$ is determined up to isomorphism by its rank. This equals the dimension over $\QQ$ of the rational
vector space $D_\kk(\Gamma)\otimes_\ZZ \QQ$.
This space can be recast as the $r^2$-dimensional $\QQ$-vector space $F_r:= F_\kk(\Gamma) \otimes_\ZZ \QQ$ (which in fact depends only on the number of colours $r$), generated by the symbols
$b_{\lambda,\mu}$ for $1 \leq \lambda ,\mu \leq r$, quotiented by the subspace $R_\kk(\Gamma) \otimes_\ZZ \QQ$  spanned by  the relations (1)--(3)
after Lemma~\ref{le:invariants}. It follows that
\begin{equation}\label{therank}
{\rm rk} \,  D_\kk (\Gamma) = r^2 - \dim_\QQ (R_\kk(\Gamma)\otimes_\ZZ \QQ).
\end{equation}

\begin{lemma}
\label{k-independence}
The rank of the Abelian group $D_\kk (\Gamma)$ does not depend on $\kk$.
\end{lemma}
 \begin{proof}
Let $\kk$ and $\kk'$ be two different coloured degrees associated to the same graph $\Gamma$ with $r$ vertices.
They define two different subspaces $R_\kk(\Gamma), R_{\kk'}(\Gamma)$ of the rational 
vector space $F_r$.
There exists an automorphism $\varphi$ of $F_r$ that restricts to
an isomorphism between these two subspaces; in particular
\[
\varphi(R_\kk (\Gamma)\otimes \QQ) =  R_{\kk'}(\Gamma)\otimes \QQ.
\]
This automorphism is given in the defining generators as
$$\varphi(b_{\lambda,\mu})=\frac {k_\lambda' k_\mu'}{k_\lambda k_\mu}\,  b_{\lambda,\mu}.$$
So we conclude from (\ref{therank}) that 
$${\rm rk}\, D_\kk (\Gamma) = {\rm rk} \, D_{\kk'}(\Gamma).$$
\end{proof}

We now set $D(\Gamma):=D_{\mathbf{1}}(\Gamma)$ where $\mathbf{1}$ is the coloured degree with $\mathbf{1}_\lambda = 1$ for all colours $\lambda \in {\rm Sk}^0(\Gamma)$.
The computation of the rank simplifies for this group and, by the previous lemma, it coincides with the rank of $D_\kk(\Gamma)$ for general $\kk$.

The remaining computation can be recast as a problem in cohomology as follows.
Let $\mathcal{C}^0(\Gamma)$ be the $\QQ$-vector space
generated by the vertices of $\Gamma$, and $\mathcal{C}^1(\Gamma)$
the $\QQ$-vector space generated by the edges of
$\Gamma$. There is a unique $\QQ$-linear map 
\begin{equation} \label{mapd}
d(\Gamma):\mathcal{C}^0(\Gamma)\to \mathcal{C}^1(\Gamma)
\end{equation}
associating to
a given vertex the sum of all the edges 
incident to that vertex. As a consequence of
Theorem~\ref{thm:RationalIso}, we have 
that the vector space
$D (\Gamma)\otimes_\ZZ \QQ$ is isomorphic
to the cokernel of this map (\ref{mapd}).

\begin{proposition}
Let $\Gamma$ be a connected graph with $r$ 
vertices and $s$ edges. Then
$$
\dim_\QQ \, \coker(d(\Gamma)) =
\left\{ 
\begin{array}{ll}
s-r+1 & \text{ if $\Gamma$ is bipartite;} \\
s-r & \text{ if $\Gamma$ is not bipartite.}
\end{array}
\right.
$$
\begin{proof}
 We first determine the dimension of the kernel of the map $d(\Gamma)$.
A class $\sum_{\lambda\in {\rm Sk}^0(\Gamma)}a_\lambda \lambda \in \mathcal{C}^0(\Gamma)$ 
is in this kernel  if and only if the following condition holds:
\begin{equation*} \label{*}
\text{If there is an edge between $\lambda$ and $\mu$, then $a_\lambda = -a_\mu$.}
\end{equation*}
Since $\Gamma$ is connected, the coefficients $a_\lambda$ all agree up to sign.
A nontrivial element of the kernel determines a bipartitioning of 
$\Gamma$, according to whether the $a_\lambda$ are positive or negative at the vertices $\lambda$. 
Conversely, a bipartitioning of $\Gamma$ defines a nontrivial
element of  $\ker (d(\Gamma))$ (up to a nonzero scalar). 

It follows that if $\Gamma$ is not bipartite, then
$\ker (d(\Gamma)) = 0$. If $\Gamma$ has a bipartitioning, it is unique,
hence $\dim\ker (d(\Gamma)) = 1$.
   
The proposition now follows from this and from the dimension formula.  
\end{proof}
\end{proposition}

\begin{corollary} \label{corollrank}
  Let $\Gamma$ be a graph with connected components $\Gamma_i$.
Let the number of vertices in $\Gamma$ be $r$, the number of edges 
of $\Gamma$ be $s$, and the number of bipartite components $\Gamma_i$ be
$t$. Then the rank of the group $D_\kk(\Gamma)$ is $s-r+t$.
\begin{proof}
If $\Gamma$ is connected, this is just the preceding proposition. 
In the general case, we reduce to the connected case using the sum formula~(\ref{eq:sum}).
\end{proof}
\end{corollary}

\begin{remark}\label{zerorank}
If $\Gamma$ is a connected graph, we observe that Corollary~\ref{corollrank} implies that the group $D_\kk(\Gamma)$ is a finite group (independently of $\kk$) if and only if either of the following conditions is met:
\begin{itemize}
\item  $\Gamma$ is a tree;
\item  $\Gamma$ contains a cycle of odd order, and if we remove an edge from that cycle, the remaining graph is a tree.
\end{itemize}
\end{remark}

\subsection{Torsion elements and Diophantine equations} \label{sec:diophantine}

In order to compute very composite divisor braid groups ${\sf DB}_\kk(\Sigma,\Gamma)$, we would like to determine the corresponding finitely generated Abelian groups  
$D_\kk(\Gamma) \cong E_\kk(\Gamma)$, i.e. their central commutators.
We have already found a simple formula for the rank of $D_\kk(\Gamma)$. To obtain the full abstract structure of the group, we still
have to determine its torsion subgroup ${\rm Tor} \, D_\kk (\Gamma)$. 

The group $D_\kk(\Gamma)$  has been defined using generators and relations in (\ref{groupD}), so determining its torsion subgroup
 is a purely algebraic question. This group is the cokernel of a map between
free Abelian groups given by a matrix $D$ with integer coefficients.
Given a particular negative colour scheme $(\Gamma, \kk)$, a computer can easily calculate the Smith form of $D$, which means that
it can determine its cokernel. One interesting question is how this   
cokernel changes when one fixes the graph $\Gamma$ and varies the weights.
We will make some preliminary remarks here, but 
the algebraic problem at hand turns out to have a surprisingly rich structure. We will return to this question in~\cite{Bok}.

It is convenient to consider a slightly simpler presentation of the group.
Let $F'(\Gamma)$ be the Abelian group generated by the symbols $v_{\lambda,\mu}$ for $1 \leq \lambda < \mu \leq r$,
with the relations that $v_{\lambda,\mu}=0$ if there is no edge in $\Gamma$ between the vertices $\lambda$ and $\mu$. 
Let 
$$P'_\mu := \sum_{\lambda<\mu}k_\lambda\, v_{\lambda,\mu} + \sum_{\mu<\lambda}k_\lambda \, v_{\mu,\lambda}$$
and let $R'_\kk(\Gamma)\subset F'(\Gamma)$
be the subgroup generated by the elements $P_\mu'$ for $1\le \mu \le r$.
Clearly, there is an isomorphism $$F(\Gamma)/ R_\kk(\Gamma)\;  \stackrel{\cong}{\longrightarrow}\;  F'(\Gamma)/R'_\kk(\Gamma)$$ given on generators by $b_{\lambda,\mu}\mapsto v_{\lambda,\mu}$.

Now we consider $t\in F'(\Gamma)$ such that the reduction $[t]\in F'(\Gamma)/R'_\kk(\Gamma)$ is a torsion class of order $m\in \ZZ$, i.e.\ $mt \in R'_\kk(\Gamma)$.
Let us denote the subgroup of all such elements by
\[
\bar R_\kk'(\Gamma) := \{t \in F'(\Gamma) \vert  
\text{ $\exists\,  m\in \NN$ such that $mt \in R'_\kk(\Gamma)$}\}. 
\]
It is easy to check that the inclusion $\bar R'_\kk(\Gamma)\subset F'(\Gamma)$ induces an isomorphism
from  the quotient $\bar R_\kk'(\Gamma)/R'_\kk(\Gamma)$
to the torsion subgroup of  $F'(\Gamma)/R_\kk'(\Gamma)$.

Let us set $\mathbf{c}=(c_\lambda)_{\lambda\in {\rm Sk}^0(\Gamma)}=(c_1,\ldots,c_r)$ and define
\begin{equation}
\label{dio}
C_\kk(\Gamma):=\{ \mathbf{c} \in  (\QQ/\ZZ)^{\oplus r} \vert\,  k_\lambda c_\mu + k_\mu c_\lambda  = 0 \text{ if there is an edge between $\lambda$ and $\mu$}\}.
\end{equation}

In the case where  $\Gamma$ is a bipartite graph, there is a distinguished cyclic subgroup $\Delta_\kk(\Gamma) \subset C_\kk (\Gamma)$ 
generated by a  solution of the form 
$ ([\pm \frac{1}{k_\lambda}])_{\lambda\in {\rm Sk}^0(\Gamma)}$,
where the signs are given by 
the bipartitioning. Evidently, the subgroup itself does not depend on the choice of overall sign, as $C_\kk (\Gamma)$ is pure torsion. The order of 
$\Delta_\kk(\Gamma)$ is the least common multiple of the integers $k_{\lambda}$.

We are now ready to present the main result of this section.

\begin{thm}
Let $(\Gamma,\kk)$ be a negative colour scheme on a connected graph. Then
$$
{\rm Tor}\, D_\kk (\Gamma) \cong
\left\{ 
\begin{array}{ll}
C_\kk(\Gamma)/\Delta_\kk(\Gamma) & \text{ if $\Gamma$ is bipartite;} \\
C_\kk(\Gamma) & \text{ if $\Gamma$ is not bipartite.}
\end{array}
\right.
$$
\begin{proof}  
We define a homomorphism $\phi:C_\kk(\Gamma)\to \bar R_\kk'(\Gamma)/R_\kk'(\Gamma)$ as follows.
Let $\mathbf{c}$ be an element of  $C_\kk(\Gamma)$, represented by some $r$-tuple $(\tilde c_1,\dots,\tilde c_r)\in\QQ^{\oplus r}$. Let us set 
$$\tilde \phi(\tilde c_1,\dots,\tilde c_r):=\sum_{\lambda=1}^r \tilde c_\lambda P'_\lambda \in R'_\kk(\Gamma)\otimes \QQ.$$
Actually, $\tilde \phi(\tilde c_1,\dots,\tilde c_r) \in R'_\kk(\Gamma)\subset R'_\kk(\Gamma)\otimes \QQ$. 
To see this, note that the coefficient in $\tilde \phi(\tilde c_1,\dots,\tilde c_r)$ 
of the edge connecting $\mu$ and $\lambda$ is
$k_\lambda \tilde c_\mu+k_\mu \tilde c_\lambda$, which is an integer, since
$\mathbf{c}\in C_\kk(\Gamma)$. 
We define $\phi(\mathbf{c})$ to be the equivalence class of 
$\tilde \phi(\tilde c_1,\dots,\tilde c_r)$ in $\bar R_\kk'(\Gamma)/R_\kk'(\Gamma)$.
We need to check that this does not depend on the choice of representative $(\tilde c_1,\dots,\tilde c_r)$.
But $(\tilde c_1,\dots, \tilde c_r)$ represents $0\in C_\kk(\Gamma)$ exactly when each $\tilde c_\lambda$
is an integer. Then 
\[
\tilde \phi(\tilde c_1,\dots,\tilde c_r)=\sum_\lambda\tilde c_\lambda P'_\lambda 
\]
is in $R_\kk'(\Gamma)$ by the definitions.

We now claim that the homomorphism $\phi$ is surjective. Let
$t \in \bar R'_\kk(\Gamma)$ with $mt \in R'_\kk(\Gamma)$, $m\in \NN$. 
One can write 
\[
t\otimes 1 =\sum_{\lambda=1}^r  P'_\lambda \otimes \frac{m_\lambda}{m} \in R'_\kk(\Gamma)\otimes_\ZZ \QQ
\]
for suitable $m_\lambda \in \ZZ$. We see that $t=\tilde \phi(\frac {m_1}m,\dots,\frac{m_r}m)$, and it follows that $\phi$ is surjective.

Now we want to determine the kernel of $\phi$. Let
$\phi(\mathbf{c})=0$. Choose a representative $(\tilde c_1,\dots,\tilde c_r)\in \QQ^r$
such that $\tilde \phi(\tilde c_1,\dots,\tilde c_r)=0$ (this can easily be arranged).
Then $\sum_\lambda c_\lambda P'_\lambda=0$ is a linear dependence of the 
classes $P_\lambda$ over $\QQ$. We can rewrite this dependence as
\[
\sum_{\lambda <\mu}(k_{\lambda} \tilde c_\mu+k_\mu\tilde c_\lambda)=0.
\] 
Equating coefficients, we get that 
\[
\frac {\tilde c_\lambda}{k_\lambda}+\frac {\tilde c_\mu}{k_\mu}=0.
\]
In particular, $\vert \frac {\tilde c_\lambda}{k_\lambda}\vert$ does not depend on $\lambda$.
The sign of $\tilde c_\lambda$ does depend on $\lambda$ in such a way that if $\lambda$ and
$\mu$ are connected by an edge, then $\tilde c_\lambda$ and $\tilde c_\mu$ have opposite signs. 
This determines a bipartitioning of $\Gamma$. Therefore, if $\Gamma$ is not bipartite,
we need to have $\tilde c_\lambda=0$. On the other hand, if $\Gamma$ has a bipartitioning, 
we see that the linear dependence
has to satisfy that $\tilde c_\lambda=\pm \frac s{k_\lambda}$ where the signs are determined by the bipartitioning.
In our case we have the further condition that $\tilde c_\lambda\in \ZZ$, which say that 
$\tilde c_\lambda=\pm \frac {ms'}{k_\lambda}$, where the signs are determined by the bipartitioning, 
and with $s'\in \ZZ$. This completes the proof.
\end{proof}
\end{thm}

The most obvious approach to the problem of determining the torsion in $D_\kk(\Gamma)$ 
consists of keeping the graph $\Gamma$ fixed, and varying the coloured degree $\kk$;
this is more in the spirit of the application motivated in Section~\ref{sec:vortices}.
In the next section, we will explore some examples
where it is possible to determine the group as a function of the integers $k_\lambda$. 

There is also the possibility of allowing the graph itself to change.
There is some naturality occurring in this setting. 

Let $f:\Gamma\to \Gamma'$ be a map of graphs, in the sense given in the remarks preceding Section~\ref{sec_rankcentre}.
Firstly, if $\kk$ is a coloured degree for $\Gamma$, it induces a coloured degree $f_*\kk$ on
$\Gamma'$ by the formula $(f_*\kk)_{\lambda'} = \sum_{f(\lambda)=\lambda'}k_\lambda$. There is a map
$C_\kk(\Gamma) \to C_{f_*\kk}(\Gamma')$ given by 
$\{c_\lambda\}_{\lambda\in {\rm Sk}^0(\Gamma)} \mapsto \{c'_{\lambda'}\}_{\lambda' \in {\rm Sk}^0(\Gamma')}$ where $c'_{\lambda'}=\sum_{f(\lambda)=\lambda'}c_\lambda$.

Similarly, recall that a coloured degree $\kk'$ on $\Gamma'$ induces a coloured degree
$f^*\kk'$ on $\Gamma$ where $(f^*\kk')_\lambda = \kk'_{f(\lambda)}$. Now there is a map
$C_{\kk'}(\Gamma')\to C_{f^*\kk'}(\Gamma)$
given by $\{c'_{\lambda'}\}_{\lambda'\in{\rm Sk}^0(\Gamma')} \mapsto \{c_\lambda\}_{\lambda \in {\rm Sk}^0(\Gamma)}$ where $c_\lambda = c_{f(\lambda)}$. 

One can use standard methods of homological algebra to relate groups corresponding to
different graphs, but in general it seems rather difficult to give explicit formulas for these groups.

\section{A handful of examples from gauge theory} \label{sec.examples}

So far, the divisor braid groups that we have been considering were assigned to a general negative colour scheme $(\Gamma,\kk)$ and a closed Riemann surface $\Sigma$. We will now discuss  a few examples coming from gauge theory --- i.e. divisor braid groups that are realised by the fundamental groups of
moduli spaces of vortices in compact fibre bundles with typical toric fibre $X$ over $\Sigma$, as described in Section~\ref{sec:vortices}.

\begin{example} \label{egP1}
We start by considering the only compact toric manifold in complex dimension one, $X=\PP^1$. This can be obtained from a one-dimensional Delzant polytope $\Delta$ with the shape of a real bounded closed interval. There are two facets in $\Delta$ corresponding to the endpoints, and they are disjoint. Thus we recover the graph $\Gamma = \bullet\!\! \!-\!\!\!-\!\circ$ already discussed in Section~\ref{sec:twocolours}. Theorem~\ref{twocolors} gave a metabelian presentation (\ref{centrext}) of the divisor braid group $\DivBraid_{(k_1,k_2)}(\Sigma, \bullet\!\! \!-\!\!\!-\!\circ)$, whose finite cyclic centre $\ZZ_{{\rm gcd}(k_1,k_2)}$ is determined from the integers $k_1,k_2>1$ through straightforward arithmetic.
\end{example}

\begin{example}\label{egHirz}

\begin{figure}[b]
\label{fig.cross}
\vspace{5mm}
\includegraphics[width=3cm,angle=0]{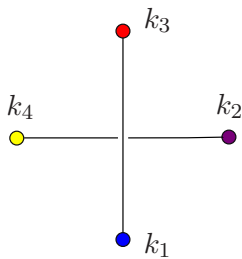} \\[-0pt] 
\vspace{-4mm}\hspace{9mm}$k_1$\\
\vspace{-23mm}\hspace{28mm}$k_2$\\
\vspace{-16mm}\hspace{9mm}$k_3$\\
\vspace{7mm}\hspace{-27mm}$k_4$\\
\vspace{18mm}
\caption{The graph $\Gamma$ for a Hirzebruch surface $X=\mathbb{F}_h$ is disconnected.}
\end{figure}

The simplest nontrivial example in complex dimension two corresponds to a Hirzebruch surface $X=\mathbb{F}_h$, where $h\in \NN_0$; its Delzant polytope $\Delta\subset N_\RR$ is a trapezium with right angles (possibly degenerating to a rectangle), see e.g.~\cite{CanSTM}. Accordingly, the graph $\Gamma$ consists of two edges, each of which connecting a pair of coloured vertices corresponding to opposite facets, as depicted in
Figure~7. As a consequence of Remark~\ref{rk:connect}, we obtain a factorisation
$$
\DivBraid_{(k_1,k_2,k_3,k_4)} (\Sigma,\Gamma) \cong \DivBraid_{(k_1,k_3)} (\Sigma,\bullet\!\! \!-\!\!\!-\!\circ) \times  \DivBraid_{(k_2,k_4) }(\Sigma,\bullet\!\! \!-\!\!\!-\!\circ)
$$
where the factors are the same divisor braid groups of Example~\ref{egP1}, following a corresponding factorisation of the group centres. In the case $\mathbb{F}_0=\PP^1 \times \PP^1$, this result also follows from a factorisation of moduli spaces implied by the explicit description in Theorem~\ref{conjecture}, Proposition~\ref{onlySk1} and functoriality of $\pi_1$; but here we see that the divisor braid groups also factorise  when the target is a nontrivial $\PP^1$-bundle over $\PP^1$ (i.e.~$h\ge 1$). This example also illustrates that disconnected graphs $\Gamma$ may well arise from connected  targets $X$ that are not Cartesian products.

\end{example}

\begin{example} \label{egBlHirz}
Consider a toric surface $X={\rm Bl}_{\mathbf{x}} \, \mathbb{F}_h$  obtained by blowing up a Hirzebruch surface (discussed in the previous example) at any of the four fixed points ${\mathbf x} \in X$ of the 2-torus action. For a discussion of the symplectic blow-up at a fixed point, see e.g.~\cite{CanSTM}; under this operation, the Delzant polytope is chopped transversally near the vertex corresponding to the chosen fixed point, so in this way we obtain an irregular pentagon in our example. 

\begin{figure}[h]
\label{fig.pentagram}
\vspace{5mm}
\includegraphics[width=3cm,angle=0]{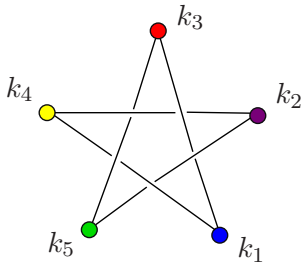} \\[-0pt] 
\vspace{-3mm}\hspace{26mm}$k_1$\\
\vspace{-25mm}\hspace{36mm}$k_2$\\
\vspace{-15mm}\hspace{10mm}$k_3$\\
\vspace{5mm}\hspace{-35mm}$k_4$\\
\vspace{16mm}\hspace{-24mm}$k_5$\\
\vspace{0mm}
\caption{The graph $\Gamma$ for the surface $X={\rm Bl}_{\mathbf x}\, \mathbb{F}_h$ is a five-cycle.}
\end{figure}

The corresponding graph $\Gamma$ has five vertices (each corresponding to a side of the pentagon), which should be labelled with different colours; introducing edges to account for empty intersections of pairs of facets, a five-cycle emerges, which we prefer to depict as a pentagram like the one in Figure~8. The circular rainbow scale that we used to colour the vertices suggests the cyclic ordering of facets in the Delzant polytope, while the disposition of vertices and edges renders the comparison with Figure~7 more immediate (the green vertex corresponds to the facet created by the blow-up). Unlike Examples~\ref{egP1} and \ref{egHirz}, this graph is not bipartite; but since its vertices and edges are equal in number, by
 Corollary~\ref{corollrank} it gives rise to a finite centre as in the previous examples.
 
 The order of the centre $E_{\kk}(\Gamma)=D_\kk(\Gamma)$ in this case is given by the product $2\,  {\rm gcd}(k_1,k_2,k_3,k_4,k_5) \prod_{\lambda=1}^5 k_\lambda$. However, the precise structure of this group is quite intricate, bifurcating over certain conditions that involve valuations. 
 We illustrate this with some examples in the following table. \\

 \begin{center}
 \begin{tabular}{c|c}
 \hline\hline \\[-7pt]
$ \kk=(k_1,k_2,k_3,k_4,k_5)$ & $D_\kk(\Gamma)$ \\ \\[-7pt] \hline \\[-8pt]
$(2,3,5,7,11)$ & $ \ZZ_{2^2\cdot 3\cdot 5\cdot 7 \cdot 11}$ \\
$(2,2,2,3,3)$ & $\ZZ_{2^3\cdot 3^2} \oplus \ZZ_{2}$  \\
$(2,2,3,2,3)$ & $\ZZ_{2^3\cdot 3}\oplus \ZZ_{2\cdot 3}$  \\
$(3,3,2,3,2)$ & $ \ZZ_{2^2\cdot 3^2} \oplus \ZZ_{2\cdot 3}$ \\
$(3,3,5,3,5)$ & $ \qquad \ZZ_{2\cdot 3^2\cdot 5} \oplus \ZZ_{3\cdot 5}\qquad $ \\
$(a,a,a,a,a), \; a\ge 2$ & $\ZZ_{2\cdot a} \oplus  \ZZ_a^{\oplus 4} $ \\[3pt]
\hline\hline
 \end{tabular}
 \end{center}
 \end{example}

\begin{example} \label{egBlcube}
Let us take the toric threefold $X={\rm Bl}_{\mathbf{x}} \, (\PP^1)^3$ obtained by blowing up the Cartesian product of three projective lines at any of the eight fixed points.

\begin{figure}[h]
\label{fig.tree}
\vspace{5mm}
\includegraphics[width=3cm,angle=0]{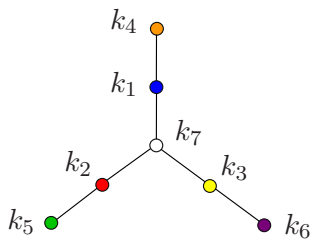} \\[-0pt] 
\vspace{-17mm}\hspace{8mm}$k_7$\\
\vspace{-11mm}\hspace{-9mm}$k_1$\\
\vspace{6mm}\hspace{-21mm}$k_2$\\
\vspace{-4mm}\hspace{20mm}$k_3$\\
\vspace{-24mm}\hspace{-9mm}$k_4$\\
\vspace{22mm}\hspace{-36mm}$k_5$\\
\vspace{-4mm}\hspace{38mm}$k_6$
\vspace{2mm}
\caption{The graph $\Gamma$ for the threefold $X={\rm Bl}_{\mathbf x}\, (\PP^1)^3$ is a tree.}
\end{figure}

It is easy to see that the corresponding graph $\Gamma$ is the tree of three branches depicted in Figure~9. The trivalent white vertex corresponds to the facet introduced by the blow-up. This graph is bipartite and the centre is once again a finite group.

For a coloured degree $\kk$ with components as indicated in Figure~9,
the order of this group is $(k_1k_2k_3)(k_7)^2 {\rm gcd}(k_1,k_2,k_3,k_4,k_5,k_6,k_7)$.  Again, resolving its structure  is an intricate problem, for which some solutions are given in the table below.  \\

\begin{center}
 \begin{tabular}{c|c}
 \hline\hline \\[-7pt]
$ \kk=(k_1,k_2,k_3,k_4,k_5,k_6,k_7)$ & $ D_\kk(\Gamma)$ \\ \\[-7pt] \hline \\[-8pt]
$(2,3,5,7,11,13,17)$  & $ \qquad \ZZ_{2\cdot 3\cdot 5\cdot 17}\oplus \ZZ_{17}\qquad $ \\
$(2,3,3,3,3,3,5)$ & $\ZZ_{2\cdot 3\cdot 5} \oplus \ZZ_{3\cdot 5}$  \\
$(2,2,3,3,3,3,3)$ & $\ZZ_{2\cdot 3}^{\oplus 2}\oplus \ZZ_{3}$ \\
 $(2,2,2,3,3,3,3)$ &  $\ZZ_{2\cdot 3}^{\oplus 2}\oplus  \ZZ_{3}$ \\
 $(2,2,2^2,2^2,2^2,2^2,2^2)$ &  $\ZZ_{2^2}^{\oplus 3} \oplus \ZZ_{2}^{\oplus 3}$\\ 
$(2^3,2,3,3,3,3,2^5)$& $ \ZZ_{2^8\cdot 3} \oplus \ZZ_{2^6}$   \\
$(2^3,2^3,3,3,3,3,2^5)$& $ \ZZ_{2^8\cdot 3} \oplus \ZZ_{2^8}$   \\
 $(2^3,2^5,3,3,3,3,2^5)$& $ \ZZ_{2^{10}\cdot 3} \oplus \ZZ_{2^8}$ \\
$(a,a,a,a,a,a,a),\; a\ge 2$ & $ \ZZ_a^{\oplus 6} $ \\[3pt]
\hline\hline
 \end{tabular}
 \end{center}

\end{example}

\begin{example} \label{egBl2cube}
Finally, we consider targets of the type $X={\rm Bl}_{\{ \mathbf{x}, \mathbf{x'} \}}\, (\PP^1)^3$, obtained by blowing up an additional fixed point $\mathbf{x}' \ne \mathbf{x}$ of the obvious $(S^1)^3$-action in the previous example. There are three distinct cases to consider, which are depicted in Figure~10. The colouring of the vertices was chosen so as to  facilitate comparison with the graph of Figure~9; in each graph, the facet introduced by the second blow-up is represented by the black vertex. Note that all these are connected graphs with eight vertices and ten edges, hence they give rise to divisor braid groups with infinite centres, in contrast with all the previous examples.
\begin{figure}[t]
\label{fig.kites}
\vspace{5mm}
\includegraphics[width=11cm,angle=0]{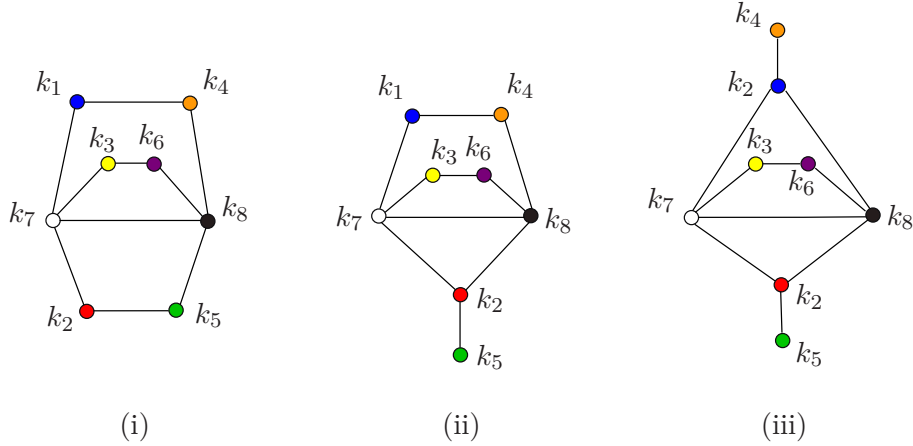} \\[-0pt] 
\vspace{-41mm}\hspace{-65mm}$k_4$\\
\vspace{13mm}\hspace{-116mm}$k_7$\\
\vspace{-22mm}\hspace{-109mm}$k_1$\\
\vspace{3mm}\hspace{-95mm}$k_3$\\
\vspace{-5mm}\hspace{-82mm}$k_6$\\
\vspace{5mm}\hspace{-60mm}$k_8$\\
\vspace{10mm}\hspace{-106mm}$k_2$\\
\vspace{-5mm}\hspace{-67mm}$k_5$\\

\vspace{-35mm}\hspace{15mm}$k_4$\\
\vspace{13mm}\hspace{-30mm}$k_7$\\
\vspace{-22mm}\hspace{-19mm}$k_1$\\
\vspace{4mm}\hspace{-5mm}$k_3$\\
\vspace{-5mm}\hspace{4mm}$k_6$\\
\vspace{5mm}\hspace{25mm}$k_8$\\
\vspace{6mm}\hspace{7mm}$k_2$\\
\vspace{3mm}\hspace{7mm}$k_5$\\

\vspace{-50mm}\hspace{75mm}$k_4$\\
\vspace{20mm}\hspace{52mm}$k_7$\\
\vspace{-20mm}\hspace{73mm}$k_2$\\
\vspace{3mm}\hspace{78mm}$k_3$\\
\vspace{1mm}\hspace{89mm}$k_6$\\
\vspace{0mm}\hspace{115mm}$k_8$\\
\vspace{6mm}\hspace{91mm}$k_2$\\
\vspace{3mm}\hspace{91mm}$k_5$\\
\vspace{5mm}
{\rm (i)} \hspace{102pt} {\rm (ii)} \hspace{101pt}{\rm (iii)}
\caption{The three graphs $\Gamma$ for $X={\rm Bl}_{\{\mathbf{x},\mathbf{x}'\} }(\PP^1)^3$: (i) is bipartite, whereas (ii) and (iii) are not.}
\end{figure}

Case (i) corresponds to choosing fixed points $\mathbf x$ and $\mathbf x'$ associated to diagonally opposite vertices in the Delzant polytope of $(\PP^1)^3$ (a cube). This graph is bipartite and determines a centre of rank three in the divisor braid groups.
The graphs (ii) and (iii) arise from choosing fixed points corresponding to vertices on a the same facet of the cube, either diagonally opposite or in the same boundary edge, respectively. Neither of them is  bipartite, and so they give rise to centres of rank two. 

The torsion groups are even more complicated now, and even presenting a formula for their order would be fairly complicated.  We restrict ourselves to a few illustrative examples in the
following table:

\begin{center}
\begin{tabular}{c|ccc}
\hline\hline \\[-7pt]
$\kk=(k_1,k_2,\ldots,
k_8)$ 
& $  \mathrm{Tor}\, D_\kk(\Gamma_{({\rm i})}) \!\!\!\!\!\!\!$ 
& $ \mathrm{Tor}\, D_\kk(\Gamma_{({\rm ii})})$  
& $ \mathrm{Tor}\, D_\kk(\Gamma_{({\rm iii})})$ \\ \\[-7pt] \hline \\[-8pt]
$(3,5,7,11,13,17,19,23)$
& $ \{0\}$ 
& $ \ZZ_{2\cdot 5\cdot 19\cdot 23}\oplus \ZZ_5$ 
& $ \ZZ_{2\cdot 3\cdot 5\cdot 19\cdot 23}$\\
$(2,3,3,3,3,3,5,7)$ & 
$\{0\} $
&$ \ZZ_{2\cdot 3\cdot 5\cdot 7}\oplus \ZZ_3$
&$\ZZ_{2\cdot 3\cdot 5\cdot 7} \oplus \ZZ_2$ \\[3pt]
$(2,2,3,3,3,3,3,5)$ 
& $\ZZ_3^{\oplus 3} $
&$  \ZZ_{2^2\cdot 3\cdot 5}\oplus \ZZ_{2\cdot 3}\oplus \ZZ_3 $
&$ \ZZ_{2^2\cdot 3\cdot 5}\oplus  \ZZ_{2\cdot 3}\oplus  \ZZ_2 $ \\[3pt]
$(2,2,2,3,3,3,3,5)$ 
& $\ZZ_3^{\oplus 3} $
&$ \ZZ_{2^2\cdot 3\cdot 5}\oplus \ZZ_{2\cdot3} \oplus \ZZ_3 $
&$\ZZ_{2^2\cdot 3\cdot 5}\oplus \ZZ_{2\cdot  3} \oplus \ZZ_2 $ \\[3pt]
$(3,3,3,5,5,5,5,7)$
& $\ZZ_5^{\oplus 3} $
&$ \ZZ_{2\cdot 3\cdot 5\cdot 7}\oplus \ZZ_{3\cdot5} \oplus \ZZ_5 $
&$ \ZZ_{2\cdot 3\cdot 5\cdot 7}\oplus \ZZ_{3\cdot5} \oplus \ZZ_3 $\\[3pt]
$(2,2,2,2^2,2^2,2^2,2^2,3)$ 
& $ \ZZ_{2^2}^{\oplus 3} $
&$ \ZZ_{2^4\cdot 3}\oplus \ZZ_{2^2}^{\oplus 2} \oplus \ZZ_2 $
&$ \ZZ_{2^4\cdot 3}\oplus \ZZ_{2^2} \oplus \ZZ_2^{\oplus 2} $ \\[3pt]
$(2^3,2,3,3,3,3,2^5,5)$ 
& $ \{0\}$
&$ \ZZ_{2^7\cdot 5}\oplus \ZZ_2 $
&$ \ZZ_{2^7\cdot 5}\oplus \ZZ_{2^3} \oplus \ZZ_{2} $ \\[3pt]
$(2^3,2^3,3,3,3,3,2^5,5)$ 
& $ \{0\}$
&$ \ZZ_{2^9\cdot 5}\oplus \ZZ_{2^3} $
&$ \ZZ_{2^9\cdot 5}\oplus \ZZ_{2^3}^{\oplus 2}$ \\[3pt]
$(2^3,2^5,3,3,3,2^5,5)$ 
& $ \{0\}$
&$ \ZZ_{2^{11}\cdot 5}\oplus \ZZ_{2^5} $
&$ \ZZ_{2^9\cdot 5}\oplus \ZZ_{2^5} \oplus \ZZ_{2^3} $ \\[3pt]
\hline\hline
 \end{tabular}
 \end{center}

From this table, we infer that it is possible to  distinguish among the three toric target manifolds through the 
order of the torsion subgroup of the centre alone.

\end{example}

\section{Divisor braids and noncommutative geometry} \label{sec:ncg}

In this final section, we want to provide a first glimpse at what the problem of computing Murray--von Neumann dimensions (that we motivated through
the discussion of two-dimensional supersymmetric QFTs in Section~\ref{sec:susy}) might look like for the divisor braid groups studied in this paper. 
The first step in this problem is to obtain a good description of the group von Neumann algebras
$\mathcal{N}(\pi_1\, \mathcal{M}^{X}_{\mathbf{h}}(\Sigma))$. We shall argue that, for very composite BPS charges $\mathbf{h}$, this can be given in terms
of objects that are familiar from noncommutative geometry~\cite{Con,GVF}.

We shall describe the von Neumann algebras $\mathcal{N}({\sf DB}_\kk (\Sigma,\Gamma))$ of our divisor braid groups (assumed to be very composite) in three stages. First, we deal with the two-colour case
$\Gamma =  \bullet\!\! \!-\!\!\!-\!\circ$; if we let
$k_1=\kk(\bullet)$ and $k_2=\kk(\circ)$, we already know that $E_{\kk}( \bullet\!\! \!-\!\!\!-\!\circ)=\ZZ_k$ with $k:={\rm gcd}(k_1,k_2)$ from our discussion in Section~\ref{sec:twocolours}, and this will lead to
totally explicit formulas. In a second stage, we consider the case where $\Gamma$ is more general but still leads to a finite group $E_\kk(\Gamma)$. Finally, we describe the general
situation where $E_\kk(\Gamma)$ may be infinite.

Consider the case  $\Gamma =  \bullet\!\! \!-\!\!\!-\!\circ$, and let us simplify the notation in Theorem~\ref{present} by denoting $\beta=\beta_{1,2}$. We shall write $R:=\CC[{\sf DB}_\kk(\Sigma, \bullet\!\! \!-\!\!\!-\!\circ)]$
for the group ring, and $\bar R=C^*({\sf DB}_\kk(\Sigma, \bullet\!\! \!-\!\!\!-\!\circ))$ for the reduced group $C^*$-algebra, that is, the completion of $R$ with respect to the operator norm of bounded operators on
the Hilbert space $\ell^2({\sf DB}_\kk(\Sigma, \bullet\!\! \!-\!\!\!-\!\circ))$ where ${\sf DB}_\kk(\Sigma, \bullet\!\! \!-\!\!\!-\!\circ)$ acts via the left-regular representation $\pi_L$ (see e.g.~Section 2.6 of \cite{BroOza}). Let  $\zeta$ be a fixed primitive $k$th root of unity.

\begin{lemma} \label{projectors}
The central elements in $R$ given by
\begin{equation} \label{projector}
\beta_j := \frac{1}{k} \sum_{i=0}^{k-1} \bar \zeta^{ij} \beta^i, \qquad j=0,\ldots, k-1
\end{equation}
are idempotent (i.e.~$(\beta_j)^2 = \beta_j$),  and they satisfy
\begin{equation*} 
\beta_j \beta_{j'} = 0 \; \text{ for } j\ne j' \qquad \text{ and } \qquad
\sum_{j=0}^{k-1} \beta_j = e.
\end{equation*}
\begin{proof}
The action of the centre $E_{\kk} (\Sigma, \bullet\!\! \!-\!\!\!-\!\circ) \cong \ZZ_k$ (generated by $\beta$) on its group algebra yields a $k$-dimensional representation which splits into cyclic summands, 
and (\ref{projector}) is a formula for the projector onto the summand where $\beta$ acts with character $\zeta^j$ (see e.g.~$\S$2.4 in~\cite{FulHar}). We observe that all the formulas
in the lemma are standard properties of these projectors.
\end{proof}
\end{lemma}

For each $0\le j\le k-1 $, we have a surjective ring homomorphism $m_{\beta_j}: R \rightarrow \beta_j R$ given by multiplication $ \gamma \mapsto \beta_j \gamma$; this follows from the idempotency of $\beta_j$, which plays the role of unit in the target ring $ \beta_j R$.  Lemma~\ref{projectors}  shows that  these $k$ homomorphisms fit together to provide an isomorphism to the product ring
\begin{equation} \label{prodring}
R \stackrel{\cong}{\longrightarrow} \prod_{i=0}^{k-1} (\beta_j R).
\end{equation}

We  are interested in the composed ring extension
\begin{equation} \label{extension}
R \hookrightarrow \bar R \hookrightarrow \bar R'' =: \mathcal{N}({\sf DB}_\kk (\Sigma,\Gamma)) =: \mathcal{N},
\end{equation}
corresponding to a completion with respect to the weak operator topology. The isomorphism (\ref{prodring}) induces a splitting $\mathcal{N} \cong \prod_{j=0}^{k-1} \mathcal{N}_j$, where each $\mathcal {N}_j$ can
be understood as the bicommutant of the target $C^*$-algebra
$\beta_j \bar R \subset \bar R$ inside $\mathcal{B}(\ell^2({\sf DB}_\kk(\Sigma, \bullet\!\! \!-\!\!\!-\!\circ)))$. These algebras also have the following description.

\begin{proposition} \label{vNcyclic}
Each factor $\mathcal{N}_j$ with $j\ne 0$ is the enveloping von Neumann algebra of a noncommutative torus of dimension $4g$, generated by unitaries $\UU_{\lambda,\ell}(j)$  (for $\lambda=1,2$ and $\ell = 1,\ldots, 2g$) satisfying the relations
\begin{equation}\label{nctorus}
\UU_{\lambda,\ell} (j) \, \UU_{\lambda',\ell'} (j)  = \exp\left(2 \pi {\rm i} \, \vartheta^{\lambda,\ell}_{\lambda',\ell'}(j) \right)\,  \UU_{\lambda',\ell'} (j)  \, \UU_{\lambda,\ell} (j) ,
\end{equation}
where 
\begin{equation} \label{nctheta}
\vartheta^{\lambda,\ell}_{\lambda',\ell'}(j) = \frac{j}{k} (1-\delta_{\lambda,\lambda'}) J_{\ell, \ell'};
\end{equation}
$\delta$ is the Kronecker delta and $J = \left[ \begin{array}{cc} {\mathbf 0} & {\mathbf 1} \\ -{\mathbf 1} & {\mathbf 0} \end{array} \right]$ the standard $2g \times 2g$ symplectic matrix.
The factor $\mathcal{N}_0$ is the commutative von Neumann algebra
\begin{equation} \label{N0}
\mathcal{N}_0 \cong \mathcal{N}(H_1(\Sigma;\ZZ)^{\oplus 2})\cong L^\infty({\rm Pic}^0(\Sigma)^{\oplus 2}).
\end{equation}

\begin{proof}

Fixing $j\ne 0$, we will show that $\beta_j \bar R$ realises the noncommutative torus in the statement of the theorem (see e.g.~\cite{Rie} or \cite{GVF} for the definition in terms of presentations of $C^*$-algebras); then the statement about $\mathcal{N}_j$ follows from the uniqueness of the enveloping von Neumann algebra (see~\cite{DixCsA}, 12.1.5).

A consequence of Theorem \ref{centralextension} is that the group ${\sf DB}_\kk (\Sigma,\Gamma)$ is amenable (since both $E_\kk(\Sigma)$ and $H_1(\Sigma;\ZZ)^{\oplus r}$ are); see Section 2.6 in~\cite{BroOza}. Amenability is equivalent (by Theorem 2.6.8 in~\cite{BroOza}) to the condition that the reduced group $C^*$-algebra $ \bar R$
is isomorphic to  the universal group $C^*$-algebra $\tilde R$ obtained by completing $R$ with respect to the norm $\| \gamma\|_{\rm univ}: =\sup_{u}\|u(\gamma)\|_{\mathcal{B}(H_u)}$, where the supremum is taken with respect to cyclic $*$-representations $u$ of $R$ (uniquely extending unitary representations of ${\sf DB}_\kk (\Sigma,  \bullet\!\! \!-\!\!\!-\!\circ )$) on some Hilbert space $H_u$. 

Suppose that a unitary representation of  
${\sf DB}_\kk (\Sigma,  \bullet\!\! \!-\!\!\!-\!\circ )$ on some Hilbert space $\mathcal{H}$ is given such that the induced $*$-homomorphism $\pi: R \rightarrow \mathcal{B}(\mathcal{H})$ factors through 
$m_{\beta_j}:R\rightarrow \beta_j R$, as in the bottom row of the diagram
\[
\xymatrix{
\bar R  \ar[r]^{m_{\beta_j}} 
& \beta_\chi \bar R 
\ar@{.>}[rd]^{\bar \rho} &\\
R \ar[r]^{m_{\beta_j}} \ar@{^{(}->}[u]^{\bar \iota}&\beta_\chi R \ar@{^{(}->}[u] \ar[r]^\rho& \mathcal{B}(\mathcal{H})
}
\]
The vertical arrows correspond to the completion to reduced group $C^*$-algebras, but since $\bar R \cong \tilde R$, we obtain by universality (see Proposition 2.5.2 of~\cite{BroOza}) a $*$-homomorphism $\bar\pi:\bar R \rightarrow \mathcal{B}(\mathcal{H})$ lifting the given
$\pi=\rho\circ m_{\beta_j}$, i.e. $\bar\pi \circ \bar\iota= \pi$. Now we claim that this $\bar\pi$ factors through a map $\bar{\rho}$ corresponding to the dotted arrow in the diagram. To see this, we observe that taking any $0\le j' \le k-1$ with $j'\ne j$ leads to
$$
\bar\pi(\beta_{j'} \bar \gamma) = \bar\pi(\beta_{j'}) \bar \pi(\bar \gamma)=  \pi(\beta_{j'}) \bar \pi(\bar \gamma) = 0 \,   \pi(\bar \gamma) =0
$$
for any $\bar \gamma \in \bar R$, and refer to Lemma~\ref{projectors}.

To establish that $\beta_j \bar R$ is the noncommutative torus in the statement, it remains to check that the relations (\ref{nctorus}) and (\ref{nctheta}) match up with the presentation
of the divisor braid group given in Theorem~\ref{present}. Under an identification of unitaries $\beta_j \alpha_{\lambda,\ell} \mapsto \UU_{\lambda,\ell}(j)$, one has
\begin{eqnarray*}
\beta_j \alpha_{\lambda,\ell}^{-1} &=& \frac{1}{k} \sum_{i=1}^k \bar \zeta^{ij} \beta^i \alpha_{\lambda, \ell}^{-1} = \frac{1}{k} \sum_{i=1}^k\bar  \zeta^{(k-i)j}  \beta^{(k-i)} \alpha_{\lambda, \ell}^{-1} \\
&=& \frac{1}{k} \sum_{i=1}^k \bar \zeta^{-ij} \alpha_{\lambda, \ell}^{-1} (\beta^{i})^{-1}  =  \frac{1}{k} \sum_{i=1}^k \bar \zeta^{-ij} (\beta^i \alpha_{\lambda, \ell} )^* \\
&= & (\beta_j \alpha_{\lambda, \ell})^* \mapsto \UU_{\lambda,\ell}^*,
\end{eqnarray*}
\begin{eqnarray*}
\beta_j &=& (\beta_j)^4 = (\beta_j)^4 [\alpha_{1,1}, \alpha_{1,1}] = \beta_j \alpha_{1,1} \beta_j  \alpha_{1,1} \beta_j  \alpha_{1,1}^{-1} \beta_j  \alpha_{1,1}^{-1} \\
&\mapsto & (\UU_{1,1} (j) )^2 (\UU^*_{1,1} (j) )^2 = \mathbb{I}
\end{eqnarray*}
and more generally
\begin{eqnarray*}
\beta_j \beta^l &= & \frac{1}{k}\sum_{i=1}^k \bar  \zeta^{ij} \beta^{i+l}  = \zeta^{jl} \frac{1}{k} \sum_{i=1+l}^{k+l} \bar \zeta^{ij} \beta^i \\
&=& \zeta^{jl} \beta_j \mapsto  \zeta^{jl} \, \mathbb{I}.
\end{eqnarray*}
Thus such an identification extends to a $*$-homomorphism from $\beta_j \bar R$ to any representation of the noncommutative torus with $4g$ unitary generators
$ \UU_{\lambda,\ell}(j)$ and relations as required.

The same argument can be followed through in the case $j=0$, but now all the generators $\UU_{\lambda,\ell}(0)$ commute; thus they can be modelled on multiplication by the coordinate functions of a $4g$-dimensional
torus, acting on the space of all bounded functions on that torus (identified whenever they agree almost everywhere). The torus itself can be naturally interpreted as a product of two copies of the complex torus ${\rm Pic}^0(\Sigma)$ of topologically trivial holomorphic line bundles (or flat ${\rm U}(1)$-connections) on the Riemann surface $\Sigma$. The isomorphisms in
(\ref{N0}) are standard from Fourier analysis.

\end{proof}
\end{proposition}

It is clear that the statement in this proposition generalises in an almost obvious fashion to any very composite 
divisor braid group ${\sf DB}_\kk(\Sigma,\Gamma)$ as in Theorem~\ref{present}, if we assume that this group has finite centre. If $r=|{\rm Sk}^0(\Gamma)|$ is the number of colours as before, then the associated von Neumann algebra $\mathcal{N}({\sf DB}_\kk(\Sigma,\Gamma))$ is a finite product of $2gr$-dimensional  noncommutative tori  $\mathcal{N}_\chi$ over the group of characters $E_\kk(\Gamma)^*$ of the centre, but with a commutative factor  $\mathcal{N}_0 = L^\infty({\rm Pic}^0(\Sigma)^{\oplus r})$ over the trivial character $\chi=0$; 
all these tori $\mathcal{N}_\chi$ (as many as the order of the centre) have rational noncommutative parameters. These can be described in relation to any basis of $H_1(\Sigma;\ZZ)$ by quantities
$$\vartheta_{\lambda',\ell'}^{\lambda,\ell}(\chi) \in \QQ,\qquad \text{ for }\; \lambda,\lambda' =1,\ldots, r,\; \ell,\ell' =1,\ldots, 2g\,  \text{ and }\chi \in E_\kk(\Gamma)^*. $$ 
Once again, the square matrices $[\vartheta_{\lambda',\ell'}^{\lambda,\ell}(\chi)]_{\lambda,\ell; \lambda',\ell' =1}^{r,2g}$ are skew-symmetric under swapping $\ell \leftrightarrow \ell'$,
symmetric under swapping $\lambda \leftrightarrow \lambda'$, and have zero entries for $\lambda=\lambda'$. Understanding their dependence on $\chi$ is  equivalent  to the calculation of the torsion
group $E_\kk(\Gamma)$, as discussed in Section~\ref{sec:diophantine}. 
A computation of $E_\kk(\Gamma)$ in the sense of primary decomposition (see e.g.~\cite{HarHaw}, Theorem~8.14) 
leads to a description of the Pontryagin dual $E_\kk(\Gamma)^*$ (see~\cite{GRS}, Chapter~4) as a sum of finite cyclic groups. Then one may choose a symplectic basis for $H_1(\Sigma;\ZZ)$, factorise the von Neumann algebra over the set of primary cyclic summands, and refer to Proposition~\ref{vNcyclic} for a concrete description of the corresponding splitting as finite product of noncommutative tori.

Our final result describes what happens in the general case, allowing for a centre $E_\kk(\Gamma)$ of the divisor braid group with nontrivial free part.

\begin{thm}
Let $(\Gamma, \kk)$ be a very composite negative colour scheme in $r$ colours. 
Then there is a product decomposition
\begin{equation}\label{generalnctori}
\mathcal{N}({\sf DB}_\kk (\Sigma,\Gamma)) = \prod_{\chi \in ({\rm Tor}\, E_\kk(\Gamma) )^*}  \int^\oplus_{ {\rm U}(1)^{{\rm rk} E_\kk (\Gamma)}}  \mathcal{N}_{\chi,\uu} \, {\rm d}\uu
\end{equation}
over the Pontryagin dual of the centre $E_\kk(\Gamma)$, where $\mathcal{N}_{\chi,\uu}$ is the enveloping von Neumann algebra of a $2gr$-dimensional noncommutative torus generated by unitaries
$\UU_{\lambda,\ell}(\chi,\uu)$ ($\lambda=1,\ldots, r$ and $\ell=1,\ldots,2g$)
for $(\chi,\uu)\ne (0,(1,\ldots,1))$, and $\mathcal{N}_{0,(1,\ldots,1)}\cong L^\infty({\rm Pic}^0(\Sigma)^{\oplus r})$.
In  (\ref{generalnctori}), a presentation of the noncommutative torus associated to $\mathcal{N}_{\chi,\uu}$ with rational noncommutative parameters 
$\vartheta_{\lambda',\ell'}^{\lambda,\ell}(\chi,\uu)$ in all the relations
$$ \UU_{\lambda,\ell}(\chi,\uu) \UU_{\lambda',\ell'}(\chi,\uu) =  \exp\left( 2 \pi {\rm i} \,  \vartheta_{\lambda',\ell'}^{\lambda,\ell}(\chi,\uu) \right) \UU_{\lambda',\ell'}(\chi,\uu) \UU_{\lambda,\ell}(\chi,\uu)$$
can be given if and only if $\uu$ is a torsion element of the ${\rm rk}\, E_\kk(\Gamma)$-torus. 

\begin{proof}
The first product in (\ref{generalnctori}) corresponds to the factorisation sketched after Proposition~\ref{vNcyclic}; it can be obtained by considering the elements of $\CC[{\rm Tor}\, E_\kk(\Gamma)]$
\begin{equation}\label{projbetai}
\beta_\chi:= \frac{1}{|{\rm Tor}\, E_\kk(\Gamma)|} \sum_{\gamma \in {\rm Tor}\, E_\kk(\Gamma) } {\rm e}^{-{\rm i} \chi(\gamma)} \gamma \qquad \text{ for }\chi \in ({\rm Tor}\, E_\kk(\Gamma))^*
\end{equation}
in analogy to (\ref{projector}). The properties
\begin{equation}\label{projprops}
\beta_\chi^2 = \beta_\chi, \quad \beta_\chi \beta_{\chi'}=0 \text{ for }\chi \ne \chi'  \quad \text{ and } \quad \sum_{\chi \in ({\rm Tor}\, E_\kk(\Gamma))^*} \beta_\chi=e
\end{equation}
can be checked exactly as in Lemma~\ref{projectors}. The idempotents (\ref{projbetai}) provide projection operators in the reduced group $C^*$-algebra acting on the Hilbert space
$$ 
\mathcal{H}:=\ell^2({\sf DB}_\kk (\Sigma,\Gamma)),
$$ 
and taking the bicommutant of their ranges we obtain von Neumann subalgebras
  $\mathcal{N}_\chi = \beta_\chi  \mathcal{N}({\sf DB}_\kk (\Sigma,\Gamma))$, as well as a finite product decomposition $\mathcal{N}({\sf DB}_\kk (\Sigma,\Gamma)) \cong \prod_{\chi \in ({\rm Tor}\, E_\kk(\Gamma))^*} \mathcal{N}_{\chi}$, which parallels our discussion in the two-colour case.

To factorise $\mathcal{N}_\chi$ (or the corresponding factors of the group $C^*$-algebra) further, we need to employ the characters of the free part of the centre; however, they do not supply a family of mutually orthogonal projectors as in the discussion of the
torsion part (see \cite{DixVNA}, Section 2 of Chapter I in Part I).
Instead, one needs to  recast each $\mathcal{N}_{\chi}$ as a (measurable) field of von Neumann algebras (see also \cite{DixVNA}, Chapter~3 in Part~II) over the torus ${\rm U}(1)^{{\rm rk} E_\kk(\Gamma)}$ equipped with its Haar measure.

The decomposition of each $\mathcal{N}_{\chi}$ as a field of von Neumann algebras can be achieved via the spectral decomposition of the Cartan subalgebra 
$$\mathcal A_\chi := \mathcal{N}_\chi \cap \mathcal{N}'_\chi \subset \mathcal{N}_\chi.$$
This construction can be understood via the technique of virtual eigenvectors introduced originally by von Neumann, and which is presented in some detail in part (i) of Section~24 (in \S 6) of Chapter~XI in reference~\cite{God}. Applying this
machinery, one decomposes the subspace $\mathcal{H}_\chi := {\rm im}\, \pi_L(\beta_\chi)  \subset \mathcal{H}$, where the operators representing $\beta \in {\rm Tor}\, E_\kk(\Gamma)$ act via the character $\chi$, as a direct integral of Hilbert spaces (see part (ii) of the same Section in~\cite{God}) over the spectrum (i.e.~space of maximal ideals) of $\mathcal{A}_\chi$. Because ${\mathcal A}_\chi \cong \mathcal{N}(\ZZ^{\oplus {\rm rk} E_\kk(\Gamma)})$, this spectrum can be identified with the torus ${\rm U}(1)^{{\rm rk} E_\kk(\Gamma)}$ (see e.g. \S 19 and Chapter IV in reference~\cite{GRS}), so one obtains a direct integral splitting
\begin{equation}\label{Hilbsplit}
\mathcal{H}_\chi = \int^\oplus_{{\rm U}(1)^{{\rm rk} E_\kk(\Gamma)}} \mathcal{H}_\chi(\uu)\, {\rm d}{\uu}.
\end{equation}
Strictly speaking, in the general theory the factors $\mathcal{H}_\chi(\uu)$ depend on a choice of measure, but one can make this choice natural by taking ${\rm d}{\uu}$ in
(\ref{Hilbsplit}) to be the normalised (probability) Haar measure on ${\rm U}(1)^{{\rm rk} E_\kk(\Gamma)}$ (see \S 24  of~\cite{GRS}).

The orthogonal decomposition $\mathcal{H}= \bigoplus_{\chi \in ({\rm Tor} E_\kk(\Gamma))^*} \mathcal{H}_\chi $ obtained from the projectors $\beta_\chi$, together with (\ref{Hilbsplit}), combine into a splitting of the whole
of $\mathcal{H}$ for which all the operators $\pi_L(\beta)\in \mathcal{B}(\mathcal{H})$ for $\beta \in E_\kk(\Gamma)$ are diagonal, and thereby the $C^*$-algebra of the centre identifies with the algebra of continuous functions on its Pontryagin dual $({\rm Tor}\, E_\kk(\Gamma))^* \times {\rm U}(1)^{{\rm rk} E_\kk(\Gamma)}$. Let us denote by $\beta^{\tt t} \in {\rm Tor}\, E_\kk(\Gamma)$ the torsion component of an element $\beta \in E_\kk(\Gamma)$. Then by construction $\mathcal{H}_\chi$ is an invariant subspace for $\pi_L(\beta)$ and
\begin{equation}\label{diagpichi}
\pi_L(\beta)|_{\mathcal{H}_\chi} = {\rm e}^{{\rm i} \chi(\beta^{\tt t})} \int^\oplus_{{\rm U}(1)^{{\rm rk} E_\kk(\Gamma)}} \uu(\beta-\beta^{\tt t})\,  {\rm id}_{\mathcal{H}_\chi (\uu)} \, {\rm d}\uu.
\end{equation}
As anticipated in the statement of our theorem, in (\ref{diagpichi}) we are using multiplicative notation for characters $\uu$ of the free part of the centre $E_\kk(\Gamma)$ for convenience (see \S 21 in~\cite{GRS}), but still
using additive notation for the torsion part.

More generally, any operator $\pi_L(\alpha) \in \mathcal{B}(\mathcal{H})$ for $\alpha \in {\sf DB}_\kk(\Gamma)$ decomposes with respect to the splitting obtained for $\mathcal H$; i.e. each $\mathcal{H}_\chi$ is invariant for $\pi_L(\alpha)$ and
$$
\pi(\alpha)|_{\mathcal{H}_\chi} = \int^\oplus_{{\rm U}(1)^{{\rm rk} E_\kk(\Gamma)}}  U_\chi(\uu) \, {\rm d}\uu,
$$
where $U_\chi(\uu)$ is a unitary operator on $\mathcal{H}_\chi(	\uu)$ for each $\uu \in {\rm U}(1)^{{\rm rk} E_\kk(\Gamma)}$. This
follows directly from Theorem~1 in \S 2.5 of Part II of reference~\cite{DixVNA}, as these operators generate the commutant of  $C^*(E_\kk(\Gamma))$.
By von Neumann's bicommutant theorem (see e.g.~Theorem 1.2.1 in \cite{Arv}), these operators generate the group von Neumann algebra of ${\sf DB}_\kk(\Gamma)$, hence the decomposition (\ref{generalnctori}) follows.

To obtain a more precise description of each integrand $\mathcal{N}_{\chi,\uu}$, we proceed in analogy with our argument in Proposition~\ref{vNcyclic}. Recall from Theorem~\ref{present} that a set of generators of the divisor braid group can be written as
$\alpha_{\lambda,\ell} \in {\sf DB}_\kk(\Gamma)$, for $\lambda \in {\rm Sk}^0(\Gamma)$ and $\ell$ an index for a $\ZZ$-basis of $H_1(\Sigma;\ZZ)$ (which we assume symplectic). For a fixed character
$\chi$ of ${\rm Tor}\, E_\kk(\Gamma)$, consider the measurable field $\mathcal{C}_\chi$ of $C^*$-algebras over ${\rm U}(1)^{{\rm rk} E_\kk(\Gamma)}$ generated by unitaries $\UU_{\lambda,\ell}(\chi)= \int^\oplus \UU_{\lambda,\ell}(\chi, \uu) \, {\rm d}\uu$  with labels as above, and whose commutators are assumed diagonalisable:
$$
[\UU_{\lambda,\ell} (\chi), \UU_{\lambda',\ell'}(\chi)]= \int^\oplus_{{\rm U}(1)^{{\rm rk} E_\kk(\Gamma)}}  \exp\left( 2 \pi {\rm i} \,  \vartheta_{\lambda',\ell'}^{\lambda,\ell}(\chi,\uu) \right)  \mathbb{I}(\uu) \, {\rm d}\uu.
 $$
Then the assignment $\pi(\alpha_{\lambda,\ell})|_{\mathcal{H}_\chi} \mapsto \UU_{\lambda,\ell}(\chi)$ provides an isomorphism $\mathcal{N}_\chi\cong \mathcal{C}''_\chi$ if one sets
$$
 \vartheta_{\lambda',\ell'}^{\lambda,\ell}(\chi,\uu) =  \frac{ J_{\ell,\ell'}}{2\pi} \chi([\alpha_{\lambda,\ell},\alpha_{\lambda',\ell'}]^{\tt t}) \arg \uu ([\alpha_{\lambda,\ell},\alpha_{\lambda',\ell'}]-[\alpha_{\lambda,\ell},\alpha_{\lambda',\ell'}]^{\tt t}).
$$
From this equality, we deduce all the remaining assertions in the statement of the theorem.

\end{proof}
\end{thm}

\end{document}